\documentclass[11pt]{amsart}

\usepackage[top=2.7cm, bottom=2.7cm, left=2.7cm, right=2.7cm]{geometry}
\usepackage{subfig}
\usepackage{tikz-cd}
\usepackage{tikz}
\usepackage{caption}
\usepackage{amssymb}
\usepackage[all]{xy}
\usepackage{amsthm}
\usepackage{mathrsfs}
\usepackage{hyperref}
\usepackage{url}
\usepackage{xcolor}
\definecolor{green}{RGB}{0,127,0}
\definecolor{blue}{RGB}{0,0,150}
\hypersetup{
    colorlinks=true,
    linkcolor=blue,
    filecolor=magenta,      
    urlcolor=cyan,
    bookmarks=true,
    citecolor=green
}

\usepackage{graphicx}
\usepackage{graphicx}

\usepackage{setspace}

\DeclareMathOperator{\N}{\mathbf{N}}
\DeclareMathOperator{\Z}{\mathbf{Z}}

\DeclareMathOperator{\dd}{\mathbf{d}}
\DeclareMathOperator{\ee}{\mathbf{e}}

\DeclareMathOperator{\Hom}{\textup{Hom}}
\DeclareMathOperator{\GL}{\textup{GL}}

\DeclareMathOperator{\Rep}{\textup{Rep}}
\DeclareMathOperator{\Ext}{\textup{Ext}}
\DeclareMathOperator{\C}{\mathbf{C}}

\DeclareMathOperator{\PP}{\mathbf{P}}

\DeclareMathOperator{\nil}{\textup{nil}}

\DeclareMathOperator{\CC}{\mathcal{C}}
\DeclareMathOperator{\p}{\mathfrak{p}}
\DeclareMathOperator{\A}{\mathbb{A}}

\DeclareMathOperator{\E}{\mathbb{E}}

\DeclareMathOperator{\mm}{\mathbf{m}}

\DeclareMathOperator{\im}{\textup{im}}
\DeclareMathOperator{\supp}{\textup{supp}}

\DeclareMathOperator{\End}{\textup{End}}

\DeclareMathOperator{\Gr}{\textup{Gr}}

\DeclareMathOperator{\OO}{\mathcal{O}}

\DeclareMathOperator{\Tr}{\textup{Tr}}

\DeclareMathOperator{\id}{\textup{id}}

\DeclareMathOperator{\Fun}{\textup{Fun}}
\DeclareMathOperator{\PC}{\mathcal{P}}
\DeclareMathOperator{\IC}{\mathcal{I}}
\DeclareMathOperator{\RC}{\mathcal{R}}
\DeclareMathOperator{\op}{\textup{op}}

\DeclareMathOperator{\Perv}{\textup{Perv}}

\DeclareMathOperator{\FF}{\mathcal{F}}
\DeclareMathOperator{\ICC}{\mathcal{IC}}
\DeclareMathOperator{\rank}{\textup{rank}}
\DeclareMathOperator{\pr}{\textup{pr}}
\DeclareMathOperator{\LocSys}{\textup{LocSys}}

\DeclareMathOperator{\reghom}{\textup{reghom}}

\renewcommand{\A}{\mathbf{A}}

\newtheorem{theorem}{Theorem}[section]
\newtheorem{lemma}[theorem]{Lemma}
\newtheorem{proposition}[theorem]{Proposition}
\newtheorem{cor}[theorem]{Corollary}
\newtheorem{conj}[theorem]{Conjecture}

\theoremstyle{definition}

\theoremstyle{remark}
\newtheorem{remark}[theorem]{Remark}

\numberwithin{equation}{section}

\begin{document}

\title[Microlocal characterization of perverse sheaves]{Microlocal characterization of Lusztig sheaves for affine quivers and $g$-loops quivers}

\author{Lucien Hennecart}
\address{Universit\'e Paris-Saclay, CNRS,  Laboratoire de math\'ematiques d'Orsay, 91405, Orsay, France}

\email{lucien.hennecart@universite-paris-saclay.fr}

\date{\today}

\subjclass[2010]{Primary }

\begin{abstract}
We prove that for extended Dynkin quivers, simple perverse sheaves in Lusztig category are characterized by the nilpotency of their singular support. This proves a conjecture of Lusztig in the case of affine quivers. For cyclic quivers, we prove a similar result for a larger nilpotent variety and a larger class of perverse sheaves. We formulate conjectures concerning similar results for quivers with loops, for which we have to use the appropriate notion of nilpotent variety, due to Bozec, Schiffmann and Vasserot. We prove our conjecture for $g$-loops quivers ($g\geq 2$).
\end{abstract}

\maketitle

\tableofcontents
\section{Introduction}\label{introduction}

\subsection{The main results}\label{mainresult}
Motivated by his theory of Character Sheaves (\cite{MR792706}), Lusztig defined in the early $90$'s a category $\mathscr{Q}$ of equivariant semisimple constructible complexes on the representation spaces of a given acyclic quiver $Q=(I,\Omega)$ (\cite{MR1088333}). His construction is detailed in Section \ref{lusztigsheaves}. Here is a brief overview. For a dimension vector $\dd\in\N^I$, we let $E_{\dd}$ be the representation space of $Q$ of dimension $\dd$. It is acted on by the product of linear groups $G_{\dd}$. We also let $\mathscr{P}_{\dd}$ be the category of $G_{\dd}$-equivariant perverse sheaves on $E_{\dd}$ which belong to the category $\mathscr{Q}$. By definition, it is the semisimple category whose simple objects are the perverse sheaves appearing as a direct summand of the push-forward of the constant sheaf $\underline{\C}$ by the morphism $\pi:Y\rightarrow E_{\dd}$, where $Y$ is the variety of pairs $(x,\underline{F})$, $x\in E_{\dd}$, $\underline{F}$ is a $x$-stable $I$-graded flag of $\C^{\dd}$ and $\pi$ is the natural projection. In \emph{op. cit.}, Lusztig studied the singular support of the perverse sheaves which belong to $\mathscr{P}_{\dd}$. This led him to define the so-called nilpotent variety $\Lambda_{\dd}\subset T^*E_{\dd}$ (Section \ref{singularsupport}). This is a closed, conical and Lagrangian subvariety of $T^*E_{\dd}$ and for any $\mathscr{F}\in\mathscr{P}_{\dd}$, $SS(\mathscr{F})\subset \Lambda_{\dd}$. The first main result of the present paper is the proof of the converse under some restrictions on the quiver.

\begin{theorem}\label{maintheorem}
 Let $Q$ be a finite type or affine quiver\footnote{In this paper, we include under the terminology \emph{affine quiver} extended Dynkin quivers, Jordan and cyclic quivers.}. Let $\mathscr{F}\in\Perv_{G_{\dd}}(E_{\dd})$ be a $G_{\dd}$-equivariant simple perverse sheaf such that $SS(\mathscr{F})\subset \Lambda_{\dd}$. Then $\mathscr{F}$ is a Lusztig perverse sheaf.
\end{theorem}
The method of proof is briefly described in Section \ref{stepsaffine}.

Lusztig conjectured more generally in \cite[\S 10.3]{MR1182165}, without restriction on the quiver, that perverse sheaves whose singular support is contained in the nilpotent variety should provide the canonical basis of the positive part of the quantum group. When combined with Lusztig's subsequent paper \cite{MR1088333}, this may be formulated as the following conjecture which was also put forward independently by Webster (see \cite{2235}).
\begin{conj}[Lusztig]\label{conjecturegeneral}
 Let $Q$ be a loop-free quiver and $\dd\in\N^I$ a dimension vector. Then $G_{\dd}$-equivariant irreducible perverse sheaves on $E_{\dd}$ whose singular support is included in $\Lambda_{\dd}$ are exactly Lusztig sheaves.
\end{conj}

See Section \ref{discl} for relevant considerations. In this paper, we prove this result for (the easy and already known cases of) finite type and cyclic quivers and for the more subtle case of affine quivers, for which the representation theory is heavily used (Auslander-Reiten theory for affine quivers).

We also conjecture (Conjecture \ref{conjec}) a modification of the previous conjecture for arbitrary quivers (possibly carrying loops) using the appropriate notion of nilpotent variety defined in \cite{bozec2017number} (Section \ref{quiversloops}). This leads to four different categories of perverse sheaves on the representation spaces of a quiver, $\mathcal{P}^{\flat}$ for $\flat\in\{\nil,\emptyset,(\nil,1),1\}$ (Section \ref{lusztigsheavesloops}), together with four corresponding \emph{nilpotent varieties}, $\Lambda^{\flat}$, $\flat\in\{\nil,\emptyset,(\nil,1),1\}$ (Section \ref{nilpotency}). These four situations are paired using the Fourier-Sato transform ($\nil$ and $\emptyset$ are paired, $(\nil,1)$ and $1$ also). It is easily shown that the singular support of a perverse sheaf in the category $\mathcal{P}^{\flat}$ is contained in the nilpotent variety $\Lambda^{\flat}$ by using favourable functorial properties of the singular support with respect to the pushforward by a proper morphism. In Section \ref{quiversloops}, we prove the converse for $g$-loops quivers ($g\geq 2$), which constitutes the second main result of this paper:

\begin{theorem}
 Let $g\geq 2$ and $Q=S_g$ be the $g$-loop quiver. Let $\dd\in\N$ be a dimension vector and $\mathscr{F}$ an irreducible perverse sheaf on $E_{Q,\dd}$ such that $SS(\mathscr{F})\subset \Lambda^{\flat}$ where $\flat=(\nil,1)$ or $\flat=1$. Then $\mathscr{F}$ is in the category $\mathcal{P}^{\flat}$.
\end{theorem}
We give the rough idea of the proof in Section \ref{stepsloops}. We would like to emphasize the fact that no equivariance assumption is made on $\mathscr{F}$: this is not necessary in the proof and happens to be a consequence of the nilpotency of the singular support.

\subsection{Analogy with character sheaves on Lie groups and Lie algebras}\label{analogy}
In \cite{MR1088333}, motivated by his theory of character sheaves on reductive groups, Lusztig defined a class of perverse sheaves on the representation spaces of an arbitrary loop-free quiver. In both cases, the perverse sheaves under consideration are defined as appropriate shifts of direct summands of the pushforward by a proper morphism of a local system. Character sheaves for a complex reductive group $G$ are obtained in this way using the morphism
\[
 \pi : \tilde{G}=\{(g,B)\in G\times \mathscr{B}\mid g\in B\}\rightarrow G
\]
where $\mathscr{B}=G/B$ denotes the flag variety of $G$.

For quivers, Lusztig considers a family of morphisms
\[
 \pi_{(i,a)} : \tilde{\FF}_{(i,a)}\rightarrow E_{\dd}
\]
where $\tilde{\FF}_{(i,a)}$ is a smooth variety and the morphism $\pi_{(i,a)}$ is proper (Section \ref{lusztigsheaves}).
By functorial properties of the singular support, it is proved that for groups, the singular support of character sheaves is contained in
\[
 \Lambda_G=\{(g,\xi^*)\mid g\in Z_G(\xi^*)\text{ and }\xi^*\in\mathscr{N}^*\}
\]
where $\mathscr{N}^*\subset \mathfrak{g}^*$ denotes the nilcone of $\mathfrak{g}^*$ and $Z_G(\xi^*)$ is the centralizer of $\xi^*$ in $G$ for the coadjoint action of $G$ on $\mathfrak{g}^*$.
For loop-free quivers, the result is the following. Lusztig perverse sheaves on $E_{\dd}$ have a singular support which is a subvariety of
\[
 \Lambda_{\dd}=\{(x,x^*)\in E_{\bar{Q},\dd}\mid \mu_{\dd}(x,x^*)=0 \text{ and } (x,x^*) \text{ is nilpotent}\}.
\]
See Section \ref{representationvariety} for the definition of the moment map $\mu_{\dd}$ (the condition $\mu_{\dd}(x,x^*)=0$ is a generalization of the commuting relation we had for $\Lambda_{G}$). In \cite{MR948107}, Mirkovi\'c and Vilonen give a proof of a conjecture of Laumon and Lusztig asserting that character sheaves on a complex connected reductive group $G$ are exactly those with singular support included in $G\times\mathscr{N}^*$. Analogously, one can try to characterize $G_{\dd}$-equivariant perverse sheaves on $E_{\dd}$ with singular support included in $\Lambda_{\dd}$. To make the analogy even clearer, note that a $G$-equivariant perverse sheaf on $G$ has singular support included in $\mu^{-1}(0)$ where
\[
\begin{matrix}
 \mu&:&T^*G=G\times \mathfrak{g}^*&\rightarrow&\mathfrak{g}^*\\
 &&(g,\xi^*)&\mapsto&ad^*_{g}(\xi^*)
\end{matrix}
\]
is a moment map of the hamiltonian action of $G$ on itself by conjugation. Therefore, any $G$-equivariant perverse sheaf on $G$ whose singular support is contained in $G\times\mathscr{N}^*$ has singular support in $\Lambda_{G}=\mu^{-1}(0)\cap (G\times\mathscr{N}^*)$ (see also {\cite[1.4]{MR948107}}). It is also possible to prove similar results for complex reductive Lie algebras, as in \cite{MR2124171}, which is close to the quiver situation. In fact, the case of the Jordan quiver in dimension $d$ coincides by definition with the case of $\mathfrak{gl}_d$.

\subsection{Steps of the proof for affine quivers}\label{stepsaffine}
The proof proceeds in the following steps. For finite type (resp. cyclic quivers), we use that the number of orbits (resp. nilpotent orbits) is finite in each dimension and appropriate resolution of their closure. A Fourier-Sato transform allows us to conclude for type $A$ affine quivers. The cases of affine types $D$ and $E$ need a more subtle work. Using an appropriate stratification of the representation spaces (given by Auslander-Reiten theory), we reduce the problem to sheaves on the regular locus. Cyclic quivers allow us to describe a neighbourhood of a non-homogeneous tube in the representation space of $Q$. The question can now be answered by studying a class of perverse sheaves on the representation spaces of cyclic quivers. This class is slightly larger than the class of Lusztig sheaves. Therefore, we call it \emph{extended Lusztig category}. It turns out that this class of perverse sheaves contains exactly the Fourier-Sato transforms of the intersection cohomology complexes on nilpotent orbits for the opposite orientation. We describe explicitly this class of perverse sheaves together with a microlocal characterization of the simple perverse sheaves it contains, analogously to the the main theorem of this paper. Transfering the question from an affine quiver to cyclic quivers gives a proof of the theorem.

\subsection{Steps of the proof for $g$-loops quivers}
\label{stepsloops}
In Section \ref{quiversloops}, we consider general quivers, possibly carrying loops. For them, four different \emph{nilpotent varieties} are available: see \cite{bozec2017number}. Accordingly, there are four different categories of perverse sheaves on the representation spaces of the quiver. We conjecture a relationship between equivariant simple perverse sheaves whose singular support is a union of irreducible components of one of the nilpotent varieties and the corresponding category of perverse sheaves. The conjecture is proved for $g$-loops quivers with $g\geq 2$. The loops at the vertices ensure the smallness of some proper morphisms and this allows us to describe precisely the categories of perverse sheaves under consideration. We use also that the different closed subvarieties of the representation space contributing to the singular support are of codimension at least two in the support of our perverse sheaf. Then the proof rests essentially on the fact that the proper morphisms we use to define the perverse sheaves on the representation spaces are small and the existence of a bijection between isomorphism classes of simple objects in the category under consideration and irreducible components of the corresponding nilpotent variety. This bijection for quivers with loops and discrete flag-types\footnote{A flag-type is called discrete if it is an uplet of dimension vectors, each of them being supported at one vertex.} is due to Bozec. Such a bijection is not known for flag-types which are not discrete (note that for quivers with one vertex, all flag-types are discrete).

\subsection{Contents of the paper}\label{contents}
We shortly describe the contents of the different sections of this paper. Section \ref{representationtheory} contains basic results on the representation theory of finite type quivers and acyclic quivers (Auslander-Reiten theory). We put the emphasis on affine quivers (decomposition of the category of representations in three parts, preprojective, regular and preinjective). We define stratifications of the representation spaces induced by this decomposition of the category (Auslander-Reiten stratification). In the case of affine quivers, we give a refinement of this stratification. In Section \ref{lusztigsheaves}, we recall the definition of Lusztig category for a general quiver. We give an explicit description of this category for finite type quivers and cyclic quivers, together with proofs. We recall the description of Lusztig perverse sheaves for affine quivers (without proof). We define the induction and restriction functors used by Lusztig to categorify the operations of the quantum group. In Section \ref{singularsupport}, we study the singular support of Lusztig sheaves. We define Lusztig nilpotent variety and describe it explicitly for finite type and affine quivers. This uses the stratifications previously defined. We prove two technical lemmata, the first of which allows us to consider only perverse sheaves on the regular locus. In Section \ref{prooffinitetype}, we prove the microlocal characterization for finite type quivers. In Section \ref{proofcyclic}, we prove it for type $A$ affine quivers. In Section \ref{extendedps}, we define a class of perverse sheaves on the representation spaces of cyclic quivers and give its basic properties: singular support, explicit description and microlocal characterization. In Section \ref{neighb}, we explain how to describe a neighbourhood of a non-homogeneous tube in the representation space of an affine quiver using cyclic quivers. Section \ref{proofaffineq} contains the proof of Theorem \ref{mainresult} for affine quivers. Last, Section \ref{quiversloops} describes a general conjecture for quivers possibly carrying loops and cycles in the cases of general or discrete flag-types. We prove the conjecture in the case of $g$-loops quivers, $g\geq 2$. The method is completely different from that for affine quivers. In the appendices, we collect useful facts on local systems, equivariant perverse sheaves, singular supports and Fourier-Sato transform.

\subsection{Remark}
\label{discl}
After this paper was written, the author learned from Ben Webster that a proof of Conjecture \ref{conjecturegeneral} based on an unpublished work of Baranovsky and Ginzburg is written in \cite{MR3651581}, see Proposition 4.8 of \emph{op. cit.} and the comment following it. This result of Baranovsky and Ginzburg claims injectivity of the characteristic cycle map for quantized conical resolution (Theorem 4.9 of \emph{op. cit.}). It is proved in the special case of quantized quiver varieties \emph{for finite type quivers} and \emph{for affine quivers with the framing $\epsilon_0$ (one dimensional framing at the extending vertex only)} (not for wild quivers) in the paper of Bezrukavnikov and Losev, \cite{bl}. A proof of the Baranovsky-Ginzburg theorem for affine quivers and general framing can be reconstructed using a result of a subsequent paper of Losev (\cite{losev}). Indeed, Losev proves in \cite[Theorem 3.12]{losev} the Claim (II) of \cite[Section 6, page 49]{bl} for quantized quiver varieties associated to affine quivers and \emph{any} framing. By the proof of the claim (III) of \cite[Section 6, page 49]{bl} given in Section 12 at the end of page 87 of \emph{op. cit.}, (II) is the only thing we need to have the injectivity of the characteristic cycle map. Hence the injectivity of the characteristic cycle map in that case\footnote{The author thanks Ivan Losev for having helped him understanding the interactions between these papers.}. Webster's argument then gives an alternative proof of Theorem \ref{mainresult}. Nevertheless, our methods are completely different and rest on a detailed understanding of the geometry of the representation spaces of affine quivers. Note also that a priori the Baranovsky-Ginzburg's theorem does not apply for quivers with loops and that to the knowledge of the author, no proof of it for quantized quiver varieties associated to \emph{wild} loop-free quivers is written.

\subsection{Notations}\label{notations}
We denote by $\mathscr{P}$ the set of partitions. For $G$ a linear algebraic group and $X$ a complex $G$-variety, $D^b_{G}(X)$ denotes the constructible derived category of $X$ with complex coefficients. We denote by $\Perv_{G}(X)$ the category of $G$-equivariant perverse sheaves on $X$. The underlying formalism is developed in \cite{MR1299527}. If $\mathscr{F}$ is a simple perverse sheaf on a variety $X$ and $i:Y\rightarrow X$ a locally closed immersion such that $\overline{Y}=\supp\mathscr{F}$, we let $\mathscr{F}_Y:=i^*\mathscr{F}$ be the restriction of $\mathscr{F}$ to $Y$. It is still a simple perverse sheaf and $\mathscr{F}=i_{!*}\mathscr{F}_Y$, the intermediate extension of $\mathscr{F}_Y$ to $X$. A quiver is a pair $Q=(I,\Omega)$ where $I$ is the set of vertices and $\Omega$ the set of arrows\footnote{The letter $I$ will sometimes be used to denote a preinjective representation of the quiver $Q$ (provided it is acyclic), see Section \ref{representationtheory}. This should not cause any confusion.}. Both are assumed to be finite sets. For $\dd\in \N^I$ a dimension vector for $Q$, $\C^{\dd}$ is a $\N^I$-graded vector space of dimension $\dd$. For $\dd'\in\N^I$, will consider the grassmannian $\Gr(\dd',\dd)$ of $\dd'$-dimensional graded subspaces of $\C^{\dd}$. It is of course smooth, projective, and nonempty if and only if $\dd'\leq \dd$. For us, cyclic quivers are by definition quivers of type $A_{n}^{(1)}$ for some $n\geq 0$ with a possibly non-cyclic orientation. We will explicitly write when cyclic quivers are considered with a cyclic orientation. The opposite quiver of $Q$ is the quiver $Q^{\op}=(I,\Omega^{\op})$ having the same set of vertices as $Q$ but all arrows have the reverse direction. When $X$ is an algebraic (or analytic) variety and $d\in\N$ an integer, we let $\Delta\subset X^d$ be the big diagonal, that is the closed subvariety of $d$-uplets $(x_1,\hdots,x_d)$ with two or more of the $x_i$'s equal. We let $S^dX$ be the $d$-th symmetric power of $X$ and again the symbol $\Delta$ denotes its diagonal. If $X$ is a $G$-variety and $\Lambda\subset T^*X$ a subset, we let $\Perv_G(X,\Lambda)$ be the full, abelian subcategory of $\Perv_G(X)$ of perverse sheaves whose singular support is contained in $\Lambda$. The letter $k$ denotes a field. We implicitly specialize to $k=\C$ whenever we use analytic notions. Let $H\subset G$ be a closed subgroup and $X$ an $H$-variety. The group $H$ acts freely on $X\times G$ by $h\cdot (x,g)=(hx,gh^{-1})$. We let $X\times^HG$ be the quotient variety. It is a $G$-variety under the $G$-action $g'\cdot (x,g)=(x,g'g)$. If $X$ is a finite set, $\sharp X$ denotes its number of elements. The symbols $\subset$ and $\subseteq$ are used for inclusions of sets which can be an equality. When the inclusion is strict, we use $\subsetneq$.

\subsection{Terminology}
If $X$ is a $G$-variety, we say that a $G$-orbit $\OO\subset X$ is \emph{equivariantly simply connected} if its stabilizer is connected. It is the case for all orbits for quiver representations. If $(f,\phi):(X,G)\rightarrow (Y,H)$ is an equivariant morphism between a $G$-variety $(X,G)$ and an $H$-variety $(Y,H)$ (for any $(x,g)\in X\times G$, $f(g\cdot x)=\phi(g)\cdot f(x)$), we say that $f$ is an equivariant $\pi_1$-equivalence if $f^*:\LocSys_H(Y)\rightarrow \LocSys_G(X)$ is an equivalence of categories between the categories of equivariant local systems. If $\phi:H\rightarrow G$ is the inclusion of a closed subgroup and $X$ a $H$-variety, $(f,\phi):(X,H)\rightarrow (X\times^H G,G)$, $x\in X\mapsto(x,e)\in X\times^HG$ is the prototypical example of an equivariant $\pi_1$-equivalence.

\section{Some representation theory of quivers and stratification of the representation spaces}
\label{representationtheory}
We use in a fundamental manner the very explicit representation theory of finite type and affine quivers to prove the main theorem. For the sake of completeness and to fix notations, we briefly recall the basic facts we will use.

\subsection{Representation theory of finite type quivers}\label{finitetype}

The theorem ruling the representation theory of finite type quivers is the following. It does not depend on the base field.

\begin{theorem}[Gabriel, \cite{MR332887}]\label{gabriel}
Let $Q$ be a quiver. Then $Q$ has a finite number of indecomposable representations if and only if $Q$ is of type $ADE$. Moreover, taking the dimension induces a one-to-one correspondence between indecomposable representations of $Q$ and dimension vectors $\dd\in\N^I$ such that $\langle\dd,\dd\rangle=1$\footnote{$\langle\dd,\dd\rangle=\sum_{i\in I}\dd_i^2-\sum_{\alpha:i\rightarrow j\in\Omega}\dd_i\dd_j$ is the Euler form of the quiver.}. 
\end{theorem}

\subsection{Nilpotent representations of cyclic quivers}\label{cyclicquivers}
Let $C_n$ be the cyclic quiver with $n$ vertices (of type $A_{n-1}^{(1)}$) and cyclic orientation. For convenience, we label the vertices by the set $\Z/n\Z$. We assume that for $i\in \Z/n\Z$, we have an arrow $i\rightarrow i+1$. For any $i\in \Z/n\Z$ and $l\in\N_{\geq 1}$, there is exactly one indecomposable nilpotent representation of $C_n$ with top $S_i$ and length $l$. We denote it $I_{i,l}$. Let $M$ be a nilpotent representation of $C_n$. Then $M$ is called \emph{aperiodic} if for any $l\geq 1$, not all the representations $I_{i,l}$ for $i\in\Z/n\Z$ are direct summands of $M$. Nilpotent representations of $C_n$ are parametrized by multipartitions, that is functions
\[
 \mm:\Z/n\Z\rightarrow \mathscr{P}.
\]
The nilpotent representation corresponding to $\mm$ is
\[
 N_{\mm}=\bigoplus_{\substack{i\in I\\l\geq 1}}I_{i,{\mm(i)}^l}
\]
where $\mm(i)=(\mm(i)^1,\mm(i)^2,\hdots)$.
Accordingly, the (total) dimension of the multipartition $\mm$ is
\[
 \dim\mm=\dim N_{\mm}=\sum_i\lvert\mm(i)\rvert.
\]

\subsection{Some Auslander-Reiten theory}\label{auslanderreiten}
Let $Q$ be an acyclic quiver. We recall here the needed facts about the representation theory of $Q$ with a particular emphasis on affine quivers. Such theory has been known for some time now. See \cite{MR774589} for a useful account. 

\begin{theorem}
 Let $k$ be a field. Then, there exists an adjunction
\[
\tau^- : \Rep_{Q}(k) \rightleftarrows \Rep_Q(k) : \tau
\]
with bi-natural isomorphisms\footnote{We say that $(\tau^-,\tau)$ is a Serre adjunction.} (the star means the dual with respect to the $k$-vector space structure):
\[
\Ext^1(M,N)^*\simeq \Hom(N,\tau M),\quad\quad \Ext^1(M,N)^*\simeq \Hom(\tau^-N,M).
\]
\end{theorem}

The functors $\tau$ and $\tau^-$ are known as \emph{Auslander-Reiten translates}. From the above properties of $\tau^-$ and $\tau$, it is immediate that a representation $M$ of $Q$ over $k$ is projective if and only if $\tau(M)=0$ and injective if and only if $\tau^-(M)=0$. We call an indecomposable representation $M$ of $Q$ over $k$
\begin{enumerate}
\item preprojective if $\tau^nM=0$ for $n\gg 0$,
\item preinjective if $\tau^{-n}M=0$ for $n\gg0$,
\item regular if $\tau^nM\neq 0$ for all $n\in\Z$.
\end{enumerate}
Furthermore, we call a representation $M$ of $Q$ over $k$ preprojective if all its indecomposable direct summands are preprojective, and we adopt similar terminology for preinjective and regular representations. By abuse, the zero representation is preprojective, regular and preinjective. The full subcategory of $\Rep_Q(k)$ of preprojective (resp. preinjective, resp. regular) representations is denoted by $\Rep_Q^{\PC}(k)$ (resp. $\Rep_Q^{\IC}(k)$, resp. $\Rep_Q^{\RC}(k)$). These are extension closed subcategories of $\Rep_Q(k)$, hence exact categories. Moreover, for $Q$ an affine quiver, $\Rep_Q^{\RC}(k)$ is an abelian category (though not stable under taking subobjects in the bigger category $\Rep_Q(k)$). The three categories $\Rep_Q^{\RC}(k), \Rep_Q^{\PC}(k)$ and $\Rep_Q^{\IC}(k)$ are disjoint and a crucial fact for affine quivers is the following: The category to which an indecomposable $M$ belongs is given by the sign of its defect defined by 
\begin{equation}\label{defect}
\partial M=\langle\delta,\dim M\rangle,
\end{equation}
where
\[
\begin{matrix}
 \langle-,-\rangle&:&\Z^I\times\Z^I&\rightarrow& \Z\\
 &&(\dd,\ee)&\mapsto&\sum_{i\in I}\dd_i\ee_i-\sum_{\alpha : i\rightarrow j\in\Omega}\dd_i\ee_j
\end{matrix}
\]
is the (non-symmetrized) Euler form of $Q$ and $\delta$ is the indivisible positive imaginary root of $Q$ (see \cite[Fig 8.3]{MR3308668}).
 An indecomposable representation $M$ is preprojective if and only if $\partial M<0$, preinjective if and only of $\partial M>0$ and regular if and only if $\partial M=0$: it only depends on the dimension $\dim M$ of $M$.
The following proposition gives the interactions between these three subcategories.
\begin{proposition}\label{extensions}
For $M\in \Rep_Q^{\PC}(k)$, $N\in\Rep_Q^{\IC}(k)$, $L\in\Rep_Q^{\RC}(k)$, we have
\[
\Hom(N,M)=\Hom(N,L)=\Hom(L,M)=0,
\]
\[
\Ext^1(M,N)=\Ext^1(L,N)=\Ext^1(M,L)=0.
\]
\end{proposition}

\begin{cor}\label{uniquefiltration}
 Let $M=P\oplus R\oplus I$ be a representation of an affine quiver $Q$ with $P$ preprojective, $R$ regular and $I$ preinjective. Let $\dd_P=\dim P, \dd_R=\dim R$ and $\dd_I=\dim I$. Then $M$ has a unique subrepresentation of dimension $\dd_I$ and a unique subrepresentation of dimension $\dd_I+\dd_R$. In particular, it has a unique filtration $(0\subset M_1\subset M_2\subset M_2=M)$ with subquotients of dimensions $\dd_I, \dd_R, \dd_P$.
\end{cor}

\begin{proof}
 Let $N=P'\oplus R'\oplus I'$ be a subrepresentation of $M$ with $\dim N=\dd_I$ ($P'$ preprojective, $R'$ regular and $I'$ preinjective). Its defect is $\langle\delta,\dim N\rangle=\partial I=\partial P'+\partial I'\leq\partial I'$. By Proposition \ref{extensions}, $I'\subset I$, and $\partial (I/I')=\partial I-\partial I'\leq0$. Therefore, if $I/I'$ is nonzero, it has preprojective or regular direct summands. By Proposition \ref{extensions} again, $I$ has regular or preprojective direct summands: contradiction. Therefore, $I'=I$ and by considering the dimensions, $N=I$. The same argument works for subrepresentations of dimension $\dd_R+\dd_I$.
\end{proof}

The simple objects of the abelian category $\Rep_Q^{\RC}(k)$ are called simple regular. A simple regular representation $M$ is called homogeneous if $\tau M\simeq M$. It is called non-homogeneous otherwise.

\begin{theorem}[Ringel, \cite{MR774589}]\label{ringelth}
Let $Q$ be an affine acyclic quiver and $k$ an arbitrary field. Let $d$ and $p_1,\hdots,p_d$ be attached to $Q$ as in the table below. Then
\begin{enumerate}
\item There is a degree preserving bijection $M_a\leftrightarrow a$ between the set of homogeneous regular simple modules and $\lvert\PP^1_k\rvert\setminus D$ where $D$ consists of $d$ closed points of degree one\footnote{For $X$ a scheme, $\lvert X\rvert$ denotes the set of closed points of $X$.},
\item There are $d$ $\tau$-orbits $\OO_1,\hdots,\OO_d$ of non-homogeneous regular simple modules of size $p_1,\hdots,p_d$\footnote{\emph{i.e.} the set of isomorphism classes of simple objects in $\OO_j$, $1\leq j\leq d$ is of cardinality $p_j$ and the Auslander-Reiten translates $\tau$ and $\tau^-$ act as inverse cycles on it.},
\item The category $\Rep_Q^{\RC}(k)$ decomposes as a product sum of blocks\footnote{There are no morphisms or extensions between the objects of different categories in the product.}:
\[
\Rep_Q^{\RC}(k)=\prod_{a\in\mid\PP^1_k\mid\setminus D}\CC_{M_a}\times\prod_{l=1}^d\CC_{\OO_l}
\]
where $\CC_{M_a}$ is the full subcategory of objects which are extensions of $M_a$ and $\CC_{\OO}$ is the full subcategory of $\Rep_Q^{\RC}(k)$ of objects whose regular simple factors lie in $\OO$.
\end{enumerate}

\begin{figure}[h!]
\begin{tabular}{|c|c|c|}
\hline
\text{type of } $Q$ &$d$&$p_1,\hdots,p_d$\\
\hline
$A_1^{(1)}$&$0$& \\
\hline
$A_n^{(1)}, n>1$&$2$& $p_1=$\text{number of arrows going clockwise}\\
                        &  & $p_2=$\text{number of arrows going counterclockwise}\\
                        \hline
$D_n^{(1)}$&$3$&$2, 2, n-2$\\
\hline
$E_{n}^{(1)}, n=6,7,8$&$3$&$2,3,n-3$\\
\hline
\end{tabular}
\caption{Non-homogeneous tubes of affine quivers and their period \cite[(3.18)]{MR3202707}}
\label{tubes}
\end{figure}
\end{theorem}
In Theorem \ref{ringelth}, the subcategories $\CC_{\OO_l}$ are called the non-homogeneous tubes while the subcategories $\CC_{M_a}$ are the homogeneous tubes. These are finite-length categories. For any object in one of them, its quasi-length is its length in the given tube. A representation all of those indecomposable direct summands are contained in (non-)homogeneous tubes is called \emph{(non-)homogeneous}. A representation all of those indecomposable direct summands are in homogeneous tubes is called regular homogeneous. The number of non-homogeneous tubes is $d$ (see however Remark \ref{rem}) and the integers $p_1,\hdots,p_d$ are the periods. They do not depend on the chosen field. For $a\in \lvert\PP^1_k\rvert\setminus D$ and $n\geq 1$, we let $S_a[n]$ be an indecomposable representation of $Q$ in the tube $a$ of quasi-length $n$. For $\lambda$ a partition, we let $S_a[\lambda]=\bigoplus_{i\geq 1}S_a[\lambda_i]$. We obtain all representations in $C_{M_a}$ in this way.
\begin{remark}\label{rem}
In type $A_n^{(1)}$, $n\geq 2$ in the case where all arrows except one go in the same direction, we have in fact $d=1$, \emph{i.e.} there is only one non-homogeneous tube. The Kronecker quiver $K2$ (type $A_{1}^{(1)}$) has only homogeneous tubes.
\end{remark}
Define $\PP_1^{\hom}=\lvert\PP_k^1\rvert\setminus D$ the set indexing homogeneous tubes.

\subsection{Quiver representation varieties}\label{representationvariety}
Let $Q=(I,\Omega)$ be a finite quiver with set of vertices $I$ and set of arrows $\Omega$. For a field $k$ and an $I$-graded vector space $V$ over $k$, the associated representation variety of $Q$ is
\[
 E_{Q}(V)=\bigoplus_{i\stackrel{\alpha}{\rightarrow}j\in\Omega}\Hom(V_i,V_j).
\]
If $\dd\in\N^I$ is a dimension vector, then the representation variety of $Q$ in dimension $\dd$ is
\[
 E_{Q,\dd}:=E_Q(V)
\]
for $V=(k^{\dd_i})_{i\in I}$.
When the context is clear, we write $E(V)=E_{Q}(V)$ and $E_{\dd}=E_{Q,\dd}$. Elements of $E_{\dd}$ can be seen as representations of $Q$. If $x\in E_{\dd}$, we let $(k^{\dd},x)$ be the associated representation. To shorten the notation, we sometimes write only $x$. Also, we will sometimes consider quivers $Q'=(I,\Omega')$ with the same underlying graph as $Q$ but a possibly different orientation $\Omega'$. In this case, we write $E_{\dd}^{\Omega'}=E_{Q',\dd}$. 

The variety $E_{Q}(V)$ is acted on by the product of linear groups
\[
 G(V)=\prod_{i\in I}\GL(V_i).
\]

If $V=(k^{\dd_i})_{i\in I}$, we write $G_{\dd}=G(V)=\prod_{i\in I}\GL_{\dd_i}(k)$.
The orbits of $E_{\dd}$ under $G_{\dd}$ are in one-to-one correspondence with isomorphism classes of representations of $Q$ of dimension $\dd$ over $k$. Note that the diagonal embedded copy of $\C^*$ inside $G_{\dd}$ acts trivially on $E_{\dd}$. If $M$ is a representation of $Q$ of dimension $\dd$, we let $\OO_M\subset E_{\dd}$ be its orbit.

We denote by $\overline{Q}=(I,\Omega\sqcup\overline{\Omega})$ the doubled quiver having the same set of vertices as $Q$ and for each $\alpha\in\Omega$, an additional arrow $\bar{\alpha}\in \overline{\Omega}$ going in the opposite direction. The symplectic action of $G_{\dd}$ on $T^*E_{\dd}$ identified with $E_{\overline{Q},\dd}$ via the trace pairing has (quadratic) moment map
\[
\begin{matrix}
 \mu_{\dd} &:& E_{\overline{Q},\dd}&\rightarrow& \prod_{i\in I}\mathfrak{gl}_{\dd_i}\\
 &&x&\mapsto&\sum_{\alpha\in\Omega}[x_{\alpha},x_{\overline{\alpha}}]
\end{matrix}.
 \]
Its zero-level $\mu_{\dd}^{-1}(0)$ is of fundamental importance in the geometry of quiver representations, as shown in \cite{MR1834739}. This can be explained by the fact that $[\mu_{\dd}^{-1}(0)/G_{\dd}]$ has a strong relation to the cotangent stack to $[E_{\dd}/G_{\dd}]$.

\subsection{Stratification of the representation spaces of acyclic quivers}\label{stratification}
We stratify the representation spaces of acyclic quivers using Auslander-Reiten theory. For finite type quivers, the given stratification is trivial, but can be refined by the stratification by orbits. For affine quivers, we give an other refinement which is given by Ringel in \cite{MR1676227}. In both cases, these stratifications have the nice property that Lusztig nilpotent variety is the union of the conormal bundles to some of the strata, as we will see in Section \ref{irreducible components}.

\subsubsection{Auslander-Reiten stratification of the representation spaces}
Let $Q=(I,\Omega)$ be an acyclic quiver and $\dd$ a dimension vector. For $x\in E_{Q,\dd}$, we let $\dd_{P}(x), \dd_{R}(x)$ and $\dd_{I}(x)$ be the dimension vector of the preprojective, regular and preinjective direct summands of $x$ seen as a representation of $Q$. For $\dd_P, \dd_R, \dd_I\in\N^I$ such that $\dd_P+\dd_R+\dd_I=\dd$, we define
\[
 E_{\dd_P,\dd_R,\dd_I}=\{x\in E_{Q,\dd}\mid \dd_P(x)=\dd_P, \dd_R(x)=\dd_R, \dd_I(x)=\dd_I\}.
\]
It is a smooth locally closed subvariety of $E_{Q,\dd}$ and we have
\[
 E_{Q,\dd}=\bigsqcup_{\dd_P+\dd_R+\dd_I=\dd}E_{\dd_P,\dd_R,\dd_I}.
\]
We can refine this stratification by fixing the isoclass of the preprojective and preinjective summands in the following way. Choose a decomposition $\dd=\dd_P+\dd_R+\dd_I$, a preprojective representation $P$ of dimension $\dd_P$ and a preinjective representation $I$ of dimension $\dd_I$. Define
\[
 E_{[P],\dd_R,[I]}=\{x\in E_{\dd}\mid (\C^{\dd},x)\simeq P\oplus R\oplus I, \text{ $R$ regular of dimension $\dd_R$} \}.
\]
\begin{lemma}
 We have the decomposition into smooth locally closed subvarieties
 \[
  E_{Q,\dd}=\bigsqcup_{[P],\dd_R,[I]}E_{[P],\dd_R,[I]}
 \]
where the sum runs over the triples $([P],\dd_R,[I])$ where $[P]$ is a preprojective isoclass, $[I]$ a preinjective isoclass and $\dd_R\in\N^I$ a dimension vector subject to the condition $\dim P+\dd_R+\dim I=\dd$.
\end{lemma}

\subsubsection{Ringel stratification of the representation spaces of acyclic affine quivers}\label{ringelstrat}
To describe the irreducible components of Lusztig nilpotent variety for affine quivers, Ringel introduced in \cite{MR1676227} a stratification of the representation spaces. Let $P$ an isomorphism class of a preprojective representation, $I$ of a preinjective and $N$ of a non-homogeneous regular representation. Let
\[
 \mu : \mathscr{P}\rightarrow \N
\]
be a function with finite support. Then $(P,I,N,\mu)$ is called a $\emph{type}$. Its dimension is $\dim(P,I,N,\mu)=\dim P+\dim I+\dim N+\sum_{\lambda\in\mathscr{P}}\mu(\lambda)|\lambda|\delta$. We also let $\dim\mu=\sum_{\lambda\in\mathscr{P}}\mu(\lambda)|\lambda|$. If $\dd=\dim(P,I,N,\mu)$, the type $(P,I,N,\mu)$ defines a stratum $\Xi(P,I,N,\mu)\subset E_{\dd}$. Let first $\lambda_1,\hdots,\lambda_r$ be partitions such that any partition $\lambda$ appears $\mu(\lambda)$ times in this list. The stratum consists of $x\in E_{\dd}$ which are isomorphic to a direct sum
\[
 P\oplus I\oplus N\oplus \bigoplus_{i=1}^rS_{a_{i}}[\lambda_i]
\]
where $a_1,\hdots,a_r\in \lvert\PP^1_{\C}\rvert\setminus D$ are pairwise distinct. It is clear (by Theorem \ref{ringelth}) that $E_{\dd}$ is the disjoint union of the strata $\Xi(P,I,N,\mu)$ where $(P,I,N,\mu)$ is a type of dimension $\dd$. To ease the notation, we will write $\Xi(P,I,N,\mu)=\Xi(P\oplus I\oplus N,\mu)$. Also when $\mu=0$, $\Xi(P,I,N,0)$ is the orbit of $P\oplus I\oplus N$ and will usually be denoted by $\OO_{P\oplus I\oplus N}$. We also define $E_{\dd}^{reg}=\{x\in E_{\dd}\mid (\C^{\dd},x)\text{ is regular}\}$ and $E_{\dd}^{\reghom}=\{x\in E_{\dd}\mid (\C^{\dd},x)\text{ is regular homogeneous}\}$. By Remark \ref{rem}, both coincide for the Kronecker quiver. Recall from Theorem \ref{ringelth} that $D\subset \lvert\PP^1\rvert$ is the set of non-homogeneous tubes. It consists of $0,1,2$ or $3$ points. It $T\subset D$ is a subset, we let $E_{\dd}^{T}$ be the set of $x\in E_{\dd}^{reg}$ such that all indecomposable inhomogeneous direct summands of $x$ are contained in the tubes indexed by $T$. In particular, $E_{\dd}^{\emptyset}=E_{\dd}^{\reghom}$. It is an open subset of $E_{\dd}^{reg}$ and thus of $E_{\dd}$ since it is the set of $x\in E_{\dd}^{reg}$ such that $\Hom(N,x)=0$ for any inhomogeneous indecomposable representation of $Q$ in the tubes indexed by $D\setminus T$ of dimension $\leq\dd$, and the isoclasses of such $N$ are in finite number. If $T\subset D$ and $N$ is a regular inhomogeneous representation of $Q$ whose indecomposable summands are in the tubes indexed by $T$, for any $\dd\in\N^I$, we let $E_{[N],\dd}^{D\setminus T}$ be the set of $x\in E_{\dd+\dim N}^{reg}$ such that $x$ is isomorphic to $N\oplus R$ for a regular representation $R$ all of those regular non-homogeneous direct summands belong to the tubes indexed by $D\setminus T$.

The map $\mu : \mathscr{P}\rightarrow\mathscr{P}$ is called \emph{regular} if it takes nonzero values only on partitions of length one, and regular semisimple if it is nonzero only on the partition $(1)$. A type $\Xi(P,I,N,\mu)$ is called \emph{regular} (resp. \emph{regular semisimple}) if $\mu$ is regular (resp. regular semisimple).
\begin{lemma}\label{irrclosm}
 The stratum $\Xi(N,\mu)\subset E_{\dd}^{reg}$ where $N$ is regular non-homogeneous and $\mu$ is regular semisimple are irreducible, locally closed and smooth.
\end{lemma}
We will need such a result to use Lemma \ref{sysloczerosection} in the Section \ref{singularsupports}. It is more generally also true that the stratum $\Xi(P,I,N,\mu)$ is smooth for any quadruple $(P,I,N,\mu)$. The above result is sufficient for our purposes.

Let $N$ be a non-homogeneous regular representation and $\mu:\mathscr{P}\rightarrow \N$. Let $d=\dim \mu$.

\begin{lemma}\label{isoreg}
 The natural map induced by the direct sum
 \[
  (\OO_N\times\Xi(\mu))\times^{G_{\dd_N}\times G_{d\delta}}G_{\dd}\rightarrow \Xi(N,\mu)
 \]
is an isomorphism.
\end{lemma}
\begin{proof}
 It follows directly from the fact that any regular representation of $Q$ can be uniquely decomposed as a direct sum of a regular non-homogeneous representation and a regular homogeneous one.
\end{proof}

\subsubsection{Quotient of the regular homogeneous locus}\label{quotientreg}
Consider $\theta:=\langle\delta,-\rangle:\Z^I\rightarrow \Z$. By \cite{MR1906875}, the open subset of $E_{\delta}$ of $\theta$-stable representations coincide with $E_{\delta}^{\reghom}$ and $E_{\delta}^{\reghom}/(G_{\delta}/\C^*)\simeq \PP_1^{\hom}$. As a consequence, we have a map
\[
 E_{\delta}^{\reghom}\rightarrow \PP_1^{\hom}
\]
whose fibers are orbits of simple regular representations of $Q$. More generally, if $d\in\N$, $E_{d\delta}^{\reghom}$ is an open subset of the $\theta$-semistable locus and the quotient map is now
\[
 E_{d\delta}^{\reghom}\rightarrow S^d\PP_1^{\hom}.
\]
If $\mu$ is regular semisimple with $\dim\mu=d$, $\Xi(\mu)\subset E_{d\delta}^{\reghom}$. Moreover, the morphism above gives a morphism
\[
 \chi_{\mu}:\Xi(\mu)\rightarrow S^d\PP_1^{\hom}\setminus\Delta.
\]
If $N$ is a non-homogeneous regular representation, then we obtain a morphism
\[
 (\OO_N\times \Xi(\mu))\times GL_{\dd}\rightarrow S^d\PP_1^{\hom}\setminus\Delta
\]
which factorizes through the action of $\GL_{\dd_N}\times \GL_{\dim\mu}$ and by Lemma \ref{isoreg} gives rise to a $G_{\dd}$-equivariant morphism:
\[
 \chi_{N,\mu}:\Xi(N,\mu)\rightarrow S^d\PP_1^{\hom}\setminus\Delta.
\]

\subsection{Stratification of the representation spaces of the Jordan and cyclic quivers}
For the Jordan and cyclic quivers, we obtain a stratification similar to that in Section \ref{ringelstrat}.
\subsubsection{Stratification of the representation spaces of the Jordan quiver}\label{stratjordan} Let $J$ be the Jordan quiver (one vertex and one loop). In dimension $d\in\N$, the representation space $E_{J,d}$ is the Lie algebra $\mathfrak{gl}_d$ endowed with the adjoint action of $\GL_d$. We let $G=\GL_d$ and $\mathfrak{g}=\mathfrak{gl}_d$. We describe Lusztig stratification of $\mathfrak{gl}_d$ (the same kind of stratification exists for any reductive Lie algebra and also on any reductive group, \cite{MR732546}, \cite[5.5]{MR2124171}). For $x\in \mathfrak{g}$, we let $Z_{\mathfrak{g}}(x)$ be the centralizer of $x$ in $\mathfrak{g}$. For a Levi subalgebra $\mathfrak{l}\subset \mathfrak{g}$, let $Z_r(x)=\{x\in\mathfrak{l}\mid Z_{\mathfrak{g}}(x)=\mathfrak{l}\}$. Let $\OO\subset \mathfrak{l}$ be a nilpotent orbit. Then, we obtain a stratum
\[
 \Xi(\mathfrak{l},\OO)=G\cdot(Z_r(\mathfrak{l})+\OO).
\]
\begin{proposition}
 The partition
 \[
  \mathfrak{g}=\bigsqcup_{(\mathfrak{l},\OO)}\Xi(\mathfrak{l},\OO)
 \]
 where the sum is indexed by pairs $(\mathfrak{l},\OO)$ up to conjugation is a stratification by smooth locally closed subvarieties.
\end{proposition}

We will give a very explicit description of the strata. We let $\mathcal{S}$ be the set of pairs $(\mathfrak{l},\OO)$ as above up to conjugation. We let $\Fun_d(\mathscr{P},\N)$ be the set of functions $\mu : \mathscr{P}\rightarrow \N$ of weight $d$, \emph{i.e.} such that $\sum_{\lambda\in\mathscr{P}}\mu(\lambda)\lvert\lambda\rvert=d$. We will construct a bijection $\Fun_d(\mathscr{P},\N)\rightarrow \mathcal{S}$. Let $\mu : \mathscr{P}\rightarrow \N$ be a function of weight $d$. Let $\lambda_1,\hdots, \lambda_r$ be the collection of partitions ordered by decreasing lengths such that any partition $\lambda\in\mathscr{P}$ appears exactly $\mu(\lambda)$ times. It defines the diagonal Levi subalgebra
\[
 \mathfrak{l}_{\mu}=\prod_{i=1}^r\mathfrak{gl}_{\lvert\lambda_i\rvert}
\]
and the nilpotent orbit
\[
 \OO_{\mu}=\prod_{i=1}^r\OO_{\lambda_i}
\]
of $\mathfrak{l}$, where for any partition $\lambda$ of an integer $n$, $\OO_{\lambda}$ is the corresponding nilpotent orbit of $\mathfrak{gl}_n$ (the partition $(n)$ corresponds to the regular nilpotent orbit).

The map
\[
 \begin{matrix}
  \Fun_d(\mathscr{P},\C)&\rightarrow&\mathcal{S}\\
  \mu&\mapsto&(\mathfrak{l}_{\mu},\OO_{\mu})
 \end{matrix}
\]
is a bijection. We write $\Xi(\mu)=\Xi(\mathfrak{l}_{\mu},\OO_{\mu})$. In explicit terms, $\Xi(\mu)$ is the smallest $G$-invariant subset of $\mathfrak{g}$ containing the matrices
\[
 J_{\mu}(x_1,\hdots,x_r)=\begin{pmatrix}
  J_{\lambda_1}(x_1)&&&\\
  &J_{\lambda_2}(x_2)&&\\
  &&\ddots&\\
  &&&J_{\lambda_r}(x_r)
 \end{pmatrix}.
\]
where $x_1,\hdots,x_r$ are pairwise distinct complex numbers, and for $x\in \C$ and a partition $\lambda$, $J_{\lambda}(x)$ is the standard Jordan matrix with eigenvalue $x$. That is,
$J_{\lambda}(x)=x I_{\lvert\lambda\rvert}+J_{\lambda}(0)$,
\[
 J_{\lambda}(0)=\begin{pmatrix}
  J_{\lambda^{1}}&&&\\
  &J_{\lambda^2}&&\\
  &&\ddots&\\
  &&&J_{\lambda^s} \end{pmatrix}
\]
if $\lambda=(\lambda^1,\hdots,\lambda^s)$ and for $\ell\in\N$,
\[
 J_{\ell}=(\delta_{i+1,j})_{1\leq i,j\leq \ell}
 \]
where $\delta_{k,l}$ is the Kronecker symbol.

\subsubsection{Stratification of the representation spaces of cyclic quivers}\label{stratcyclic}
Let $C_n$ be the cyclic quiver with $n$ vertices indexed by $\N/n\Z$ and having arrows $i\rightarrow i+1$ for $i\in\Z/n\Z$. Let $d\in \N$ and $\delta=(1,\hdots,1)\in\N^{\Z/n\Z}$. A representation of $C_n$ is a $n$-tuple $(x_i)_{i\in\Z/n\Z}$ of linear maps $x_i : V_i\rightarrow V_{i+1}$ for $i\in\Z/n\Z$. We have a closed immersion
\[
\begin{matrix}
 i_d &:& E_{J,d}&\rightarrow&E_{Q,d\delta}\\
 &&x&\mapsto&(\id,\hdots,\id,x).
\end{matrix}
\]
Let $N$ be a nilpotent representation of $C_n$ and $\dd_N$ its dimension. Let $\mu : \mathscr{P}\rightarrow \N$ be a finitely supported function.  We let $\dim\mu=\sum_{\lambda\in\mathscr{P}}\mu(\lambda)\lvert\lambda\rvert$ and $\dd=\dim N+\delta\dim\mu$. We define the subset $\Xi(N,\mu)$ of elements $x$ of $E_{Q,\dd}$ such that $(\C^{\dd},x)$ is isomorphic to $N\oplus R$ as a representation of $C_n$, where $R$ is a representation of $C_n$ such that the orbit $\OO_R\subset E_{Q,\dim\mu}$ intersects $i_d(\Xi(\mu)\cap \GL_{\dim\mu})$, where $\Xi(\mu)$ is the corresponding stratum in $E_{J,\dim\mu}$ defined in Section \ref{stratjordan} and $\GL_{\dim\mu}\subset E_{J,\dim\mu}$ is the set of invertible elements.

Let $\Xi(N,\mu)\subset E_{\dd}$, $\dd=\dim N+\delta\dim\mu$, be a stratum. Since any representation of $C_n$ can be uniquely decomposed as a direct sum of a nilpotent representation and an invertible one, we obtain the following lemma.
\begin{lemma}\label{isocyclic}
 The natural map induced by the direct sum
 \[
  (\OO_N\times \Xi(\mu))\times^{G_{\dd_N}\times G_{\delta\dim\mu}}G_{\dd}\rightarrow \Xi(N,\mu)
 \]
is an isomorphism.
\end{lemma}

\begin{proposition}\label{partcyclic}
 The partition
 \[
  E_{C_n,\dd}=\bigsqcup_{(N,\mu)}\Xi(N,\mu)
 \]
where the sum runs over pairs $(N,\mu)$, $N$ is a nilpotent representation (taken up to isomorphism) and $\mu$ are such that $\dim N+\dim\mu\delta=\dd$ is a locally closed stratification of $E_{C_n,\dd}$. Moreover, if $\C^*$ is the action by multiplication on $E_{C_n,\dd}$, this stratification is $\C^*$-stable. 
\end{proposition}

\subsubsection{Isomorphism classes of representations of the cyclic quiver} For $\mu:\mathscr{P}\rightarrow \N$ such that $\dim\mu=d\delta$ and $(x_1,\hdots,x_r)\in(\C^*)^r$, $r=\sum_{\lambda\in\mathscr{P}}\mu(\lambda)$, we let $J_{\mu}^{C_n}(x_1,\hdots,x_r)=i_d(J_{\mu}(x_1,\hdots,x_r))$. For any partition $\lambda$ of $d$ and $x\in\C^*$, we let $J_{\lambda}^{C_n}(x)=i_d(J_{\lambda}(x))$. When the context is clear, we drop the exponent. In particular, we write
\[
 J_{\mu}(\underline{x})=J_{\mu}^{C_n}(x_1,\hdots,x_r).
\]
Then, any representation of $C_n$ is isomorphic to a representation of the form
\[
 N_{\mathbf{m}}\oplus J_{\mu}(\underline{x})
\]
for a unique pair $(\mathbf{m},\mu)$, where $\mathbf{m}$ is a multipartition, $\mu:\mathscr{P}\rightarrow \N$ and $\underline{x}\in(\C^*)^{\sum_{\lambda}\mu(\lambda)}$. When $\mu$ is regular semisimple (it means by definition that for any $\lambda\in\mathscr{P}$, $\mu(\lambda)\neq 0\implies \lambda=(1)$),
\[
 J_{\mu}(x)\simeq \bigoplus_{j=1}^r J_1(x_j).
\]

\subsection{Open subsets of the representation spaces of cyclic quivers}\label{opensubsets}
Let $\dd\in\N^{\Z/n\Z}$. We have a morphism of algebraic varieties
\[
\begin{matrix}
\varphi &:& E_{C_n,\dd}&\rightarrow&E_{J,\dd_0}\\
&&(x_0,\hdots,x_{n-1})&\mapsto&x_{n-1}x_{n-2}\hdots x_0
\end{matrix}.
\]
Let
\[
 \chi_J : E_{J,\dd_0}\rightarrow S^{\dd_0}\C=E_{J,\dd_0}/\!\!/\GL_{\dd_0}
\]
be the quotient map and $\chi=\chi_J\circ \varphi$. Let $S^{\dd_0}D(0,1)\subset S^{\dd_0}\C$ where $D(0,1)\subset \C$ is the open unit disk. We let $E_{C_n,\dd}^{<1}=\chi^{-1}(S^{\dd_0}D(0,1))$. It is an open analytic subset of $E_{C_n,\dd}$. The stratification of $E_{C_n,\dd}$ induces a stratification of $E_{C_n,\dd}^{<1}$ (it just constrains the eigenvalues to have absolute value $<1$). For a stratum $\Xi(N,\mu)\subset E_{\dd}$, we let $\Xi^{<1}(N,\mu)=\Xi(N,\mu)\cap E_{\dd}^{<1}$ be the corresponding stratum. To any subset of strata $\mathcal{S}$ of $E_{C_n,\dd}$, we define $\mathcal{S}^{<1}=\{S\cap E_{C_n,\dd}^{<1}:S\in\mathcal{S}\}$ the corresponding set of strata of $E_{C_n,\dd}^{<1}$. For any stratum $\Xi(N,\mu)$, we let
\[
 j:\Xi(N,\mu)\rightarrow E_{C_n,\dd},
\]
\[j^{<1}:\Xi^{<1}(N,\mu)\rightarrow E_{C_n,\dd}^{<1}
 \]
 and
 \[
  j_{N,\mu}:\Xi^{<1}(N,\mu)\rightarrow \Xi(N,\mu)
 \]
be the natural inclusions. The group $G_{\dd}$ acts on $\Xi(N,\mu)$, $\Xi^{1}(N,\mu)$, $E_{C_n,\dd}^{<1}$ and all the above maps are $G_{\dd}$ equivariant. For a subset $D\subset \C$, we let 
\[
 E_{C_n,\dd}^{\C\setminus D}=
 \left\{
 \begin{aligned}
  &\chi^{-1}(S^{\dd_0}(\C\setminus D)) \text{ if $D$ does not contain $0$,}\\
  &\{x\in\chi^{-1}(S^{\dd_0}(\C\setminus D))\mid x\text{ has no nilpotent direct summands}\} \text{ otherwise}
 \end{aligned}
 \right.
\]

By Proposition \ref{partcyclic}, if $\C^*$ acts on $E_{C_n,\dd}$ with weight one, Lusztig strata $\Xi(N,\mu)$ are $\C^*$-invariant. This gives the following result.

\begin{proposition}
 The inclusion $j_{N,\mu}$ induces an isomorphism at the level of fundamental groups.
\end{proposition}

Recall the isomorphism
\[
 (\OO_N\times \Xi(\mu))\times^{G_{\dd_N}\times G_{\dim\mu\delta}}G_{\dd}\rightarrow \Xi(N,\mu)
\]
of Lemma \ref{isocyclic}.
Let $d=\dim\mu$. The map $\chi:E_{C_nd\delta}\rightarrow S^d(\C)$ restricts to
\[
 \chi_{\mu}:\Xi(\mu)\rightarrow S^d(\C^*).
\]
By the previous isomorphism, it is easily seen to give a map
\[
 \chi_{N,\mu}:\Xi(N,\mu)\rightarrow S^d(\C^*).
\]
If moreover $\mu$ is regular semisimple, it takes value in the complement of the diagonal:
\[
 \chi_{N,\mu}:\Xi(N,\mu)\rightarrow S^d(\C^*)\setminus\Delta.
\]

\section{Lusztig perverse sheaves, Induction and Restriction functors}\label{lusztigsheaves}
In this Section, we briefly recall how Lusztig sheaves are built and the operations of induction and restriction. We let the reader consult \cite{MR1227098}, \cite{MR1088333}, \cite{MR3202708} for more details on the link with quantum groups.

\subsection{Lusztig perverse sheaves}\label{lusztigperversesheaves}
In his foundational paper \cite{MR1088333}, Lusztig introduced a semisimple category of constructible complexes on the representation varieties of a quiver giving a categorification of one half of the quantum group and providing the so-called canonical basis. We briefly recall here how Lusztig sheaves are obtained.

Let $\dd\in\N^I$ be a dimension vector. A flag-type of dimension $\dd$ is an uplet $\underline{\dd}=(\dd_1,\hdots,\dd_l)\in(\N^I)^l$ for some $l\geq 1$ such that $\sum_{j=1}^l\dd_i=\dd$. Given a flag-type $\underline{\dd}$ as above, define the partial flag variety
\[
 \mathcal{F}_{\underline{\dd}}=\{(0=F_0\subset \hdots\subset F_l=k^{\dd})\mid \dim(F_j/F_{j-1})=\dd_j \text{ for $1\leq j\leq l$}\}.
\]
Define also
\[
 \tilde{\mathcal{F}}_{\underline{\dd}}=\{(x,\underline{F})\in E_{\dd}\times \mathcal{F}_{\underline{\dd}}\mid x(F_j)\subset F_j \text{ for $1\leq j\leq l$}\}.
\]
For quivers with cycles, it is also useful to consider the nilpotent version
\[
 \tilde{\mathcal{F}}^{\nil}_{\underline{\dd}}=\{(x,\underline{F})\in E_{\dd}\times \mathcal{F}_{\underline{\dd}}\mid x(F_j)\subset F_{j-1} \text{ for $1\leq j\leq l$}\}.
\]
We have natural projections $\pi_{\underline{\dd}} : \tilde{\mathcal{F}}_{\underline{\dd}}\rightarrow E_{\dd}$, $\pi_{\underline{\dd}}^{\nil} : \tilde{\mathcal{F}}^{\nil}_{\underline{\dd}}\rightarrow E_{\dd}$ which are projective and $\tilde{\mathcal{F}}_{\underline{\dd}}, \tilde{\mathcal{F}}^{\nil}_{\underline{\dd}}$ are smooth, being affine fibrations over the flag manifold $\mathcal{F}_{\underline{\dd}}$.

For quivers without oriented cycles, there is no need to consider both $\tilde{\mathcal{F}}_{\underline{\dd}}$ and $\tilde{\mathcal{F}}^{\nil}_{\underline{\dd}}$. However, for quivers with loops or more generally oriented cycles, this leads to different (although related) stories, as is already seen for the Jordan quiver (on the one-hand we have the Grothendieck-Springer resolution and on the other hand the Springer resolution). See for example \cite{MR1261904,MR2553376,MR3368082,MR3569998} for some perspective.

For a dimension vector $\dd\in\N^I$ and a flag-type $\underline{\dd}$, $\tilde{\mathcal{F}}_{\underline{\dd}}$ is smooth and $\pi_{\underline{\dd}}$ is proper and therefore by the decomposition theorem (\cite[Th\'eor\`eme 6.2.5]{MR751966}) $(\pi_{\underline{\dd}})_*\underline{\C}$ is a semisimple constructible complex on $E_{\dd}$. In his paper \cite{MR1088333}, for loop-free quivers, Lusztig considers the category $\mathscr{Q}_{\dd}$ of semisimple constructible complexes on $E_{\dd}$ whose direct summands are shifts of some of the direct summands of the complexes $(\pi_{\underline{\dd}})_*\underline{\C}$ for various \emph{discrete} flag-types $\underline{\dd}$ of dimension $\dd$. We call $\mathscr{Q}=\prod_{\dd\in\N^I}\mathscr{Q}_{\dd}$ the Hall category. He also considers the category of perverse sheaves $\mathscr{P}_{\dd}$ which are in $\mathscr{Q}_{\dd}$. When $i\in I$ and $n\geq 0$, we let $L_{ne_i}=(\pi_{(ne_i)})_{*}\underline{\C}$. It is a constructible complex on $E_{ne_i}$ ($e_i$ is the $i$-th vector of the canonical basis of $\Z^I$). Observe that for $n=0$, it does not depend on $i$.

\subsection{Lusztig perverse sheaves for finite type quivers}
It is possible to give a complete description of Lusztig sheaves for finite type quivers. This task is easy since for any dimension vector $\dd$, $E_{\dd}$ has a finite number of $G_{\dd}$-orbits. This is the content of Theorem \ref{lusztigsheavesft}, of which we provide a geometric proof. This theorem can be proved differently using that for finite type quivers, the representation varieties are union of finite number of orbits in any dimension and combining results of Ringel (\cite[Theorem 3.16]{MR3202707}) and Lusztig (categorification of the quantum group, \cite{MR1088333}).  
\subsubsection{Description of Lusztig perverse sheaves for finite type quivers}
\label{Lusztigft}
This Section relies on the desingularization of finite type orbits given in \cite{MR1985731}. The main theorem of \emph{loc. cit.} can be formulated as follows.

\begin{theorem}[Reineke, {\cite[Theorem 2.2]{MR1985731}}]\label{Reinekedesing}
 Let $Q=(\Omega,I)$ be a finite type quiver, $\dd\in\N^I$ a dimension vector and $\OO\subset E_{\dd}$ a $G_{\dd}$-orbit. Then there exists a flag-type $\underline{\dd}=(\dd_1,\hdots,\dd_l)$ with $\sum_{i=1}^l\dd_i=\dd$ such that the projective morphism
 \[
  \pi_{\underline{\dd}} : \FF_{\underline{\dd}}\rightarrow E_{\dd}
 \]
 factorizes through $\overline{\OO}$ and induces a desingularization of $\overline{\OO}$.

\end{theorem}

We obtain immediately the following description (obtained by Lusztig with a different approach) of Lusztig sheaves for finite type quivers.
\begin{theorem}[{\cite{MR1035415}}]\label{lusztigsheavesft}
 For a dimension vector $\dd$, $\mathscr{P}_{\dd}$ is the semisimple category generated by the simple objects $\ICC(\OO,\underline{\C})$ for all $G_{\dd}$-orbits $\OO\subset E_{\dd}$.
\end{theorem}
\begin{proof}
 It is clear that $\ICC(\OO,\underline{\C})$ appears in $(\pi_{\underline{\dd}})_*\underline{\C}$ where $\underline{\dd}$ is given by Theorem \ref{Reinekedesing}. Since $E_{\dd}$ has a finite number of $G_{\dd}$-orbits and each of them has a connected stabilizer in $G_{\dd}$, all $G_{\dd}$-equivariant perverse sheaves on $E_{\dd}$ are of this form. This concludes the proof.
\end{proof}

\subsection{Lusztig perverse sheaves for cyclic quivers with cyclic orientation}
For cyclic quivers with cyclic orientation, the situation is very close to that of finite type quivers since Lusztig sheaves are supported on the nilpotent locus which has only a finite number of $G_{\dd}$-orbits.
\subsubsection{Partial resolutions of aperiodic orbits}\label{resolutionsaperiodic}
We give a resolution of aperiodic nilpotent orbits of cyclic quivers in the spirit of \cite[Proposition 1.1]{MR2057407}. The idea is to construct a flag-type associated to any nilpotent orbit giving a resolution of its closure and then to refine it in order to consider a discrete flag, which is only possible for aperiodic orbits.

Let $n\geq 2$ and $C_n$ be the cyclic quiver with $n$ vertices and cyclic orientation. For $\dd\in\N^{\Z/n\Z}$, define the counterclockwise rotation of $\dd$ by
\[
 \dd_{+1}=(\dd_{i+1}e_i)_{i\in\Z/n\Z}.
\]
We first give a lemma.

Recall the parametrization of nilpotent orbits of $C_n$ by multipartitions (Section \ref{cyclicquivers}). Let $\mathbf{m}=(\lambda^{(i)})_{i\in\Z/n\Z}$ be a multipartition of dimension $\dd\in\N^{\Z/n\Z}$. Let $x\in\OO_{\mathbf{m}}$ and $N=\max\{s\geq 0\mid x^{s+1}=0\text{ and }x^s\neq 0\}$. Let
\[
  \underline{\dd'}=(\dd'_0,\hdots,\dd'_N)
 \]
where
\[
 \dd_j'=\dim \im(x^{N-j})/\im(x^{N+1-j}).
\]
Then we have the following lemma whose proof is an easy consequence of the description of nilpotent orbits by multisegments.

\begin{lemma}\label{aperiodicdimension}
 For any $0\leq j\leq N$,
 \[
  \dd'_j-(\dd'_{j-1})_{+1}=(\sharp\{t : \lambda_t^{(i-(N-j))}=N-j\})_{i\in\Z/n\Z},
 \]
 where it is understood that $\dd'_{-1}=0$.
In particular, if $\mathbf{m}$ is aperiodic, $\dd'_j-(\dd'_{j-1})_{+1}$ has some zero coordinate.
\end{lemma}
\begin{proof}
 It suffices to note that, by the very definition of multisegments, for any $0\leq j\leq N$,
 \[
  (\dd'_j)_i=\sharp\{t : \lambda_t^{i-(N-j)}\geq N-j\}.
 \]

\end{proof}

\begin{theorem}\label{resolutioncyclic}
 Let $\dd\in\N^{\Z/n\Z}$ be a dimension vector Let $\OO\subset E_{C_n,\dd}$ a nilpotent aperiodic orbit. Then there exists a discrete flag-type $\underline{\dd}$ such that the proper morphism $\pi_{\underline{\dd}}:\FF_{\underline{\dd}}\rightarrow E_{C_n,\dd}$ has image $\overline{\OO}$ and induces a resolution of singularities of $\overline{\OO}$.
\end{theorem}

\begin{proof}
 Assume $\OO=\OO_{\mathbf{m}}$ for some aperiodic multipartition $\mathbf{m}$. First define a (usually non-discrete) flag-type of dimension $\dd$ as follows. Let $x\in\OO$ and $N=\max\{s\geq 0\mid x^{s}\neq 0\text{ and }x^{s+1}=0\}$. Define
 \[
  \underline{\dd'}=(\dd'_0,\hdots,\dd'_N)
 \]
where
\[
 \dd_j'=\dim \im(x^{N-j})/\im(x^{N+1-j})
\]
as before. Then, the dual of the proof of \cite[Proposition 1.1]{MR2057407} shows that $\pi_{\underline{\dd'}}:\FF_{\underline{\dd'}}^{\nil}\rightarrow E_{C_n,\dd}$ induces by corestriction to $\overline{\OO}$ a resolution of singularities of $\overline{\OO}$.

The next step if to refine the flag-type into another one $\underline{\dd}$, using that $\OO$ is aperiodic, such that the forgetful morphism (forgetting the additional steps of the flags) $\FF_{\underline{\dd}}\rightarrow\FF_{\underline{\dd'}}^{\nil}$ is an isomorphism over $\OO$ and the projection $\pi_{\dd}$ to $E_{\dd}$ has image $\overline{\OO}$.

Since $\mathbf{m}$ is aperiodic, $(\dd'_0)_i=0$ for some $i\in\Z/n\Z$. We replace $\dd'_0$ by the sequence of discrete dimension vectors
\[
 ((\dd'_0)_{i-1}e_{i-1},(\dd'_0)_{i-2}e_{i-2},\hdots,(\dd'_0)_{i-(n-1)}e_{i-(n-1)}).
\]
Suppose next by induction that for some $1\leq j\leq N$, $\dd'_0,\hdots,\dd'_{j-1}$ have been replaced by sequences of discrete dimension vectors. In particular, $\dd'_{j-1}$ has been replaced by
\[
 (\alpha_1e_{i_1},\hdots,\alpha_re_{i_r})
\]
for some nonnegative integers $\alpha_j$ and $i_{j}\in\Z/n\Z$ for $1\leq j\leq r$. Let $\tilde{\dd}=\dd'_j-(\dd'_{j-1})_{+1}$. By Lemma \ref{aperiodicdimension}, there exists $i\in\Z/n\Z$ such that $\tilde{\dd}_i=0$.  We replace now $\dd'_j$ by
\[
 (\alpha_{1}e_{i_1-1},\hdots,\alpha_{r}e_{i_r-1},\tilde{\dd}_{i-1}e_{i-1},\hdots,\tilde{\dd}_{i-(n-1)}e_{i-(n-1)}).
\]
The flag-type $\underline{\dd}$ obtained fulfills the conditions of the theorem. Indeed, it suffices to note that by construction, any representation of $\OO$ admits a filtration whose subquotients are of the dimensions prescribed by $\underline{\dd}$ and that the image of $\pi_{\underline{\dd}}$ is included in the image of $\pi_{\underline{\dd'}}$.
\end{proof}

\subsubsection{Description of Lusztig perverse sheaves for cyclic quivers with cyclic orientation}\label{Lusztigcyclic}
\begin{proposition}[{\cite[\S 15]{MR1088333}}]\label{lusztigcyclicq}
 For a dimension vector $\dd\in\N^{\Z/n\Z}$, $\mathscr{P}_{\dd}$ is the semisimple category generated by the simple objects $\ICC(\OO,\underline{\C})$ by varying nilpotent aperiodic $G_{\dd}$-orbits $\OO\subset E_{\dd}$.
\end{proposition}
The following proof uses the description of the singular support of Lusztig sheaves for cyclic quivers obtained in Section \ref{lusztignilpotentcyclic} and Section \ref{singsupl} to prove that non-aperiodic orbits do not give rise to Lusztig sheaves.
\begin{proof}
 Let $\OO\subset E_{\dd}$ be a nilpotent aperiodic orbit. If $\pi_{\underline{\dd}}$ is given by Proposition \ref{resolutioncyclic}, then $\ICC(\OO,\underline{\C})$ appears as a direct summand of $(\pi_{\underline{\dd}})_*\underline{\C}$. Conversely, take $\mathscr{F}$ a simple Lusztig sheaf of $E_{\dd}$. It is supported on the nilpotent locus, which consists of a finite number of orbits, each of which having a connected stabilizer in $G_{\dd}$. Therefore, $\mathscr{F}=\ICC(\OO,\underline{\C})$ for some nilpotent orbit $\OO\subset E_{\dd}$. By Section \ref{lusztignilpotentcyclic}, the singular support of $\mathscr{F}$ is contained in the union of $\overline{T^*_{\OO}E_{\dd}}$ for aperiodic nilpotent orbits $\OO\subset E_{\dd}$. Since $\overline{T^*_{\OO}E_{\dd}}$  is an irreducible component of $SS(\mathscr{F})$, it must be an aperiodic orbit.
\end{proof}

\subsection{Lusztig perverse sheaves for affine quivers}\label{Lusztigperverseaffine}
 In this Section, we recall the description of Lusztig simple perverse sheaves for affine acyclic quivers. See \cite{MR1215594,MR2371959} for proofs.
\begin{proposition}[{\cite[Proposition 6.7]{MR1215594}, \cite[Proposition 5.10]{MR2371959}}]\label{lusztigorbits}
 Let $\dd\in\N^I$ be a dimension vector. If $\OO_M\subset E_{\dd}$ is the orbit of the representation $M$ where $M$ does not have regular homogeneous direct summands or regular inhomogeneous non-aperiodic summands, then $\ICC(\OO,\underline{\C})$ is a Lusztig sheaf. 
\end{proposition}

Let $\Xi(P,I,N,\mu)$ a stratum as in Section \ref{ringelstrat}, where $\mu$ is regular semisimple and $N$ is aperiodic. Let $d\in\N$ such that $\dim\mu=d\delta$. Let
\[
 \tilde{\Xi}(P,I,N,\mu)=\{(x,\underline{x})\in \Xi(P,I,N,\mu)\times(\PP_1^{\hom})^d\mid x\simeq P\oplus I\oplus N\oplus\bigoplus_{j=1}^dS_{x_j}[1]\}.
\]
The map $\pi:=\pi_{P,I,N,\mu} : \tilde{\Xi}(P,I,N,\mu)\rightarrow \Xi(P,I,N,\mu)$ is a $\mathfrak{S}_d$-covering. Therefore, we have a decomposition
\[
 \pi_*(\underline{\C})\simeq \bigoplus_{\lambda\in\mathscr{P}_d}\mathscr{L}_{\lambda}.
\]
\begin{theorem}[\cite{MR2371959}]\label{explicitaffine}
 The simple perverse sheaves in the category $\mathcal{P}_{\dd}$ are exactly the intersection cohomology complexes $\ICC(\Xi(P,I,N,\mu),\mathscr{L}_{\lambda})$ for $(P,I,N,\mu)$ and $\lambda$ as above, with $\dim P+\dim I+\dim N+\dim\mu=\dd$. 
\end{theorem}
We call the local systems $\mathscr{L}_{\lambda}$ which appear \emph{Lusztig local systems}. A consequence of Theorem \ref{explicitaffine} is that a local system $\mathscr{L}$ on $\Xi(P,I,N,\mu)$ with $\mu$ regular semisimple is a Lusztig local system if and only if $\pi_{P,I,N,\mu}^*\mathscr{L}$ is the trivial local system on $\tilde{\Xi}(P,I,N,\mu)$.

\subsection{Local systems on the regular part}
Let $\mu$ be the regular semisimple type of dimension $d$ and $N$ a regular non-homogeneous representation. We let $\dd=\dim N+\delta\dim\mu$. We have a cartesian square (see Section \ref{Lusztigperverseaffine} and \ref{quotientreg} for the notations):
\[
\xymatrix{
\tilde{\Xi}(N,\mu)\ar[r]^{\pi_{N,\mu}}\ar[d]_{\tilde{\chi}_{N,\mu}}& \Xi(N,\mu)\ar[d]^{\chi_{N,\mu}}\\
(\PP_1^{\hom})^d\setminus\Delta\ar[r]^{\pi_d}&(S^d\PP_1^{\hom})\setminus\Delta.
}
\]

\begin{lemma}\label{locsysreg}
 Let $\mathscr{L}$ be a $G_{\dd}$-equivariant local system on $\Xi(N,\mu)\subset E_{\dd}$ for the regular semisimple type $\mu$ of dimension $d\delta$. Then it is the pull-back by $\chi_{N,\mu}$ of a local system $\mathscr{L}'$ on $S^d\PP_1^{\hom}\setminus\Delta$. Moreover, the intersection complex $\ICC(\mathscr{L})$ on $E_{d\delta}$ is a Lusztig sheaf if and only if $\pi_d^*\mathscr{L}'$ is the trivial local system.
\end{lemma}
\begin{proof}
 We postpone the proof to Section \ref{locsysrssaff} since it is analogous to that of Lemma \ref{locsysg}.
\end{proof}

\subsection{Induction and restriction of constructible complexes}\label{inductionrestriction}
Induction and restriction are the operations on constructible complexes categorifying respectively the multiplication and the comultiplication of the quantum group. We invite the reader to consult \cite{MR1088333,MR1227098,MR1062796} for more on this construction and the Ringel-Hall algebra construction of the quantum group of Ringel.

We closely follow \cite{MR3202708} and refer to it for properties of the induction and restriction functors.

\subsubsection{The induction functor}
Let $\dd',\dd''\in\N^I$ and $\dd=\dd'+\dd''$. We have the induction diagram:
\[
 \xymatrix{
 &E_{\dd',\dd''}^{(1)}\ar[r]^r\ar[ld]_p&E_{\dd',\dd''}\ar[rd]^q&\\
 E_{\dd'}\times E_{\dd''}&&&E_{\dd'+\dd''}.
 }
\]
where
\[
 E_{\dd',\dd''}=\{(x,W)\in E_{\dd}\times \Gr(\dd'',\dd)\mid \text{ and }xW\subset W\}
\]
and
\[
 E_{\dd',\dd''}^{(1)}=\{(x,W,g',g'')\mid (x,W)\in E_{\dd',\dd''}, g'' : W\stackrel{\sim}{\rightarrow} \C^{\dd''}, g' : \C^{\dd}/W\stackrel{\sim}{\rightarrow} \C^{\dd'}\}.
\]
The morphisms $r$ and $q$ are the natural projections while 
\[p(x,W,g',g'')=(g'x_{|\C^{\dd}/W}(g')^{-1},g''x_{W}(g'')^{-1}).
\]
The map $r$ is a $GL_{\dd'}\times \GL_{\dd''}$-torsor, hence induces a triangulated equivalence 
\[
r^* : D^b_{GL_{\dd}\times \GL_{\dd'}\times \GL_{\dd''}}(E_{\dd',\dd''}^{(1)})\rightarrow D^b_{\GL_{\dd}}(E_{\dd',\dd''}).
\] A quasi-inverse is denoted by $r_{\flat}$. The induction functor is
\[
 m:=q_!r_{\flat}p^*[\dim p-\dim r].
\]
For $\mathscr{F}\in D^b_{\GL_{\dd'}}(E_{\dd'})$ and $\mathscr{G}\in D^b_{\GL_{\dd''}}(E_{\dd''})$, we define
\[
 \mathscr{F}\star\mathscr{G}=m(\mathscr{F}\boxtimes \mathscr{G}).
\]
It is possible to show the associativity of $m$ (up to a natural transformation)(\cite[Proposition 1.9]{MR3202708}). This allows us to define the induction $\mathscr{F}_1\star\hdots\star\mathscr{F}_r$ for equivariant constructible complexes $\mathscr{F}_1,\hdots, \mathscr{F}_r$ on $E_{\dd_1},\hdots, E_{\dd_r}$ respectively. We write
\[
 \mathscr{F}_1\star\hdots\star\mathscr{F}_r=m(\mathscr{F}_1\boxtimes\hdots\boxtimes\mathscr{F}_r)=m_{\dd_1,\hdots,\dd_r}(\mathscr{F}_1\boxtimes\hdots\boxtimes\mathscr{F}_r)
\]
depending whether or not we want to put the emphasis on the dimensions.

\subsubsection{The restriction functor}
The restriction diagram is:
\[
 \xymatrix{
 &F_{\dd',\dd''}\ar[ld]_{\kappa}\ar[rd]^{\iota}&\\
 E_{\dd'}\times E_{\dd''}&&E_{\dd}
 }
\]
where $F_{\dd',\dd''}=\{x\in E_{\dd}\mid x\C^{\dd''}\subset \C^{\dd''}\}$, $\C^{\dd''}\subset \C^{\dd}$ being the natural inclusion. The restriction functor is
\[
 \Delta:=\kappa_!\iota^*[-\langle\dd',\dd''\rangle]
\]
where $\langle-,-\rangle$ is the Euler form of the quiver. As for the induction functor, it is possible to prove the coassociativity of the restriction. It allows us to define
\[
 \Delta_{\dd_1,\hdots,\dd_r}(\mathscr{F})
\]
for $\dd_1,\hdots,\dd_r\in\N^I$ and $\mathscr{F}$ a $G_{\dd}$-equivariant constructible complex on $E_{\dd}$, $\dd=\dd_1+\hdots+\dd_r$.

\section{Singular support of Lusztig perverse sheaves}\label{singularsupport}

\subsection{Lusztig nilpotent variety and singular support of Lusztig's sheaves}\label{lusztiglagrangian}
\subsubsection{Lusztig nilpotent variety}
We use the notation of Section \ref{representationvariety}. In particular, recall the moment map $\mu_{\dd}$. For any $\dd\in\N^I$, we define Lusztig nilpotent variety as follows:
\[
 \Lambda_{\dd}=\{x\in E_{\overline{Q},\dd}\mid \mu_{\dd}(x)=0 \text{ and $x$ is nilpotent}\}.
\]
The stacky quotient $\Lambda_{\dd}/G_{\dd}$ parametrizes nilpotent representations of the preprojective algebra of $Q$. It will be convenient for us to consider restrictions of the nilpotent variety. First let $\pi_{\dd} : T^*E_{\dd}\rightarrow E_{\dd}$ be the cotangent bundle of $E_{\dd}$. Let $\pi_{\dd}^{\Lambda}: \Lambda_{\dd}\rightarrow E_{\dd}$ be its restriction at the source. For an open subset $U\subset E_{\dd}$, we let $\Lambda_{\dd}^U=(\pi_{\dd}^{\Lambda})^{-1}(U)$. When $U=E_{\dd}^{reg}$ (resp. $E_{\dd}^{\reghom}$), see end of Section \ref{stratification}, we write $\Lambda_{\dd}^{reg}$ (resp. $\Lambda_{\dd}^{\reghom})$.

\begin{proposition}[Lusztig]
 The variety $\Lambda_{\dd}$ is a closed, conical (\emph{i.e.} $\C^{\times}$-invariant), Lagrangian subvariety of $T^*E_{\dd}$.
\end{proposition}
Such a variety can be written as
\[
 \Lambda_{\dd}=\bigcup_{S}\overline{T^*_SE_{\dd}}
\]
for some locally closed subvarieties $S\subset E_{\dd}$. The goal for Section \ref{irreducible components} is to identify such strata for finite type and affine quivers. This task is known to be much more difficult for wild quivers. In particular, the description we give rests on the representation theory of finite type and affine quivers. It seems therefore hopeless to give a uniform description for all quivers.
\subsubsection{Singular support of Lusztig sheaves}
\label{singsupl}
\begin{theorem}[Lusztig, {\cite[Corollary 13.6]{MR1088333}}]
\label{ssls}
 Let $\mathscr{F}$ be a Lusztig perverse sheaf on $E_{\dd}$. Then its singular support is a union of irreducible components of $\Lambda_{\dd}$.
\end{theorem}
\begin{proof}[Sketch of the proof]
 We only deal with constructible sheaves, therefore singular supports are Lagrangian subvarieties of $T^*E_{\dd}$ and it suffices to prove that $SS(\mathscr{F})\subset \Lambda_{\dd}$. As a shift of $\mathscr{F}$ appears as a direct summand of $(\pi_{\underline{\dd}})_*\underline{\C}$ for some discrete flag-type $\underline{\dd}$, it suffices to prove that $SS((\pi_{\underline{\dd}})_*\underline{\C})\subset \Lambda_{\dd}$. It follows from the fact that $\pi_{\underline{\dd}}$ is proper and Proposition \ref{pushforward}.
\end{proof}

\subsection{Irreducible components of Lusztig's Lagrangian for finite type and affine quivers}\label{irreducible components}

\subsubsection{Lusztig nilpotent variety for finite type quivers}
For finite type quivers, one can show that the nilpotency hypothesis in the definition of $\Lambda_{\dd}$ is redundant, therefore that $\Lambda_{\dd}=\mu_{\dd}^{-1}(0)$. Moreover we have the following description of its irreducible components.
\begin{proposition}[Lusztig, {\cite[Proposition 14.2]{MR1088333}}]
We have
\[
 \Lambda_{\dd}=\bigcup_{\OO\subset E_{\dd}}\overline{T^*_{\OO}E_{\dd}}.
\]
where the sum is indexed by $G_{\dd}$-orbits in $E_{\dd}$. 
\end{proposition}

\subsubsection{Lusztig nilpotent variety for cyclic quivers}\label{lusztignilpotentcyclic}
Recall from Section \ref{cyclicquivers} the two types of nilpotent orbits: aperiodic and not aperiodic orbits. We have the following result due to Lusztig.
\begin{proposition}[Lusztig, {\cite[Proposition 15.5]{MR1088333}}]
 The nilpotent variety is
 \[
  \Lambda_{\dd}=\bigcup_{\OO\subset E_{\dd}}\overline{T^*_{\OO}E_{\dd}}
 \]
where the union is indexed by \emph{aperiodic} nilpotent orbits.
\end{proposition}

\subsubsection{Lusztig nilpotent variety for acyclic affine quivers}
In our terminology, all results of this Section also concern cyclic quivers with acyclic orientation. Recall the Ringel stratification of the representation spaces of an affine quiver defined in Section \ref{stratification}. We have the following result.
\begin{proposition}[Ringel, {\cite[Corollary 5.3]{MR1676227}}]\label{Ringelstrat}
 We have
 \[
  \Lambda_{\dd}=\bigcup_{P,I,H,\mu}\overline{T^*_{\Xi(P,I,H,\mu)}E_{\dd}}
 \]
where the sum is indexed by the quadruples $(P,I,H,\mu)$ as in Section \ref{stratification}, where $\mu$ is regular (if $\mu(\lambda)\neq 0$ for a partition $\lambda$, then $\lambda$ has length one) and $\dim (P,I,H,\mu)=\dd$.
\end{proposition}
This result is analogous to the one in Section \ref{jordanquiver} giving a stratification of the Lie algebras of linear groups.

\subsection{Two technical lemmas}
\subsubsection{First Lemma}
Let $\mathscr{F}\in D^b_{G_{\dd}}(E_{\dd})$ be a simple $G_{\dd}$-equivariant perverse sheaf. Assume that $SS(\mathscr{F})\subset \Lambda_{\dd}$. Then by Proposition \ref{Ringelstrat}, there exists $\Xi:=\Xi(P,I,N,\mu)$ such that $\overline{\Xi}=\supp\mathscr{F}$. Let $\dd_P=\dim P$, $\dd_I=\dim I$ and $\dd_R=\dim N+\dim\mu$. We let
\[
 i : \OO_P\times E_{\dd_R}^{reg}\times \OO_I\rightarrow E_{\dd}
\]
be the direct sum and
\[
 j : \OO_P\times E_{\dd_R}^{reg}\times \OO_I\rightarrow E_{\dd_P}\times E_{\dd_R}\times E_{\dd_I}
\]
be the inclusion. We fix the standard flag $\underline{F}=(\C^{\dd_I},\C^{\dd_I+\dd_R},\C^{\dd_P+\dd_r+\dd_I})$ in $\C^{\dd}$ with subquotients of dimensions $\dd_I,\dd_R,\dd_P$. Associated to it there is a parabolic subgroup $P_{\underline{F}}$ of $G_{\dd}$ and a unipotent subgroup $U_{\underline{F}}\subset P_{\underline{F}}$ of elements respecting this flag.
\begin{lemma}\label{lemma1}
 With the above notations, 
 \[
  \mathscr{G}:=i^*\mathscr{F}[r]\in D^b_{G_{\dd_P}\times G_{\dd_R}\times G_{\dd_I}}(\OO_P\times E_{\dd_R}^{reg}\times \OO_I)
 \]
 is a perverse sheaf for $r=-\dim U_{\underline{F}}-(\dim G_{\dd}-\dim P_{\underline{F}})$. Moreover,
 \[
  \mathscr{G}=\underline{\C}[\dim\OO_P]\boxtimes \mathscr{G}_R\boxtimes\underline{\C}[\dim \OO_I]
 \]
where $\mathscr{G}_R$ is a $G_{\dd}$-equivariant simple perverse sheaf on $E_{\dd_R}^{reg}$ such that $SS(\mathscr{G}_R)\subset \Lambda_{\dd_R}^{reg}$ and $\mathscr{F}$ appears as a direct summand of $m_{\dd_P,\dd_R,\dd_I}(\mathscr{G})$.
\end{lemma}
\begin{proof}
 Recall the notation $E_{[P],\dd_R,[I]}$ from Section \ref{stratification}. To ease the notation, we define $E'$ to be this set. Let $E'_{\underline{F}}$ be the closed subset of $E'$ of elements $x$ preserving the standard flag $\underline{F}$. By Corollary \ref{uniquefiltration}, $E'\simeq E'_{\underline{F}}\times^{P_{\underline{F}}}G_{\dd}$. Let $i_1 : E'_{\underline{F}}\rightarrow E_{\dd}$ be the natural inclusion. $E'$ and $E'_{\underline{F}}$ possess the stratifications induced by Ringel's one on $E_{\dd}$. The stratification on $E'_{\underline{F}}$ is
 \[
  E'_{\underline{F}}=\bigsqcup_{P',I',N',\mu'}\Xi(P',I',N',\mu')\cap E'_{\underline{F}}
 \]
 where $(P',I',N',\mu')$ is defined by the same conditions as in Section \ref{stratification}. The notion of \emph{regular} strata is the same as in Proposition \ref{Ringelstrat}.
The pull-back $\mathscr{F}_1:=i_1^*\mathscr{F}[\dim P_{\underline{F}}-\dim G_{\dd}]$ is perverse on $E'_{\underline{F}}$. By Lemma \ref{indcv}, its singular support is included in the union of the closures of the conormal bundles to the regular strata of the stratification of $E'_{\underline{F}}$. Let $i_2 : \OO_P\times E_{\dd_R}^{reg}\times\OO_I\rightarrow E'_{\underline{F}}$ be the inclusion induced by the direct sum. We also have a projection $\pi : E'_{\underline{F}}\rightarrow \OO_P\times E_{\dd_R}^{reg}\times \OO_I$ which is a trivial fiber bundle of rank $r_{\underline{F}}:=\sum_{\alpha:i\rightarrow j}((\dd_R)_i(\dd_I)_j+(\dd_P)_I(\dd_R+\dd_I)_j)$. The stratification on $E'_{\underline{F}}$ coincides with the pull-back by $\pi$ of Ringel stratification on $\OO_P\times E_{\dd_R}^{reg}\times\OO_I$, where the strata are of the form $\OO_P\times \Xi(N',\mu)\times \OO_I$ for $\mu : \mathscr{P}\rightarrow \N$, $N'$ regular non-homogeneous, $\dim N'+\dim\mu=\dd_R$. Therefore, there exists an open stratum of the support of $\mathscr{F}_1$ of the form $\pi^{-1}(S)$ for $S\subset \OO_P\times E_{\dd_R}^{reg}\times \OO_I$ a Ringel stratum. By Lemmas \ref{sysloczerosection} and \ref{irrclosm}, the restriction of $\mathscr{F}_1$ to $\pi^{-1}(S)$ is a local system. Therefore, it is of the form $\pi^*\mathscr{L}[\dim S+r_{\underline{F}}]$ for a local system $\mathscr{L}$ on $S$. Consequently, $\mathscr{F}_1=\ICC(\pi^*\mathscr{L})$, and $i_2^*\mathscr{F}_1[-r_{\underline{F}}]=\ICC(\mathscr{L})$. From the definitions of $i, i_1$ and $i_2$, this last sheaf is $\mathscr{G}$. Since $\mathscr{G}$ is $G_{\dd_P}\times G_{\dd_R}\times G_{\dd_I}$-equivariant, $\mathscr{G}=\underline{\C}[\dim\OO_P]\boxtimes \mathscr{G}_R\boxtimes \underline{\C}[\dim\OO_I]$ where $\mathscr{G}_R$ is a simple $G_{\dd_R}$-equivariant perverse sheaf on $E_{\dd_R}^{reg}$. If $x_P$ (resp. $x_I$) is any element in the orbit of $P$ (resp. $I$),
\[
\begin{matrix}
 i_3 &:& E_{\dd_R}^{reg}&\rightarrow& \OO_P\times E_{\dd_R}^{reg}\times\OO_I\\
 &&x&\mapsto&x_P\oplus x\oplus x_I
\end{matrix}
 \]
then $\mathscr{G}_R=i_3^*\mathscr{G}[-\dim\OO_P-\dim\OO_I]$. Using Lemma \ref{indcv} again, $\mathscr{G}_R$ has its singular support in $\Lambda_{\dd_R}^{reg}$. This proves the first part of the lemma.

To prove the second part, we use the following diagram whose lower row is the induction diagram with three terms:
\[
  \xymatrix{
  \OO_P\times \Xi(N,\mu)\times \OO_I\ar[d]&A\ar[r]^{r'}\ar[d]\ar[l]_(.3){p'}&B\ar[r]^(.35){q'}\ar[d]&\Xi(P,N,I,\mu)\ar[d]\\
  E_{\dd_P}\times E_{\dd_R}\times E_{\dd_I}&E_{\dd_P,\dd_R,\dd_I}^{(2)}\ar[r]^r\ar[l]_(.4)p&E_{\dd_P,\dd_R,\dd_I}\ar[r]^(.55)q&E_{\dd}
  }
 \]
 where $A=p^{-1}(\OO_P\times\Xi(N,\mu)\times\OO_I)$, $B=r(A)$. By Corollary \ref{uniquefiltration}, $q'$ is an isomorphism and $\Xi\simeq \Xi_{\underline{F}}\times^{P_{\underline{F}}}G_{\dd}$. Moreover, $A\simeq \Xi_{\underline{F}}\times^{U_{\underline{F}}}G_{\dd}$ and all squares are cartesian. Therefore, to prove that $\mathscr{F}$ is a direct summand of $m_{\dd_P,\dd_R,\dd_I}(j_{!*}\mathscr{G})$, it suffices to prove that the restriction $\mathscr{F}_{\Xi}$ of $\mathscr{F}$ to $\Xi$ is equal to $q'_*(r')_{\flat}p'^*\mathscr{G}[\dim p-\dim r]$. We now have the following diagram:
 \[
  \xymatrix{
  &\Xi_{\underline{F}}\ar[ld]_{\pi}\ar[d]_{\iota_U}\ar[rd]^{\iota_P}&\\
  \OO_P\times \Xi(N,\mu)\times \OO_I&\Xi_{\underline{F}}\times^{U_{\underline{F}}}G_{\dd}\ar[l]_(.37){p'}\ar[r]^{r'}&\Xi_{\underline{F}}\times^{P_{\underline{F}}}G_{\dd}.
  }
 \]
By the same arguments as in the first part of the proof (recall that $\pi$ is a trivial fiber bundle of rank $r_{\underline{F}}$ and $\mathscr{G}_{\OO_P\times \Xi(N,\mu)\times\OO_I}$ and $\mathscr{F}_{\Xi}$ are local systems shifted by the dimension to make them perverse), $\pi^*\mathscr{G}[r_{\underline{F}}]\simeq \iota_P^*\mathscr{F}_{\Xi}[\dim G_{\dd}-\dim P_{\underline{F}}]$. Since the left triangle commutes, $\iota_U^*p'^*\mathscr{G}[\dim p-(\dim G_{\dd}-\dim U_{\underline{F}})]\simeq \iota_P^*\mathscr{F}_{\Xi}[\dim G_{\dd}-\dim P_{\underline{F}}]$. Since the right-hand side triangle also commutes, we get
\[
\begin{aligned}
 \iota_P^*(r')_{\flat}p'^*\mathscr{G}[\dim p'-\dim r'-(\dim G_{\dd}-\dim P_{\underline{F}})]&\simeq \iota_U^*p'^*\mathscr{G}[\dim p'-(\dim G_{\dd}-\dim U_{\underline{F}})]\\
 &\simeq\iota_P^*\mathscr{F}_{\Xi}[\dim G_{\dd}-\dim P_{\underline{F}}]
\end{aligned}
\]
By $G_{\dd}$-equivariance of both $\mathscr{F}_{\Xi}$ and $q'_*(r')_{\flat}p'^*\mathscr{G}[\dim p'-\dim r']$, this proves that they are isomorphic.
\end{proof}

The following corollary follows immediately from Lemma \ref{lemma1}.
\begin{cor}\label{restrictionth}
 Theorem \ref{mainresult} is true if and only if for any $\dd\in\N^I$ and any $G_{\dd}$-equivariant simple perverse sheaf $\mathscr{F}$ on $E_{\dd}^{reg}$ such that $SS(\mathscr{F})\subset \Lambda_{\dd}^{reg}$, $\mathscr{F}$ is the restriction to $E_{\dd}^{reg}$ of a Lusztig perverse sheaf on $E_{\dd}$.
\end{cor}

\subsubsection{Second Lemma}
Let $N$ be an inhomogeneous regular representation. Let $T\subset D$ be a set of inhomogeneous tubes such that any indecomposable direct summand of $N$ belongs to one of the tubes of $T$. Let $\dd_N=\dim N$ and $\dd=\dd_N+\dd_R$, $\dd_R$. We consider the map
\[
 i : \OO_N\times E_{\dd_R}^{D\setminus T}\rightarrow E_{\dd}.
\]
induced by the direct sum. We recall (see Section \ref{ringelstrat}) that $E_{[N],\dd_R}^{D\setminus T}$ is the locally closed subset of $E_{\dd_N+\dd_R}$ parametrizing representations of $Q$ isomorphic to $N\oplus R$ where $R$ is some regular representation of $Q$, none of whose indecomposable direct summands belongs to a tube indexed by $T$.

\begin{lemma}\label{lemma2}
  Let $\mathscr{F}$ be a simple $G_{\dd}$-equivariant perverse sheaf on $E_{\dd}$ such that $\supp\mathscr{F}\subset \overline{E_{[N]\dd_R}^{D\setminus T}}$. Then
  \[
   \mathscr{G}:=i^{*}\mathscr{F}[r]\in D^{b}_{G_{\dd_N}\times G_{\dd_R}}(\OO_N\times E_{\dd_R}^{D\setminus T})
  \]
is a perverse sheaf for $r=\dim G_{\dd_R}+\dim G_{\dd_N}-\dim G_{\dd}$. Moreover, $\mathscr{G}=\underline{\C}[\dim\OO_N]\boxtimes\mathscr{G}_{reg}$, where $\mathscr{G}_{reg}$ is a $G_{\dd_R}$ equivariant perverse sheaf on $E_{\dd_R}^{D\setminus T}$. If in addition $SS(\mathscr{F})\subset \Lambda_{\dd}$, then $SS(\mathscr{G}_{reg})\subset \Lambda_{E_{\dd_R}^{D\setminus T}}$.
\end{lemma}
\begin{proof}
 It is a consequence of the isomorphism given by Lemma \ref{isoreg} and Lemma \ref{indcv}.
\end{proof}

\section{Proof of the main theorem in the finite type case}\label{prooffinitetype}
Let $\mathscr{F}$ be a $G_{\dd}$-equivariant simple perverse sheaf on $E_{\dd}$. By Theorem \ref{gabriel}, the action of $G_{\dd}$ on $E_{\dd}$ has a finite number of orbits, and any orbit $\OO_M$ for $M$ a $\dd$-dimensional representation of $Q$ has a connected stabilizer (even irreducible as open subset of $\End(M)$). Therefore, any orbit is equivariantly simply-connected and $\mathscr{F}=\ICC(\OO,\underline{\C})$ for some $G_{\dd}$-orbit $\OO\subset E_{\dd}$. By the explicit description of Lusztig sheaves for finite type quivers (Section \ref{Lusztigft}), $\mathscr{F}$ is a Lusztig sheaf.

\begin{remark}
  We did not use explicitly the hypothesis $\SS(\mathscr{F})\subset \Lambda_{\dd}$. In fact, it is a consequence of $G_{\dd}$-equivariance for finite type quivers. Indeed, since $\mathscr{F}$ is $G_{\dd}$-equivariant, its singular support is a subset of $\mu_{\dd}^{-1}(0)$ and $\mu_{\dd}^{-1}(0)$ coincides with $\Lambda_{\dd}$ for finite type quivers (\cite[Proposition 4.14]{MR3202708}).
\end{remark}

\section{Proof of the result for cyclic quivers}\label{proofcyclic}
We now prove the main theorem of this paper (Theorem \ref{mainresult}) for cyclic quivers. The main tool is the resolution of singularities of nilpotent orbits closure.

\subsection{Proof for cyclic quivers with cyclic orientation}\label{proofcq}

Let $\mathscr{F}$ be a $G_{\dd}$-equivariant perverse sheaf on $E_{\dd}$ such that $SS(\mathscr{F})\subset \Lambda_{\dd}$. Then, $\supp(\mathscr{F})=\pi_{\dd}(SS(\mathscr{F}))\subset \pi_{\dd}(\Lambda_{\dd})\subset E_{\dd}^{\nil}$ where $\pi_{\dd} : T^*E_{\dd}\rightarrow E_{\dd}$ denotes the cotangent bundle of $E_{\dd}$. Since the nilpotent locus $E_{\dd}^{\nil}\subset E_{\dd}$ has a finite number of orbits, each of which is equivariantly simply-connected, $\mathscr{F}=\mathcal{IC}(\OO,\underline{\C})$ for some nilpotent orbit $\OO\subset E_{\dd}^{\nil}$. Then, $SS(\mathscr{F})$ contains $\overline{T^*_{\OO}E_{\dd}}$. But by Section \ref{lusztignilpotentcyclic}, $\overline{T^*_{\OO}E_{\dd}}$ is contained in $\Lambda_{\dd}$ if and only if $\OO$ is a nilpotent aperiodic orbit. By Theorem \ref{lusztigcyclicq}, $\mathscr{F}$ is a Lusztig sheaf.

\begin{remark}
 As pointed out to us by \'Eric Vasserot, we can prove that if $\mathscr{F}$ is a $G_{\dd}$-equivariant perverse sheaf on $E_{\dd}$ such that $\supp\mathscr{F}$ and $\supp\Phi\mathscr{F}$ are nilpotent, for $\Phi$ the Fourier-Sato transform reversing the orientation, then $\mathscr{F}$ is a Lusztig sheaf for the cyclic quiver. The hypothesis is weaker that the one on the singular support. This result is due to Lusztig.
\end{remark}

\subsection{Proof for cyclic quivers with arbitrary orientation}
A $G_{\dd}$-equivariant perverse sheaf on $E_{\dd}$ whose singular support is in $\Lambda_{\dd}$ is monodromic with respect to the actions of $\C^*$ by dilatation on any of the arrows. Indeed, noting that Ringel strata are $\C^*$-invariant with respect to any of these actions, this is a consequence of Corollary \ref{cormono}. Using a Fourier-Sato transform $\varPhi$ making the orientation cyclic, we obtain the sheaf $\varPhi\mathscr{F}$. By Theorem \ref{ssfourier}, this sheaf is accountable to Theorem \ref{maintheorem} for a cyclic quiver with cyclic orientation. It is therefore a Lusztig sheaf and so is $\mathscr{F}$.

 Theorem \ref{mainresult} is now proved for affine quivers of type $A$. It remains the case of affine quivers of type $D$ and $E$. It is more subtle since these quivers have three non-homogeneous tubes, and we have to see what happens around each of them. It is possible thanks to cyclic quivers.

\section{A larger class of perverse sheaves for cyclic quivers}\label{extendedps}
\subsection{Extension of the Hall category}
In this Section, we let $C_n$ be the cyclic quiver, with labeling of vertices and arrows as in Section \ref{cyclicquivers}. Let $\dd\in\N^{\Z/n\Z}$. We define $\tilde{\mathscr{Q}}_{\dd}$ as the full additive subcategory of $D^b_{c,G_{\dd}}(E_{\dd})$ generated by direct summands of the constructible sheaves $(\pi_{\underline{\dd}})_*\underline{\C}_{\tilde{\FF}_{\underline{\dd}}}$ for \emph{all} (\emph{i.e.} not necessarily discrete) flag-types $\underline{\dd}$ of dimension $\dd$. We call $\tilde{\mathscr{Q}}_{\dd}$ the \emph{extended Hall category} and we denote by $\tilde{\mathscr{P}}_{\dd}$ the full subcategory of perverse sheaves which are in $\tilde{\mathscr{Q}}_{\dd}$.

We will describe the simple perverse sheaves in $\tilde{\mathscr{P}}_{\dd}$ explicitly by exhibiting for each of them a local system on a smooth open subset of its support. We call such local systems \emph{extended Lusztig local systems}. It is very analogous to Theorem \ref{explicitaffine}. We will also be able to describe their Fourier transform when reversing all the arrows of $C_n$ and their singular support.

\subsection{Singular support: the extended nilpotent variety}
Recall the moment maps
\[
 \mu_{\dd}:E_{\overline{C}_n,\dd}=T^*E_{C_n,\dd}\rightarrow \mathfrak{gl}_{\dd}
\]
from Section \ref{representationvariety}. An element of $T^*E_{\C_n,\dd}=E_{C_n,\dd}\oplus E_{\C_n,\dd}^*$ is denoted $(x,x^*)$. We first recall the definition of $*$-semi-nilpotent elements from \cite[Section 1.1]{bozec2017number} in the particular case of cyclic quivers (although the general definition is exactly the same). An element $(x,x^*)\in E_{\overline{C}_n,\dd}$ is called $*$-semi-nilpotent if there exists a $\Z/n\Z$-graded flag $(0=F_0\subset F_1\subset\hdots\subset F_r=\C^{\dd})$ of $\C^{\dd}$ such that for any $1\leq j\leq r$,
\[
 x^*(F_j)\subseteq F_{j-1} \quad\text{ and }\quad x(F_j)\subseteq F_j.
\]

We let
\[
 \tilde{\Lambda}_{\dd}=\{(x,x^*)\in E_{\overline{C_n},\dd}\mid \mu_{\dd}(x,x^*)=0 \text{ and $(x,x^*)$ is $*$-semi-nilpotent}\}.
\]
Recall the stratification of the representation spaces of cyclic quiver from Section \ref{stratcyclic}.
\begin{proposition}\label{extendednilvar}
 The subvariety $\tilde{\Lambda}_{\dd}$ of $T^*E_{C_n,\dd}$ is closed , conical and Lagrangian. We have
 \[
  \tilde{\Lambda}_{\dd}=\bigcup_{(N,\mu)}\overline{T^*_{\Xi(N,\mu)}E_{C_n,\dd}},
 \]
where the union runs over pairs $(N,\mu)$ with $N$ a nilpotent aperiodic representation of $C_n$ and $\mu:\mathscr{P}\rightarrow \N$ is regular such that $\dim N+\delta\dim\mu=\dd$
\end{proposition}
\begin{proof}
 The first two properties (closed and conical) are fairly obvious. It is Lagrangian by \cite[Theorem 1.4]{MR3569998}. The proof of this decomposition in irreducible components is completely analogous to the one of \cite{MR1676227}. Everything can be adapted by replacing the Auslander translation $\tau$ by the rotation of representations: if $x=(x_i)_{i\in\Z/n\Z}$ is a representation of $C_n$, $\tau(x)=(x_{i-1})_{i\in\Z/n\Z}$.
\end{proof}

\begin{proposition}
 From the point of view of the projection $\tilde{\Lambda}_{\dd}\rightarrow E_{C_n,\dd}^{\op}$, $(x,x^*)\mapsto x^*$,
 \[
  \tilde{\Lambda}_{\dd}=\bigcup_{\OO\subset E_{C_n,\dd}^{\op}}T^*_{\OO}E_{C_n,\dd}^{\op}
 \]
 where the union runs over all nilpotent orbits $\OO\subset E_{C_n,\dd}^{\op}$.

\end{proposition}\label{extendednilvarnil}
\begin{proof}
 If $(x,x^*)\in T^*_{\OO} E_{C_n,\dd}^{\op}$ for some nilpotent orbit $\OO\subset E_{C_n,\dd}^{\op}$, then $x^*\in \OO$, therefore $x^*$ is nilpotent. Moreover, for any $a\in\mathfrak{g}_{\dd}$, $\Tr([a,x]x^*)=0$, since $T_x\OO=\{[a,x]:a\in\mathfrak{g}_{\dd}\}$ is the tangent space of $\OO$ at $x$. Then, for any $a\in \mathfrak{g}_{\dd}$, $\Tr(a[x,x^*])=0$, and as a consequence, $\mu_{\dd}(x,x^*)=[x,x^*]=0$. We now see $x$ and $x^*$ as endomorphisms of $\C^{\dd}$. Let $r\geq 0$ such that $(x^*)^r=0$. We choose
 \[
  \underline{F}=(0\subset \ker x^*\subset \hdots\subset\ker (x^*)^r=\C^{\dd})
 \]
and let $F_j=\ker (x^*)^j$ for $0\leq j\leq r$. Obviously $x^*F_j\subset F_{j-1}$ for $1\leq j\leq r$ and since $x$ and $x^*$ commute, $xF_j\subset F_j$ for $1\leq j\leq r$. This shows that $(x,x^*)$ is $*$-semi-nilpotent. We proved that the left-hand-side contains the right-hand-side. Now, if $(x,x^*)\in\tilde{\Lambda}_{\dd}$, $x^*$ is nilpotent by definition. Let $\OO\subset E_{C_n,\dd}^{\op}$ be its orbit. The condition $\mu_{\dd}(x,x^*)=0$ implies that $(x,x^*)\in T^*_{\OO}E_{C_n,\dd}^{\op}$, proving the reverse inclusion.
\end{proof}
\begin{remark}
 Proposition \ref{extendednilvarnil} provides an alternative proof that $\tilde{\Lambda}_{\dd}$ is Lagrangian.
\end{remark}

\begin{theorem}\label{ssextendednilvar}
 The singular supports of the sheaves of the extended Hall category $\tilde{\mathscr{Q}}_{\dd}$ are unions of irreducible components of $\tilde{\Lambda}_{\dd}$.
\end{theorem}
\begin{proof}
  The proof is an easy adaptation of \cite[Section 13]{MR1088333}.
  \end{proof}

\subsection{Explicit description}\label{explicitdescription}
We give a description of the simple perverse sheaves in the category $\tilde{\mathcal{P}}_{\dd}$ in the spirit of \cite{MR1215594} and \cite{MR2371959}. Let $\dd\in\N^{\Z/n\Z}$. We consider the Fourier-Sato transform $\Phi$ which reverse all the arrows of $C_n$. The Fourier-Sato transforms of perverse sheaves on $E_{C_n,\dd}$ are perverse sheaves on $E_{C_n,\dd}^{\op}$. Moreover, the Fourier-Sato transform of a $G_{\dd}$-equivariant perverse sheaf is again $G_{\dd}$-equivariant.
\begin{lemma}\label{fouriernilp}
 The Fourier transforms of the simple perverse sheaves in $\tilde{\mathcal{P}}_{\dd}$ are intersection cohomology complexes $\ICC(\OO,\underline{\C})$ of nilpotents orbits $\OO\subset E_{C_n,\dd}^{\nil}$.
\end{lemma}
\begin{proof}
 Let $\mathscr{F}$ be a simple perverse sheaf in $\tilde{\mathcal{P}}_{\dd}$. Its singular support is contained in $\tilde{\Lambda}_{\dd}$ by Theorem \ref{ssextendednilvar}. The support of $\Phi\mathscr{F}$ is the projection of $SS(\mathscr{F})$ to $E_{C_n,\dd}^{\op}$ by Theorem \ref{ssfourier}. It is contained in $E_{C_n,\dd}^{\nil,\op}$. This proves that $\Phi\mathscr{F}=\ICC(\OO,\underline{\C})$ for some nilpotent orbit $\OO\subset E_{C_n,\dd}^{\op}$.
\end{proof}
We now give a combinatorial parametrization of nilpotent orbits of cyclic quiver.

\begin{lemma}\label{nilporbit}
 Nilpotent orbits in $E_{C_n,\dd}$ are parametrized by pairs $(N,\lambda)$ where $N$ is aperiodic and $\lambda$ is a partition such that $\dim N+\lvert\lambda\rvert\delta=\dd$. The nilpotent orbit corresponding to $(N,\lambda)$ is the orbit of the representation 
 \[
N\oplus \bigoplus_{\substack{r\geq 1\\i\in\Z/n\Z}}I_{i,\lambda_r}.
\]
\end{lemma}
\begin{proof}
 Let $M$ be a nilpotent representation of $C_n$ of dimension $\dd$. We can write it $M=N\oplus P$ where $N$ is aperiodic and $P$ is completely periodic, meaning that $P=N_{\mathbf{m}}$ for a multipartition $\mathbf{m}:\Z/n\Z\rightarrow \mathscr{P}$ such that $\mathbf{m}(i)=\mathbf{m}(j)$ for any $i,j\in\Z/n\Z$. Therefore, the data of $\mathbf{m}$ is equivalent to the data of a partition $\lambda=\mathbf{m}(0)$. Moreover, the total dimension of $P$ is $\sum_{i\in\Z/n\Z}(\dim P)_i=n\lvert\lambda\rvert$ and by periodicity, $(\dim P)_i=(\dim P)_j$ for any $i,j\in\Z/n\Z$. As a consequence, $\lvert\lambda\rvert=(\dim P)_0$ and $\dim N+\lvert\lambda\rvert\delta=\dd$. Putting these facts together, we obtain the lemma.
\end{proof}

We are ready to describe extended Lusztig perverse sheaves. Let $(N,\mu)$ be a pair with $N$ the isoclass of a nilpotent aperiodic representation and $\mu$ regular semisimple such that $\dim N+\delta\dim \mu=\dd$. Let $d\in\N$ such that $\dim\mu=d\delta$. We can describe te stratum $\Xi(N,\mu)$:
\[
 \Xi(N,\mu)=\{x\in E_{C_n,\dd}\mid x\simeq N\oplus\bigoplus_{j=1}^dJ_{1}(x_j) \text{ for some $(x_1,\hdots,x_d)\in (\C^*)^d\setminus\Delta$}\}.
\]
Let $\tilde{\Xi}(N,\mu)=\{(x,x_1,\hdots,x_r)\in E_{C_n,\dd}\times ((\C^*)^d\setminus \Delta)\mid x\simeq N\oplus\bigoplus_{j=1}^d J_1(x_j)\}$. The map 
\[
\begin{matrix}
\pi_{N,\mu} &:& \tilde{\Xi}(N,\mu)&\rightarrow& \Xi(N,\mu)\\
&&(x,\underline{x})&\mapsto&x
\end{matrix}
\]
 is a $\mathfrak{S}_d$ cover. Therefore, we obtain a family $(\mathscr{L}_{N,\lambda})_{\lambda\in\mathscr{P}_d}$ of local systems on $\Xi(N,\mu)$ indexed by partition of $d$. The main theorem of this section is the following.
 
 \begin{theorem}\label{explicitcyclic}
  The simple perverse sheaves in $\tilde{\mathcal{P}}_{\dd}$ are the intersection cohomology complexes $\ICC(\Xi(N,\mu),\mathscr{L}_{N,\lambda})$ for $(N,\mu)$ and $\lambda$ such that $N$ is nilpotent aperiodic, $\mu$ is regular semisimple, $\dim\mu=d$, $\dim N+\dim \mu=\dd$ and $\lambda$ is a partition of $d$.
 \end{theorem}
\begin{proof}
 Let $\dd_N=\dim N$. Let $\underline{\dd}_N$ be a discrete flag-type given by Theorem \ref{resolutioncyclic} for the orbit $\OO_N$. We define the flag-type $\underline{\dd}=(\underline{\dd}_N,\delta,\hdots,\delta)$ of $\dd$ (there is $d$ copies of $\delta$). Then we obtain the projective morphism $\pi_{\underline{\dd}}:\tilde{\FF_{\underline{\dd}}}\rightarrow E_{\dd}$ whose image is $\overline{\Xi(N,\mu)}$. Moreover, the restriction of $\pi_{\underline{\dd}}$ to
 \[
  \pi_{\underline{\dd}}^{-1}(\Xi(N,\mu))\rightarrow \Xi(N,\mu)
 \]
is a $\mathfrak{S}_d$-covering. Hence,
\[
 \bigoplus_{\lambda\in\mathscr{P}_d}\ICC(\mathscr{L}_{N,\lambda})
\]
is a direct factor of $(\pi_{\underline{\dd}})_*\underline{\C}$. This proves that the perverse sheaves defined in the theorem are extended Lusztig sheaves.

Now, all the perverse sheaves defined are pairwise non-isomorphic. Combining Lemma \ref{fouriernilp} and the fact that this sheaves are parametrized by the same set as nilpotent orbits (Lemma \ref{nilporbit}), it follows that the theorem gives a complete description of the extended Lusztig category.

\end{proof}
We will call the local systems $\mathscr{L}_{N,\lambda}$ which appear during this process \emph{extended Lusztig local systems}.

A a corollary of the proof of Theorem \ref{explicitcyclic}, we obtain the following result.
\begin{cor}\label{fouriernilp}
 The Fourier-Sato transforms of the simple perverse sheaves in $\tilde{\mathcal{P}}_{\dd}$ are exactly the intersection cohomology complexes $\ICC(\OO,\underline{\C})$ for nilpotents orbits $\OO\subset E_{C_n,\dd}^{\nil}$.
\end{cor}

 Consider the map
\[
 \chi_{N,\mu}:\Xi(\mu)\rightarrow S^d(\C^*)\setminus\Delta
\]
defined in Section \ref{opensubsets}. We have a cartesian diagram
\begin{equation}
 \xymatrix{
 \tilde{\Xi}(N,\mu)\ar[r]^{\pi_{N,\mu}}\ar[d]_{\tilde{\chi}_{N,\mu}}&\Xi(N,\mu)\ar[d]^{\chi_{N,\mu}}\\
 (\C^*)^d\setminus\Delta\ar[r]^{\pi_d}&S^d(\C^*)\setminus\Delta
 }
\end{equation}
\begin{lemma}\label{pbcyclic}
 A $G_{\dd}$-equivariant local system $\mathscr{L}$ on $\Xi(N,\mu)$ is the pull-back of a local system $\mathscr{L}'$ on $S^d(\C^*)\setminus\Delta$. The local system $\mathscr{L}'$ is unique up to isomorphism. Moreover, $\mathscr{L}$ is an extended Lusztig local system if and only if $\pi_d^*\mathscr{L}'$ is trivial.
\end{lemma}
\begin{proof}
 We postpone the proof to Section \ref{locsysrssaff}.
\end{proof}

\subsection{Microlocal characterization of sheaves in the extended Hall category}
The following result is a statement of the main theorem of this paper, Theorem \ref{mainresult}, for cyclic quivers and the extended Hall category.

\begin{theorem}\label{miccharec}
 Let $\mathscr{F}\in\Perv_{G_{\dd}}(E_{\dd})$ be a simple $G_{\dd}$-equivariant perverse sheaf on $E_{\dd}$ such that $SS(\mathscr{F})\subset \tilde{\Lambda}_{\dd}$. Then, $\mathscr{F}$ is in $\tilde{\mathcal{P}}_{\dd}$.
\end{theorem}
\begin{proof}
 Since the strata $\Xi(N,\mu)$ defined is Section \ref{stratcyclic} are $\C^*$-invariant, by Corollary \ref{cormono}, $\mathscr{F}$ is monodromic. By Theorem \ref{ssfourier} and the definition of $\tilde{\Lambda}_{\dd}$, its Fourier-Sato transform $\Phi\mathscr{F}$ is a simple perverse sheaf on $E_{C_n,\dd}^{\op}$ with nilpotent support. Therefore, $\mathscr{F}=\ICC(\OO,\underline{\C})$ for some nilpotent orbit $\OO\subset E_{C_n,\dd}^{\op}$. By Theorem \ref{fouriernilp}, $\mathscr{F}$ is in $\tilde{\mathcal{P}}_{\dd}$.
\end{proof}

\subsection{Restriction of perverse sheaves to a neighbourhood of nilpotent representations}
Define the restriction of the extended nilpotent variety
\[
 \tilde{\Lambda}_{\dd}^{<1}=\bigcup_{(N,\mu)}\overline{T^*_{\Xi^{<1}(N,\mu)}E_{\dd}^{<1}}.
\]
Let
\[
 j_{\dd}^{<1}:E_{\dd}^{<1}\rightarrow E_{\dd}
\]
be the inclusion. 
\begin{lemma}\label{equivalence}
 The restriction of perverse sheaves $(j_{\dd}^{<1})^*$ induces an equivalence of categories
 \[
  (j_{\dd}^{<1})^*:\Perv_{G_{\dd}}(E_{\dd},\tilde{\Lambda}_{\dd})\rightarrow \Perv_{G_{\dd}}(E_{\dd}^{<1},\tilde{\Lambda}_{\dd}^{<1}).
 \]

\end{lemma}
\begin{proof}
 We construct a quasi inverse to $(j_{\dd}^{<1})^*$. Let $\mathscr{F}^{<1}\in\Perv_{G_{\dd}}(E_{\dd}^{<1},\tilde{\Lambda}_{\dd}^{<1})$. There exists a stratum $\Xi^{<1}(N,\mu)$ and a local system $\mathscr{L}^{<1}$ on it such that $\mathscr{F}=\ICC(\mathscr{L})$. But since $\Xi(N,\mu)$ and $\Xi^{<1}(N,\mu)$ have isomorphic fundamental groups, $\mathscr{L}^{<1}$ can be extended to a local system $\mathscr{L}$ on $\Xi(N,\mu)$. We set $\mathscr{F}=\ICC(\mathscr{L})$. Moreover, since $\mathscr{L}$ is $\C^*$-equivariant, $\mathscr{F}$ is also $\C^*$-equivariant, which implies that $SS(\mathscr{F})\subset \tilde{\Lambda}_{\dd}$.
\end{proof}

\section{Neighbourhood of non-homogeneous tubes in the representation spaces of affine quivers}\label{neighb}
 We recall the construction of the Hall functor and the Hall morphism appearing in \cite[Section 3]{MR2371959}. Let $Q$ be an affine quiver. Let $D\subset \PP^1(\C)$ be the subset indexing the non-homogeneous tubes of $Q$. In the sequel, we assume that $D$ is non-empty. Let $\CC$ be a non-homogeneous tube of $Q$ of period $p>1$ (Theorem \ref{ringelth}) corresponding to $t\in D$. Let $N_0,\hdots,N_{p-1}$ be a full set of representatives of simple regular representations in the tube $\CC$ ordered such that for any $s\in \Z/p\Z$, $\Ext^1(N_s,N_{s+1})\simeq \C$. Let us fix for each $s\in \Z/p\Z$ a nontrivial extension:
 \[
  0\rightarrow N_{s+1}\rightarrow E_s\rightarrow N_s\rightarrow 0.
 \]
As vector spaces, $E_s=N_{s+1}\oplus N_{s}$. The extension $E_s$ gives linear maps
\[
 z_{\alpha,s} : (N_s)_i\rightarrow (N_{s+1})_j
\]
for any $\alpha:i\rightarrow j\in\Omega$. The representation $N_s$ of $Q$ is given by the linear maps $a_{s,\alpha}:(N_s)_i\rightarrow (N_s)_j$ for $\alpha:i\rightarrow j\in\Omega$.
Write $\dd_s=\dim N_s$ for $s\in\Z/p\Z$. We have the map between dimension lattices
\[
\begin{matrix}
 \dim &:& \Z^{\Z/p\Z}&\rightarrow& \Z^{I}\\
 &&\ee=(e_s)&\mapsto&\tilde{\ee}=\sum_{s\in\Z/p\Z}e_s\dd_s.
\end{matrix}
\]
Let $\ee=(e_s)_{s\in\Z/p\Z}\in \Z^{\Z/p\Z}$. Let $V=(V_s)_{s\in\Z/p\Z}$ be a fixed $\Z/p\Z$-graded vector space. Let $W$ be the underlying $I$-graded vector space of $\bigoplus_{s\in\Z/p\Z}V_s\otimes N_s$. We construct a fully faithful functor
\[
\begin{matrix}
 F &:& \Rep_{C_p}(k)&\rightarrow& \Rep_{Q}(k)\\
 &&(V_s,v_s)_{s\in \Z/p\Z}&\mapsto&(W,x)
\end{matrix}
\]
as follows: for $\alpha:i\rightarrow j\in\Omega$,
\[
 \begin{matrix}
  x_{\alpha} &:& W_i&\rightarrow& W_j\\
 \end{matrix}
\]
is defined for any $v\otimes n\in V_s\otimes (N_s)_i$ by
\[
 x_{\alpha}(v\otimes n)=v\otimes a_{s,\alpha}(n)+v_s(v)\otimes z_{\alpha,s}(n).
\]
At the level of representation varieties, the same formulas give a morphism:
\[
 \iota_E:E_{C_p}(V)\rightarrow E_{Q}(W).
\]
Define now the closed embedding of algebraic groups:
\[
\begin{matrix}
 \iota_G&:&GL(V)&\rightarrow& GL(W)\\
 &&(g_s)_{s\in\Z/p\Z}&\mapsto&(\sum_{s\in \Z/p\Z}g_s\otimes \id_{(N_s)_i})_{i\in I}.
\end{matrix}
\]
We obtain a locally closed immersion
\[
 j_V:E_{C_p}(V)\times^{GL(V)}GL(W)\rightarrow E_{Q}(W).
\]
We let $\dd=\dim V$ so that $\tilde{\dd}=\dim W$ and we see $\iota_E$, $\iota_G$ and $j_{\dd}:=j_V$ as morphism between the corresponding representation varieties:
\[
 \iota_E:E_{C_p,\dd}\rightarrow E_{Q,\tilde{\dd}},
\]
\[
 \iota_G:G_{C_p,\dd}\rightarrow G_{W,\tilde{\dd}},
\]
and
\[
 j_{\dd}:E_{C_p,\dd}\times^{G_{C_p,\dd}}G_{Q,\tilde{\dd}}\rightarrow E_{Q,\tilde{\dd}}.
\]
We will also need restrictions of $\iota_E$ and $j_{\dd}$:
\[
 \iota_E^{<1}:E_{C_p,\dd}^{<1}\rightarrow E_{Q,\tilde{\dd}}
\]
and
\[
 j_{\dd}^{<1}:E_{C_p,\dd}^{<1}\times^{G_{C_p,\dd}}G_{Q,\tilde{\dd}}\rightarrow E_{Q,\tilde{\dd}}.
\]

From \cite[Lemma 3.2 and Lemma 3.3]{MR2371959}, we have the following proposition.
\begin{proposition}\label{identiftubes}
 $F$ is an exact fully faithful functor. It takes value in the full subcategory of regular representations of $Q$. By restriction, it induces a equivalence of categories
 \[
  \Rep^{\nil}_{C_p}(k)\rightarrow \CC.
 \]
 (Recall that $\CC$ is the chosen non-homogeneous tube of $Q$). Let $\Rep_Q^{reg\CC}(k)$ be the full subcategory of objects of $\Rep_Q(k)$ isomorphic to a direct sum $N\oplus R$ where $N$ is an object of $\CC$ and $R$ is regular homogeneous. Let $\Rep_{C_p}^{\CC}(k)$ be the full subcategory of objects of $\Rep_{C_p}(k)$ whose image by $F$ belongs to the subcategory $\Rep_{Q}^{reg\CC}(k)$ of regular representations of $Q$ whose non-homogeneous direct summands lie in $\CC$. Then $F$ restricts to an equivalence of categories
 \[
  \Rep_{C_p}^{\CC}(k)\rightarrow \Rep_Q^{reg\CC}(k).
 \]
 \end{proposition}
\begin{remark}
 The image of F contains all homogeneous regular tubes and misses exactly one inhomogeneous tube (if $Q$ has three inhomogeneous tubes, that is $Q$ is of type $D$ or $E$).
\end{remark}

 At the level of representation varieties, we consider the subvariety $E_{Q,\tilde{\dd}}^{\{t\}}$ of $E_{Q,\tilde{\dd}}$ defined in Section \ref{ringelstrat}. It is an open subvariety contained in the image of $j_{\dd}$. The subvariety $j_{\dd}^{-1}(E_{Q,\tilde{\dd}}^{\{t\}})=\iota_E^{-1}(E_{Q_{\tilde{\dd}}}^{\{t\}})\times^{G_{\dd}}G_{\tilde{\dd}}$ is open. The open subset $\iota_E^{-1}(E_{Q,\tilde{\dd}}^{\{t\}})$ consists of the elements $x\in E_{C_p,\dd}$ which avoid exactly one nonzero eigenvalue if $Q$ has three non-homogeneous tubes, and coincide with $E_{C_p,\dd}$ otherwise. In this case, by choosing properly the extensions $E_s$, we can assume that the avoided eigenvalue is $1$. We let
 \[
  D'=
\left\{
\begin{aligned}
 &\{0\} \text{ if $Q$ has one or two non-homogeneous tubes}\\
 &\{0,1\}\text{ else}.
\end{aligned}
\right.
 \]

 In dimension $\delta$, we obtain a cartesian diagram
\[
 \xymatrix{
 E_{C_p,\delta}^{\C\setminus D'}\times^{G_{C_p,\delta}}G_{Q,\delta}\ar[r]\ar[d]&E_{Q,\delta}^{\reghom}\ar[d]\\
 \A^1\setminus \{0,1\}\ar[r]&\PP^1\setminus D
 }
\]
whose horizontal arrows are isomorphisms.

We also have the cartesian diagram
\[
 \xymatrix{
 E_{C_p,\delta}^{\C\setminus\{1\}}\times^{G_{C_p,\delta}}G_{Q,\delta}\ar[r]\ar[d]&E_{Q,\delta}^{\{t\}}\ar[d]\\
 (\A^1\setminus \{1\})\ar[r]&(\PP^1\setminus D)\cup\{t\}
 }
\]
whose horizontal rows are isomorphisms. If $N=F(N')$ is a non-homogeneous regular representation of $Q$ for some nilpotent representation $N'$ of $C_p$, and $\mu$ is regular semisimple such that $\dd=\dim N+(\dim\mu)\delta$, we obtain a cartesian square: 
\[
 \xymatrix{
 \Xi(N',\mu)\ar[r]\ar[d]&\Xi(N,\mu)\ar[d]\\
 S^d(\A^1\setminus\{1\})\setminus\Delta\ar[r]&S^d((\PP^1\setminus D)\cup\{t\})\setminus\Delta
 }
\]
 where $d=\dim \mu$.
 
We will also need the open immersion
\[
 j_{\dd}^{<1}:E_{C_p,\dd}^{<1}\times^{G_{C_p,\dd}}G_{Q,\tilde{\dd}}\rightarrow E_{Q,\tilde{\dd}}^{reg}.
\]
A crucial property is given by the following proposition whose proof is obvious from Proposition \ref{Ringelstrat}.

\begin{proposition}\label{restnil}
 The restriction of the Lusztig nilpotent variety $\Lambda_{\dd}$ by $j_{\dd}^{<1}$ is
 \[
  (j_{\dd}^{<1})^*\Lambda_{\dd}=\bigcup_{(N,\mu)}\overline{T^*_{\Xi^{<1}(N,\mu)\times^{G_{C_p,\dd}}G_{Q,\tilde{\dd}}}(E_{C_p,\dd}^{<1}\times^{G_{C_p,\dd}}G_{Q,\tilde{\dd}})}.
 \]
Consequently, if $\mathscr{F}$ is a perverse sheaf on $E_{Q,\tilde{\dd}}$ such that $SS(\mathscr{F})\subset\Lambda_{\tilde{\dd}}$, then $\mathscr{G}:=(\iota_E^{<1})^*\mathscr{F}[\dim G_{C_p,\dd}-\dim G_{Q,\tilde{\dd}}]$ is a perverse sheaf on $E_{C_p,\dd}^{<1}$ such that $SS(\mathscr{G})\subset \tilde{\Lambda}_{\dd}^{<1}$. 
\end{proposition}

\section{Proof of the main theorem in the affine case}\label{proofaffineq}

\subsection{Local systems on a stratum of the cyclic quiver}
We record the following results, which are an easy consequence of the action of $\C^*$ on all the representation varieties and stratifications considered.
 
\begin{lemma}
 Let $\Xi(N,\mu)\subset E_{C_p,\dd}$ be a stratum. Then the inclusion $\Xi^{<1}(N,\mu)\rightarrow \Xi(N,\mu)$ induces an isomorphism between the fundamental groups $\pi_1(\Xi(N,\mu))$ and $\pi_1(\Xi^{<1}(N,\mu))$.
\end{lemma}
Therefore, there is a canonical bijective correspondance between isomorphism classes of local systems $\mathscr{L}$ on $\Xi(N,\mu)$ and $\mathscr{L}^{<1}$ on $\Xi^{<1}(N,\mu)$.

\begin{lemma}
Let $\mathscr{L}^{<1}$ be a local system on $\Xi^{<1}(N,\mu)$ such that $SS(j^{<1}_{!*}\mathscr{L}^{<1})=\bigcup_{S\in\mathcal{S}^{<1}}\overline{T^*_SE_{C_p,\dd}^{<1}}$ where $\mathcal{S}^{<1}$ is a set of strata of the form $\Xi^{<1}(N,\mu)$. Then, if $\mathscr{L}$ denotes the local system corresponding to $\mathscr{L}^{<1}$ and $\mathcal{S}$ the set of strata corresponding to $\mathcal{S}^{<1}$, 
\[
 SS(j_{!*}\mathscr{L})=\bigcup_{S\in\mathcal{S}}\overline{T^*_SE_{C_p,\dd}}.
\]

\end{lemma}

\subsection{Proof of Theorem \ref{mainresult} for affine quivers}
Let $\mathscr{F}$ be a $G_{\dd}$-equivariant perverse sheaf on $E_{\dd}^{reg}$ whose singular support is contained in $\Lambda_{\dd}^{reg}$. Then, by Corollary \ref{restrictionth},  it is supported on $\overline{\Xi(N,\mu)}$ for some non-homogeneous regular representation $N$ and $\mu$ regular. Let $d=\dim\mu$. Our argument proceeds in two steps: first we prove that $\mu$ is regular semisimple; then, since $\Xi(N,\mu)$ is smooth and $\SS(\mathscr{F}_{\Xi(N,\mu)})=T^*_{\Xi(N,\mu)}E_{\dd}$, by Lemma \ref{sysloczerosection}, $\mathscr{F}=\ICC(\mathscr{L})$ for a local system $\mathscr{L}$ on $\Xi(N,\mu)$; we then prove that $\mathscr{L}$ is a Lusztig local system on $\Xi(N,\mu)$.

Let $t\in D$ be a non-homogeneous tube. of period $p>1$. Write $N=N_t\oplus N'$ where $N_t$ is in the tube indexed by $t$ and none of the direct summands of $N'$ belongs to the tube indexed by $t$. Write $\dd'=\dim N_t+\dim\mu\delta$. By the isomorphism
\[
 (\OO_{N'}\times E_{\dd'}^{\{t\}})\times^{G_{\dd_{N'}}\times G_{\dd'}}G_{\dd}\rightarrow E_{[N'],\dd'}^{\{t\}},
\]
we obtain a perverse sheaf $\mathscr{G}$ on $E_{\dd'}^{\{t\}}$ such that $SS(\mathscr{G})\subset (\Lambda_{\dd'})_{E_{\dd'}^{\{t\}}}$ (see Lemma \ref{lemma2}). By Proposition \ref{identiftubes}, there exists $\ee\in\N^{\Z/p\Z}$ such that $\tilde{\ee}=\dd'$.

Let $i_{\ee}^{<1}:E_{C_p,\ee}^{<1}\rightarrow E_{Q,\tilde{\ee}}$ be the composition of $j_{\ee}^{<1}$ with the closed immersion $E_{C_p\ee}^{<1}\rightarrow E_{C_p,\ee}^{<1}\times^{G_{C_p,\ee}}G_{Q,\tilde{\ee}}$. Then
\[
 (i_{\ee}^{<1})^*\mathscr{G}[\dim G_{C_p,\dd}-\dim G_{Q,\tilde{\ee}}]
\]
is a perverse sheaf on $E_{C_p,\ee}^{<1}$ and by Proposition \ref{restnil} and Lemma \ref{indcv}, $SS(\mathscr{G})\subset \tilde{\Lambda}_{\ee}^{<1}$. By Lemma \ref{equivalence}, it can be extended to a perverse sheaf $\mathscr{G}'$ on $E_{C_p,\ee}$ such that $SS(\mathscr{G}')\subset \tilde{\Lambda}_{\ee}$. By Theorem \ref{miccharec}, it is a Lusztig sheaf in the extended category. Therefore, by Theorem \ref{explicitcyclic}, $\mu$ has to be regular semisimple. Now, as stated above, $\mathscr{F}=\ICC(\mathscr{L})$ for some local system $\mathscr{L}$ on $\Xi(N,\mu)$. By Lemma \ref{locsysreg}, there exists a local system $\tilde{\mathscr{L}}$ on $S^d\PP_1^{\hom}\setminus\Delta$ such that $\mathscr{L}\simeq \chi_{N,\mu}^*\tilde{\mathscr{L}}$, where $\chi_{N,\mu}$ is the support map whose construction is sketched in Section \ref{quotientreg}. We have to prove that $\pi_d^*\tilde{\mathscr{L}}$ is trivial (recall that $\pi_d:(\PP_1^{\hom})^d\setminus\Delta\rightarrow S^d\PP_1^{\hom}\setminus\Delta$ is the $\mathfrak{S}_d$-covering).

Again by Lemma \ref{sysloczerosection}, the perverse sheaf $\mathscr{G}'$ is given by a local system $\mathscr{L}'$ on $\Xi(N_t',\mu)$, where $N'_t$ is a representation of $C_p$ such that $F(N'_t)\simeq N$ and $F$ is the functor of Section \ref{neighb}. Moreover, by Lemma \ref{pbcyclic}, $\pi_{N'_t,\mu}^*\mathscr{L}'$ is the trivial local system. By Lemma \ref{pbcyclic}, $\mathscr{L}'={\chi}_{N'_t,\mu}\tilde{\mathscr{L}'}$ for some local system $\tilde{\mathscr{L}'}$ on $S^d(\C^*)\setminus \Delta$. By the cartesian diagram
\[
 \xymatrix{
 \tilde{\Xi}(N'_t,\mu)\ar[r]^{\pi_{N'_t,\mu}}\ar[d]_{\tilde{\chi}_{N'_t,\mu}}&\Xi(N'_t,\mu)\ar[d]^{\chi_{N'_t,\mu}}\\
(\C^*)^d\setminus\Delta\ar[r]^{\pi_d}&S^d(\C^*)\setminus\Delta
 },
\]
see also Lemma \ref{pbcyclic}, $\pi^*_{d}\tilde{\mathscr{L}'}$ is the trivial local system.

By the cartesian diagram
\[
 \xymatrix{
 \Xi^{<1}(N'_t,\mu)\ar[r]^{h}\ar[d]_{\chi_{N'_t,\mu}}&\Xi(N,\mu)\ar[d]^{\chi_{N,\mu}}\\
 S^dD^*(0,1)\ar[r]^{i_d}&S^d\PP_1^{\hom}\setminus\Delta
 }
\]
where $i_d$ is the natural injection, and the upper horizontal arrow $h$ is the composition of the injection $\Xi^{<1}(N'_t,\mu)\rightarrow \Xi(N'_t,\mu)$ and the injection given by $\iota_E$, we have that $h^*\mathscr{L}=\mathscr{L}'$, and therefore, the hypothesis of Lemma \ref{triviallocal} is verified for $\tilde{\mathscr{L}}$ and $t\in D$. By repeating this for all $t\in D$, it shows that $\mathscr{L}$ is trivialized by $\pi_d$ at all points of $D$. By Lemma \ref{triviallocal}, $(\pi_d)^*\mathscr{L}$ is trivial. By Lemma \ref{locsysreg}, $\mathscr{L}$ is a Lusztig local system.

\section{The case of quivers with loops}
\label{quiversloops}
In this section, we show that the class of simple perverse sheaves defined by Lusztig on representation spaces of quivers is determined by a nilpotent condition on the singular support for the $g$-loop quiver. The method is completely different than that for finite type or affine quivers.

\subsection{Smallness of some morphisms for negative quivers}

Let $Q$ be quiver. For a flag-type $\underline{\dd}$, we let $E_{\underline{\dd}}$ be the image of the proper morphism $\pi_{\underline{\dd}}$ and $E_{\underline{\dd}}^{\nil}$ the image of $\pi_{\underline{\dd}}^{\nil}$. For convenience, we let $E_{\underline{\dd}}^{\flat}=E_{\underline{\dd}}$ if $\flat=1$ or $\flat=\emptyset$ and $E_{\underline{\dd}}^{\flat}=E_{\underline{\dd}}^{\nil}$ if $\flat=\nil$ or $\flat=(\nil,1)$. Technically speaking, if $\flat=1$ or $\flat=(\nil,1)$, $E_{\underline{\dd}}^{\flat}$ exists only if $\underline{\dd}$ is a discrete flag-type. These are closed irreducible subvarieties of $E_{Q,\dd}$. Define $\mathscr{C}_{\dd}$ as the set of flag-types of dimension $\dd$ and $\mathscr{C}^{1}_{\dd}$ as the set of discrete flag-types of dimension $\dd$. To ease the notation, for $\flat\in\{\nil,\emptyset,(\nil,1),1\}$, we let $\mathscr{C}^{\flat}_{\dd}=\mathscr{C}_{\dd}$ if $\flat=\nil$ or $\flat=\emptyset$ and $\mathscr{C}^{\flat}_{\dd}=\mathscr{C}^1_{\dd}$ else and the same remark as before applies. Similarly, $\pi_{\underline{\dd}}^{\flat}=\pi_{\underline{\dd}}$ if  $\flat=\emptyset$ or $\flat=1$ and $\pi_{\underline{\dd}}^{\flat}=\pi_{\underline{\dd}}^{\nil}$ if $\flat =\nil$ or $\flat=(\nil,1)$.

\subsubsection{A quadratic form}
Let $\dd\in\N^I$ be a dimension vector and $\underline{\dd}=(\dd_1,\hdots,\dd_r)$ a flag-type of dimension $\dd$ (that is, $\sum_{j=1}^r\dd_j=\dd$). In this section, we reformulate the main result of \cite{MR1261904} in the case where the flag type $\underline{\dd}$ is not necessarily discrete and for both morphisms $\pi_{\underline{\dd}}$ and $\pi_{\underline{\dd}}^{\nil}$ (Lusztig considered the case of $\pi_{\underline{\dd}}$). Recall that $\mathcal{F}_{\underline{\dd}}$ denotes the flag variety of flags of type $\underline{\dd}$ inside $\C^{\dd}$. As in \cite{MR1074310} in the non-graded case or in \cite{MR1261904} for quivers and discrete flag-types, the relative position of two flags $\underline{F}$ and $\underline{F'}$ of type $\underline{\dd}$ is encoded in an array $z=(z^{pq})_{\substack{1\leq p\leq r\\1\leq q\leq r}}$, such that for any $1\leq s\leq r$,
\begin{equation}
\label{fteq}
\sum_{1\leq p\leq r}z^{ps}=\sum_{1\leq q\leq r}z^{sq}=\dd_s
\end{equation}
For any $1\leq p,q\leq r$, $z^{pq}=(z_i^{pq})_{i\in I}\in\N^I$ is defined by
\[
 z^{pq}=\dim \left(\frac{F_{p-1}+F_p\cap F'_q}{F_{p-1}+F_p\cap F'_{q-1}}\right).
\]
By splitting the refined flag
\[
 (F_{p-1}+F_p\cap F'_q)_{1\leq p,q\leq r}
\]
(where tuples $(p,q)$ are lexicographically ordered: $(p,q)\leq (p',q')\iff p<p' \text{ or }(p=p' \text{ and }q\leq q'$)), we see that $\underline{F}$ and $\underline{F'}$ are in relative position $z$ for some $z$ as above if and only there exists vector spaces $V^{pq}$ such that $\dim V^{pq}=z^{pq}$, $\C^{\dd}=\bigoplus_{1\leq p,q\leq r}V^{pq}$ and for any $1\leq s\leq r$,
\[
 F_s=\bigoplus_{1\leq q\leq r}V^{sq},\quad F'_s=\bigoplus_{1\leq p\leq r}V^{ps}.
\]
Of course, all this applies for $I$-graded vector spaces (where $I$ will denote the vertices of the quiver $Q$). We let $\Theta(\underline{\dd})$ be the set of all possible relative positions between two flags of type $\underline{\dd}$, that is the set of arrays $z=(z^{pq})_{1\leq p,q\leq r}$ satisfying the equalities \eqref{fteq}. The variety $\mathcal{F}_{\underline{\dd}}\times \mathcal{F}_{\underline{\dd}}$ is then stratified by the relative position:
\[
 \mathcal{F}_{\underline{\dd}}\times \mathcal{F}_{\underline{\dd}}=\bigsqcup_{z\in \Theta(\underline{\dd})}(\mathcal{F}_{\underline{\dd}}\times \mathcal{F}_{\underline{\dd}})_z
\]
where $(\mathcal{F}_{\underline{\dd}}\times \mathcal{F}_{\underline{\dd}})_z$ denotes the locally closed subvariety of pairs $(\underline{F},\underline{F'})$ in relative position $z$. Recall the proper morphism $\pi_{\underline{\dd}}^{\sharp}:\tilde{\mathcal{F}}_{\underline{\dd}}\rightarrow E_{\dd}$ for $\sharp\in\{\emptyset,\nil\}$. We have a morphism $\varphi_{\underline{\dd}}:\tilde{\mathcal{F}}^{\sharp}_{\underline{\dd}}\times_{E_{\dd}}\tilde{\mathcal{F}}^{\sharp}_{\underline{\dd}}\rightarrow\mathcal{F}_{\underline{\dd}}\times \mathcal{F}_{\underline{\dd}}$. Therefore, we can stratify $\tilde{\mathcal{F}}^{\sharp}_{\underline{\dd}}\times_{E_{\dd}}\tilde{\mathcal{F}}^{\sharp}_{\underline{\dd}}$ by the pull-back of the stratification of $\mathcal{F}_{\underline{\dd}}\times \mathcal{F}_{\underline{\dd}}$:
\[
 \tilde{\mathcal{F}}^{\sharp}_{\underline{\dd}}\times_{E_{\dd}}\tilde{\mathcal{F}}^{\sharp}_{\underline{\dd}}=\bigsqcup_{z\in\Theta(\underline{\dd})}(\tilde{\mathcal{F}}^{\sharp}_{\underline{\dd}}\times_{E_{\dd}}\tilde{\mathcal{F}}^{\sharp}_{\underline{\dd}})_z.
\]
In first lecture, the reader can directly jump to the statement of Corollary \ref{formulasmall} and then to Section \ref{smallnegquiv}.

\begin{lemma}
\label{dimbase}
 The stratum $(\mathcal{F}_{\underline{\dd}}\times\mathcal{F}_{\underline{\dd}})_z$ is a connected $G_{\dd}$-homogeneous space of dimension
 \[
 \dim G_{\dd}-\sum_{\substack{i\in I\\1\leq t\leq p\leq r\\1\leq s\leq q\leq r}}z_{i}^{pq}z_i^{ts}= \sum_{\substack{i\in I\\1\leq t,p\leq r}}(\dd_p)_i(\dd_t)_i-\sum_{\substack{i\in I\\1\leq t\leq p\leq r\\1\leq s\leq q\leq r}}z_{i}^{pq}z_i^{ts}.
 \]

\end{lemma}
\begin{proof}
 See \cite[\S 17.]{MR1261904}.
\end{proof}

\begin{lemma}
\label{dimfiber}
 The morphism $\varphi_{\underline{\dd}}$ restricts to a fiber bundle $(\tilde{\mathcal{F}}^{\sharp}_{\underline{\dd}}\times_{E_{\dd}}\tilde{\mathcal{F}}^{\sharp}_{\underline{\dd}})_z\rightarrow (\mathcal{F}_{\underline{\dd}}\times\mathcal{F}_{\underline{\dd}})_z$ with fibers of dimension
 \[
  \sum_{\alpha:i\rightarrow j\in\Omega}\sum_{\substack{1\leq t\leq p\leq r\\1\leq s\leq q\leq r}}z_i^{pq}z_{j}^{ts}
 \]
if $\sharp=\emptyset$ and
\[
  \sum_{\alpha:i\rightarrow j\in\Omega}\sum_{\substack{1\leq t< p\leq r\\1\leq s< q\leq r}}z_i^{pq}z_{j}^{ts}
 \]
 if $\sharp=\nil$.
\end{lemma}
Note that for $z_{\underline{\dd}}=(\dd_p\delta_{pq})_{1\leq p,q\leq r}$, $(\tilde{\mathcal{F}}^{\sharp}_{\underline{\dd}}\times_{E_{\dd}}\tilde{\mathcal{F}}^{\sharp}_{\underline{\dd}})_{z_{\underline{\dd}}}\simeq \tilde{\mathcal{F}}^{\sharp}_{\underline{\dd}}$ is the diagonal copy of $\tilde{\mathcal{F}}^{\sharp}_{\underline{\dd}}$ inside $\tilde{\mathcal{F}}^{\sharp}_{\underline{\dd}}\times_{E_{\dd}}\tilde{\mathcal{F}}^{\sharp}_{\underline{\dd}}$.
\begin{proof}
  See \cite[\S 17.]{MR1261904}. Note that for $\sharp=\nil$, the inequalities in the sum are strict as a consequence of the nilpotency condition.
\end{proof}

\begin{lemma}
\label{techlemmaeq}
 Let $z=(z^{pq})_{1\leq p,q\leq r}$ be a matrix such that for any $1\leq s\leq r$,
 \[
  \sum_{1\leq p\leq r}z^{ps}=\sum_{1\leq q\leq r}z^{sq}.
 \]
Then,
\begin{equation}
\label{firsteq}
 \sum_{\substack{1\leq t<p\leq r\\1\leq q\leq s\leq r}}z^{pq}z^{ts}=\sum_{\substack{1\leq t\leq p\leq r\\1\leq q<s\leq r}}z^{pq}z^{ts}
\end{equation}
or, equivalently,
\begin{equation}
\label{secondeq}
 \sum_{\substack{1\leq t<p\leq r\\1\leq q\leq r }}z^{pq}z^{tq}=\sum_{\substack{1\leq p\leq r\\1\leq q<s\leq r}}z^{pq}z^{ps}.
\end{equation}
\end{lemma}
\begin{proof}
 The first equality \eqref{firsteq} can be deduced from the second one \eqref{secondeq} by adding on both side $\sum_{\substack{1\leq t<p\leq r\\1\leq q< s\leq r}}z^{pq}z^{ts}$. To prove \eqref{secondeq}, note that $z^{pq}z^{tq}$ is symmetric in $p,t$. Let $L= \sum_{\substack{1\leq t<p\leq r\\1\leq q\leq r }}z^{pq}z^{tq}$ be the left-hand side of \eqref{secondeq} and $R=\sum_{\substack{1\leq p\leq r\\1\leq q<s\leq r}}z^{pq}z^{ps}$ be its right-hand side. Then
 \[
  \begin{aligned}
   2L+\sum_{\substack{1\leq p\leq r\\1\leq q\leq r}}(z^{pq})^2
   &=\sum_{1\leq p,q,t\leq r}z^{pq}z^{tq}\\
   &=\sum_{1\leq q\leq r}\left(\sum_{1\leq p\leq r}z^{pq}\right)\left(\sum_{1\leq t\leq r}z^{tq}\right)\\
   &=\sum_{1\leq q\leq r}\left(\sum_{1\leq p\leq r}z^{qp}\right)\left(\sum_{1\leq t\leq r}z^{qt}\right)\\
   &=\sum_{1\leq p,q,t\leq r}z^{qp}z^{qt}
  \end{aligned}
 \]
and,
\[
 2R+\sum_{\substack{1\leq p \leq r\\1\leq q\leq r}}(z^{pq})^2=\sum_{1\leq p,q,s\leq r}z^{pq}z^{ps}
\]
so that $L=R$.

\end{proof}

\begin{cor}
\label{formulasmall}
 We have the following formulas:
 \[
  \dim\tilde{\mathcal{F}}_{\underline{\dd}}-\dim(\tilde{\mathcal{F}}_{\underline{\dd}}\times_{E_{\dd}}\tilde{\mathcal{F}}_{\underline{\dd}})_z=\sum_{\alpha:i\rightarrow j\in\Omega}
\sum_{\substack{1\leq t\leq p\leq r\\1\leq q<s\leq r}}z_i^{pq}z_j^{ts}-\sum_{i\in I}\sum_{\substack{1\leq t\leq p\leq r\\1\leq q<s\leq r}}z_i^{pq}z_i^{ts}
\]
and
\[
  \dim\tilde{\mathcal{F}}^{\nil}_{\underline{\dd}}-\dim(\tilde{\mathcal{F}}^{\nil}_{\underline{\dd}}\times_{E_{\dd}}\tilde{\mathcal{F}}^{\nil}_{\underline{\dd}})_z=\sum_{\alpha:i\rightarrow j\in\Omega}
\sum_{\substack{1\leq t< p\leq r\\1\leq q\leq s\leq r}}z_i^{pq}z_j^{ts}-\sum_{i\in I}\sum_{\substack{1\leq t< p\leq r\\1\leq q\leq s\leq r}}z_i^{pq}z_i^{ts}
\]
\end{cor}
\begin{proof}
 By Lemmas \ref{dimbase} and \ref{dimfiber},
 \begin{equation}
 \label{eq1}
  \dim\tilde{\mathcal{F}}_{\underline{\dd}}=\sum_{\alpha:i\rightarrow j\in\Omega}\sum_{1\leq t\leq p\leq r}(\dd_p)_i(\dd_t)_j+\sum_{\substack{i\in I\\1\leq t,p\leq r}}(\dd_p)_i(\dd_t)_i-\sum_{\substack{i\in I\\1\leq t\leq p\leq r}}(\dd_p)_i(\dd_t)_i
 \end{equation}
 and
 \begin{equation}
 \label{eq2}
 \dim(\tilde{\mathcal{F}}_{\underline{\dd}}\times_{E_{\dd}}\tilde{\mathcal{F}}_{\underline{\dd}})_z=\sum_{\alpha:i\rightarrow j\in\Omega}\sum_{\substack{1\leq t\leq p\leq r\\1\leq s\leq q\leq r}}z_i^{pq}z_{j}^{ts}+\sum_{\substack{i\in I\\1\leq t,p\leq r}}(\dd_p)_i(\dd_t)_i-\sum_{\substack{i\in I\\1\leq t\leq p\leq r\\1\leq s\leq q\leq r}}z_{i}^{pq}z_i^{ts}.
\end{equation}
Note that
\begin{equation}
\label{eq3}
 \sum_{\substack{1\leq t\leq p\leq r\\1\leq s\leq q\leq r}}z_i^{pq}z_{j}^{ts}+\sum_{\substack{1\leq t\leq p\leq r\\1\leq q<s\leq r}}z_i^{pq}z_{j}^{ts}=\sum_{\substack{1\leq t\leq p\leq r\\1\leq q,s\leq r}}z^{pq}_iz^{ts}_j=\sum_{1\leq t\leq p\leq r}(\dd_p)_i(\dd_t)_j.
\end{equation}
Consequently, by substracting \eqref{eq1} and \eqref{eq2}, using \eqref{eq3}, we obtain the first formula of the corollary.

For the second formula, again by Lemmas \ref{dimbase} and \ref{dimfiber},
\begin{equation}
 \label{eq4}
  \dim\tilde{\mathcal{F}}^{\nil}_{\underline{\dd}}=\sum_{\alpha:i\rightarrow j\in\Omega}\sum_{1\leq t< p\leq r}(\dd_p)_i(\dd_t)_j+\sum_{\substack{i\in I\\1\leq t,p\leq r}}(\dd_p)_i(\dd_t)_i-\sum_{\substack{i\in I\\1\leq t\leq p\leq r}}(\dd_p)_i(\dd_t)_i
 \end{equation}
 and
 \begin{equation}
 \label{eq5}
 \dim(\tilde{\mathcal{F}}^{\nil}_{\underline{\dd}}\times_{E_{\dd}}\tilde{\mathcal{F}}^{\nil}_{\underline{\dd}})_z=\sum_{\alpha:i\rightarrow j\in\Omega}\sum_{\substack{1\leq t< p\leq r\\1\leq s< q\leq r}}z_i^{pq}z_{j}^{ts}+\sum_{\substack{i\in I\\1\leq t,p\leq r}}(\dd_p)_i(\dd_t)_i-\sum_{\substack{i\in I\\1\leq t\leq p\leq r\\1\leq s\leq q\leq r}}z_{i}^{pq}z_i^{ts}.
\end{equation}
the only difference with \eqref{eq1} and \eqref{eq2} is the strict inequalities in the range of summation of the first sums of \eqref{eq4} and \eqref{eq5}. Therefore, by substrating \eqref{eq4} and \eqref{eq5},
\[
\begin{aligned}
  \dim\tilde{\mathcal{F}}^{\nil}_{\underline{\dd}}-\dim(\tilde{\mathcal{F}}^{\nil}_{\underline{\dd}}\times_{E_{\dd}}\tilde{\mathcal{F}}^{\nil}_{\underline{\dd}})_z
  &=\sum_{\alpha:i\rightarrow j\in\Omega}\sum_{\substack{1\leq t< p\leq r\\1\leq q\leq s\leq r}}z_i^{pq}z_j^{ts}-\sum_{\substack{i\in I\\1\leq t\leq p\leq r\\1\leq q<s\leq r}}z_i^{pq}z_i^{ts}\\
&=\sum_{\alpha:i\rightarrow j\in\Omega}\sum_{\substack{1\leq t< p\leq r\\1\leq q\leq s\leq r}}z_i^{pq}z_j^{ts}-\sum_{\substack{i\in I\\1\leq t< p\leq r\\1\leq q\leq s\leq r}}z_i^{pq}z_i^{ts}+\underbrace{\sum_{\substack{i\in I\\1\leq t< p\leq r\\1\leq q\leq s\leq r}}z_i^{pq}z_i^{ts}-\sum_{\substack{i\in I\\1\leq t\leq p\leq r\\1\leq q<s\leq r}}z_i^{pq}z_i^{ts}}_{=0 \text{ by Lemma \ref{techlemmaeq}}}\\
&=\sum_{\alpha:i\rightarrow j\in\Omega}\sum_{\substack{1\leq t< p\leq r\\1\leq q\leq s\leq r}}z_i^{pq}z_j^{ts}-\sum_{\substack{i\in I\\1\leq t< p\leq r\\1\leq q\leq s\leq r}}z_i^{ps}z_i^{ts}
  \end{aligned}
\]
where the last equality follows from Lemma \ref{techlemmaeq}. We are done.
\end{proof}

\subsubsection{Smallness for negative quivers}
\label{smallnegquiv}
A quiver is called \emph{negative} provided it carries at least two loops at each vertex.
\begin{proposition}
 \label{smallnessmor}
 Let $Q$ be a negative quiver. For any flag-type $\underline{\dd}$ of dimension $\dd$, the morphisms $\pi_{\underline{\dd}}$ and $\pi_{\underline{\dd}}^{\nil}$ are small and are an isomorphism over a non-empty open subset of their image.
\end{proposition}
\begin{proof}
 Write $\underline{\dd}=(\dd_1,\hdots,\dd_r)$. For the smallness of $\pi_{\underline{\dd}}$, we use Corollary \ref{formulasmall}. We have to prove that for any $z\in\Theta(\underline{\dd})\setminus\{ z_{\underline{\dd}}\}$,
 \[
  \dim\tilde{\mathcal{F}}_{\underline{\dd}}-\dim(\tilde{\mathcal{F}}_{\underline{\dd}}\times_{E_{\dd}}\tilde{\mathcal{F}}_{\underline{\dd}})_z>0.
 \]
For any vertex $i\in I$, we let $g_i\geq 2$ be the number of loops of $Q$ at $i$. Let $z\in \Theta(\underline{\dd})\setminus \{z_{\underline{\dd}}\}$. Then there exists $p'\neq q'$ and $i\in I$ such that $z^{p'q'}_i\neq 0$. Choose $p'$ minimal, so that for any $1\leq p< p'$ and $q\neq p$, $z^{pq}=0$. We therefore have $z^{pp}=\dd_p$ for any $1\leq p<p'$. Since we have $\sum_{p=1}^rz^{pq}=\dd_q$, for $1\leq q<p'$, $z^{qq}+\sum_{1\leq p\leq r, p\neq q}z^{pq}=\dd_q$. Consequently, $z^{pq}=0$ for any $1\leq q<p'$, $1\leq p\leq r$ and $p\neq q$. We have therefore $q'>p'$. We have
\[
\begin{aligned}
 \dim\tilde{\mathcal{F}}_{\underline{\dd}}-\dim(\tilde{\mathcal{F}}_{\underline{\dd}}\times_{E_{\dd}}\tilde{\mathcal{F}}_{\underline{\dd}})_z
 &\geq \sum_{i\in I}(g_i-1)\sum_{\substack{t,p,q,s\\1\leq t\leq p\leq r\\1\leq q<s\leq r}}z_i^{pq}z_i^{ts}\text{ (forgetting arrows which are not loops)}\\
 &\geq (g_i-1)\sum_{\substack{p,s\\p'\leq p\leq r\\p'<s\leq r}}z_i^{pp'}z_i^{p's}\\
 &\geq (g_i-1)\sum_{\substack{p,\\ p'\leq p\leq r}}z_i^{pp'}z_i^{p'q'}\text{ (specializing to $t=q=p'$)}\\
 &=(g_i-1)(\dd_{p'})_iz_i^{p'q'}\text{ (summing over $p$)}\\
 &>0.
\end{aligned}
\]
The same reasoning shows that
\[
 \dim\tilde{\mathcal{F}}^{\nil}_{\underline{\dd}}-\dim(\tilde{\mathcal{F}}^{\nil}_{\underline{\dd}}\times_{E_{\dd}}\tilde{\mathcal{F}}^{\nil}_{\underline{\dd}})_z>0
\]
for any $z\in\Theta(\underline{\dd})\setminus\{z_{\underline{\dd}}\}$ and hence that $\pi_{\underline{\dd}}^{\nil}$ is small.

To prove that $\pi_{\underline{\dd}}$ (resp. $\pi_{\underline{\dd}}^{\nil}$) is an isomorphism over an open subset of its image $E_{\underline{\dd}}$ (resp. $E_{\underline{\dd}}^{\nil}$), it suffices to show that a sufficiently general element $x\in E_{\underline{\dd}}$ (resp. $x\in E_{\underline{\dd}}^{\nil}$) admits a unique filtration of type $\underline{\dd}$. It suffices to do this when $Q=S_g$ is a quiver with one vertex and $g\geq 2$ loops. Since $\pi_{\underline{\dd}}^{\nil}$ is small, its fiber over a general element of $E_{S_g,\underline{\dd}}^{\nil}$ is finite, but this is not sufficient. However, in type $A$, the generalized Springer resolutions $T^*(G/P)\rightarrow \mathfrak{g}$ are always birational onto their image (\cite[\S 2.7]{MR679767}). It shows that $\pi_{\underline{\dd}}^{\nil}$ is also birational onto its image.  For $\pi_{\underline{\dd}}$, a general element of $E_{S_g,\underline{\dd}}$ is of the form $(x_1,\hdots,x_g)$ with $x_1$ regular semisimple. Hence, $x_1$ stabilizes a finite number of flags of type $\underline{\dd}$ and choosing $x_2$ general enough, there will be a unique flag of type $\underline{\dd}$ stabilized by $(x_1,\hdots,x_g)$.
 
\end{proof}

\subsection{Lusztig sheaves}
\label{lusztigsheavesloops}
Let $\dd\in\N^I$. We define \emph{four} classes of $G_{\dd}$-equivariant perverse sheaves on the representation spaces $E_{Q,\dd}$ of an arbitrary quiver.
\begin{enumerate}
 \item $\mathcal{P}_{Q,\dd}^{\nil}$ is the semisimple category of $\Perv_{G_{\dd}}(E_{Q,\dd})$ generated by perverse sheaves appearing (with a possible shift) as a direct summand of $(\pi^{\nil}_{Q,\underline{\dd}})_*\underline{\C}$ where $\underline{\dd}$ is some flag type of dimension $\dd$,
  \item $\mathcal{P}_{Q,\dd}$ is the semisimple category of $\Perv_{G_{\dd}}(E_{Q,\dd})$ generated by perverse sheaves appearing (with a possible shift) as a direct summand of $(\pi_{Q,\underline{\dd}})_*\underline{\C}$ where $\underline{\dd}$ is some flag type of dimension $\dd$,
   \item $\mathcal{P}_{Q,\dd}^{\nil,1}$ is the semisimple category of $\Perv_{G_{\dd}}(E_{Q,\dd})$ generated by perverse sheaves appearing (with a possible shift) as a direct summand of $(\pi^{\nil}_{Q,\underline{\dd}})_*\underline{\C}$ where $\underline{\dd}$ is some \emph{discrete} flag type of dimension $\dd$,
    \item $\mathcal{P}_{Q,\dd}^{1}$ is the semisimple category of $\Perv_{G_{\dd}}(E_{Q,\dd})$ generated by perverse sheaves appearing (with a possible shift) as a direct summand of $(\pi_{Q,\underline{\dd}})_*\underline{\C}$ where $\underline{\dd}$ is some \emph{discrete} flag type of dimension $\dd$.
\end{enumerate}
All sheaves in this categories will be called \emph{Lusztig sheaves}, as perverse sheaves defined in this way were first considered by Lusztig.
\begin{remark}\label{relationsperverse}
We have the following inclusions between these categories:
\[
 \mathcal{P}_{Q,\dd}^1\subset \mathcal{P}_{Q,\dd},\quad \mathcal{P}_{Q,\dd}^{\nil,1}\subset \mathcal{P}_{Q,\dd}^{\nil}.
\]
Moreover, if $Q$ is acyclic, all these categories coincide:
\[
 \mathcal{P}_{Q,\dd}^{\nil}=\mathcal{P}_{Q,\dd}=\mathcal{P}_{Q,\dd}^{\nil,1}=\mathcal{P}_{Q,\dd}^{1},
\]
and if $Q$ has no cycles apart from loops, the categories in $(1)$ and $(3)$ one the one side and $(2)$ and $(4)$ on the other side coincide:
\[
 \mathcal{P}_{Q,\dd}^{\nil}=\mathcal{P}_{Q,\dd}^{\nil,1},\quad \mathcal{P}_{Q,\dd}=\mathcal{P}_{Q,\dd}^{1}.
\]
If $Q$ has no loops, the categories $(3)$ and $(4)$ coincide: $\mathcal{P}_{Q,\dd}^{\nil,1}=\mathcal{P}_{Q,\dd}^1$.
\end{remark}

Let $\underline{\dd}$ be a flag-type of dimension $\dd$. Define
\[
 L_{Q,\underline{\dd}}^{\nil}=(\pi_{Q,\underline{\dd}}^{\nil})_*\underline{\C},
 \quad
 L_{Q,\underline{\dd}}=(\pi_{Q,\underline{\dd}})_*\underline{\C}.
 \]
These sheaves are related thanks to the Fourier-Sato transform with the corresponding sheaves on the representation space of the opposite quiver.
\begin{lemma}
\label{tfourier}
 Let $\Phi:D^b_{G_{\dd}}(E_{Q,\dd})\rightarrow D^b_{G_{\dd}}(E_{Q^{\op},\dd})$ be the Fourier-Sato transform reversing all arrows of $Q$. Then,
 \[
  \Phi(L_{Q,\underline{\dd}})=L_{Q^{\op},\underline{\dd}}^{\nil},\quad \Phi(L_{Q,\underline{\dd}}^{\nil})=L_{Q^{\op},\underline{\dd}}.
 \]
\end{lemma}
\begin{proof}
  It is completely similar to the proof of Lemma 2.2 of \cite{MR3218317}.
 
\end{proof}

\subsubsection{Explicit description of Lusztig sheaves for negative quivers}

\begin{proposition}\label{explicitdesc}
Let $Q$ be a negative quiver. Then for $\flat\in\{\nil,\emptyset,(1,\nil),1\}$, the simple objects of the category $\mathcal{P}_{\dd}^{\flat}$ are the perverse sheaves $\ICC(E_{\underline{\dd}}^{\flat})$ for $\underline{\dd}\in\mathscr{C}_{\dd}^{\flat}$.
\end{proposition}
\begin{proof}
 This is an immediate consequence of Proposition \ref{smallnessmor}.
\end{proof}

\subsection{Lusztig nilpotent varieties and singular support of Lusztig sheaves}

\subsubsection{The notions of nilpotency for representations of the double quiver}
\label{nilpotency}
Let $Q$ be a quiver and $\overline{Q}$ the doubled quiver (that is, for any $\alpha:i\rightarrow j\in\Omega$, add an arrow $\alpha^*:j\rightarrow i$). Recall that a representation of $\overline{Q}$ is denoted by $\bar{x}=(x,x^*)$ where $x=(x_{\alpha})_{\alpha\in\Omega}$ and $x^*=(x^*_{\alpha})_{\alpha\in\Omega}$. Following \cite{bozec2017number}, a representation $(x,x^*)$ is called
\begin{enumerate}
 \item semi-nilpotent if there exists a flag of $I$-graded vector spaces $(0\subset F_1\subset\hdots\subset F_r=\C^{\dd})$ such that for any $\alpha\in\Omega$, $1\leq j\leq r$,
 \[
  x_{\alpha}F_j\subset F_{j-1},\quad x_{\alpha}^*F_j\subset F_{j},
 \]
 \item $*$-semi-nilpotent if there exists a flag of $I$-graded vector spaces $(0\subset F_1\subset\hdots\subset F_r=\C^{\dd})$ such that for any $\alpha\in\Omega$, $1\leq j\leq r$,
 \[
  x_{\alpha}F_j\subset F_{j},\quad x_{\alpha}^*F_j\subset F_{j-1},
 \]
  \item strongly semi-nilpotent if there exists a \emph{discrete} flag of $I$-graded vector spaces $(0\subset F_1\subset\hdots\subset F_r=\C^{\dd})$ such that for any $\alpha\in\Omega$, $1\leq j\leq r$,
 \[
  x_{\alpha}F_j\subset F_{j-1},\quad x_{\alpha}^*F_j\subset F_{j},
 \]
  \item $*$-strongly semi-nilpotent if there exists a \emph{discrete} flag of $I$-graded vector spaces $(0\subset F_1\subset\hdots\subset F_r=\C^{\dd})$ such that for any $\alpha\in\Omega$, $1\leq j\leq r$,
 \[
  x_{\alpha}F_j\subset F_{j},\quad x_{\alpha}^*F_j\subset F_{j-1}.
 \]
\end{enumerate}
\begin{remark}
\label{interactionnilp}
 These notions of nilpotency interact as follows. For a representation $\bar{x}=(x,x^*)$ of $\overline{Q}$,
 \[
  (x,x^*) \text{ strongly semi-nilpotent}\implies (x,x^*)\text{  semi-nilpotent}
 \]
and
 \[
  (x,x^*) \text{ is $*$-strongly semi-nilpotent}\implies (x,x^*)\text{ is $*$-semi-nilpotent}.
 \]
 Moreover, if $Q$ is acyclic, all these notions of nilpotency coincide, and if $Q$ has no cycles apart form loops, being strongly semi-nilpotent is the same as being semi-nilpotent and being $*$-strongly semi-nilpotent is the same as being $*$-semi-nilpotent. Lastly, if $Q$ has no loops, being strongly semi-nilpotent is the same as being $*$-strongly semi-nilpotent.
\end{remark}

\subsubsection{The nilpotent varieties}
Related to the four categories of perverse sheaves defined in Section \ref{lusztigsheavesloops} and to the four notions of nilpotency defined in Section \ref{nilpotency}, we define four different \emph{nilpotent varieties} in the representation space of the doubled quiver $\overline{Q}$. For this purpose, recall the moment map
\[
\begin{matrix}
 \mu_{\dd}&:&E_{\overline{Q},\dd}&\rightarrow&\mathfrak{gl}_{\dd}\\
 &&(x,x^*)&\mapsto&\sum_{\alpha\in\Omega}[x_{\alpha},x^*_{\alpha}].
 \end{matrix}
\]
The nilpotent varieties are defined as follows.
\begin{enumerate}
 \item $\Lambda_{Q,\dd}^{\nil}=\{(x,x^*)\in\mu_{\dd}^{-1}(0)\mid (x,x^*) \text{ is semi-nilpotent}\}$,
 \item $\Lambda_{Q,\dd}=\{(x,x^*)\in\mu_{\dd}^{-1}(0)\mid (x,x^*) \text{ is $*$-semi-nilpotent}\}$,
 \item $\Lambda_{Q,\dd}^{\nil,1}=\{(x,x^*)\in\mu_{\dd}^{-1}(0)\mid (x,x^*) \text{ is strongly semi-nilpotent}\}$,
 \item $\Lambda_{Q,\dd}^{1}=\{(x,x^*)\in\mu_{\dd}^{-1}(0)\mid (x,x^*) \text{ is $*$-strongly semi-nilpotent}\}$. 
\end{enumerate}
\begin{remark}
\label{interactionsnilvar}
Using Remark \ref{interactionnilp}, we have the following inclusions between the nilpotent varieties. For a general quiver $Q$,
\[
 \Lambda_{Q,\dd}^{1}\subset \Lambda_{Q,\dd},\quad\Lambda_{Q,\dd}^{\nil,1}\subset \Lambda_{Q,\dd}^{\nil},
\]
if $Q$ is acyclic, all the nilpotent varieties coincide and if $Q$ has no cycles apart from loops,
\[
 \Lambda_{Q,\dd}^{1}= \Lambda_{Q,\dd},\quad\Lambda_{Q,\dd}^{\nil,1}= \Lambda_{Q,\dd}^{\nil}.
\]
Lastly, if $Q$ has no loops,
\[
 \Lambda_{Q,\dd}^{\nil,1}=\Lambda_{Q,\dd}^{1}.
\]

\end{remark}

\begin{proposition}
 The varieties $\Lambda_{\dd}^{\flat}$, $\flat\in\{\nil,\emptyset,(\nil,1),1\}$, are closed, Lagrangian, conical subvarieties of $E_{\overline{Q},\dd}\simeq T^*E_{Q,\dd}$.
\end{proposition}
\begin{proof}
 That these varieties are closed and conical is immediate from their definitions. That they are Lagrangian is already mentioned in \cite[\S 1.1]{bozec2017number}.
\end{proof}

\begin{remark}
\label{eqnilvar}
 We have directly from the definitions and the natural identifications $T^*E_{Q,\dd}\simeq E_{\overline{Q},\dd}=E_{\overline{Q^{\op}},\dd}\simeq T^*E_{Q^{\op},\dd}$ the equalities
 \[
  \Lambda_{Q^{\op},\dd}^{\nil}=\Lambda_{Q,\dd}\quad\text{and}\quad\Lambda_{Q^{\op},\dd}^{\nil,1}=\Lambda_{Q,\dd}^{1}.
 \]
\end{remark}

\subsubsection{The singular support of Lusztig sheaves}
\begin{proposition}
\label{unionirrcomp}
 The singular support of sheaves of the category $\mathcal{P}^{\flat}$ ($\flat\in\{\nil,\emptyset,(\nil,1),1\}$) is a union of irreducible components of $\Lambda_{\dd}^{\flat}$.
\end{proposition}
\begin{proof}
 It suffices to show that the singular support of $(\pi_{\underline{\dd}}^{\flat})_*\underline{\C}$ is a subvariety of $\Lambda_{\dd}^{\flat}$. This is a standard argument similar to the one of \cite[Corollary 13.6]{MR1088333}.
\end{proof}

\subsubsection{Explicit description of the nilpotent varieties}
 Recall that for a closed subvariety $Z\subset X$ inside a smooth variety $X$, we let $T^*_ZX=\overline{T^*_UX}$ where $U$ is some smooth and dense open subset of $Z$.

\begin{proposition}
\label{inclusionnil}
 For any $\dd\in\N^I$ and $\flat\in\{\nil,\emptyset,(\nil,1),1\}$, we have the inclusion
 \[
  \bigcup_{\underline{\dd}\in\mathscr{C}_{\dd}^{\flat}}T^*_{E_{\underline{\dd}}^{\flat}}E_{\dd}\subset \Lambda_{\dd}^{\flat}.
 \]
\end{proposition}
\begin{proof}
 By Proposition \ref{explicitdesc}, for any $\underline{\dd}\in\mathscr{C}_{\dd}^{\flat}$, $\ICC(E_{\underline{\dd}}^{\flat})\in\mathcal{P}_{\dd}^{\flat}$; by Proposition \ref{unionirrcomp}, $SS(\ICC(E_{\underline{\dd}}^{\flat}))\subset \Lambda_{\dd}^{\flat}$; and $T^*_{E_{\underline{\dd}}^{\flat}}E_{\dd}\subset SS(\ICC(E_{\underline{\dd}}^{\flat}))$. This proves the inclusion of the lemma.
\end{proof}
\begin{conj}
\label{conjequa}
 If $Q$ is a negative quiver, then the inclusion of Proposition \ref{inclusionnil} is an equality. 
\end{conj}

\begin{proposition}
\label{truenil}
 Conjecture \ref{conjequa} is true if $Q$ is a negative quiver and $\flat=1$ or $\flat=(\nil,1)$.
\end{proposition}
\begin{proof}
 By Remark \ref{eqnilvar}, it suffices to prove Conjecture \ref{conjequa} for $\flat=(\nil,1)$. In this case, Bozec showed that the set of isomorphism classes of simple perverse sheaves in the category $\mathcal{P}_{\dd}^{\nil,1}$ is in bijection with the irreducible components of $\Lambda_{\dd}^{\nil,1}$ (see for example \cite[Theorem 3.13]{MR3569998}). Since by Proposition \ref{inclusionnil} each of the $T^*_{E_{\underline{\dd}}^{\nil,1}}E_{\dd}$ is an irreducible component of $\Lambda_{\dd}^{\nil,1}$, this proves Proposition \ref{truenil}.
\end{proof}

\subsection{Microlocal characterization of Lusztig sheaves}

\subsubsection{Conjectures}
\begin{conj}
\label{conjec}
 Let $Q$ be a quiver. Any irreducible $G_{\dd}$-equivariant perverse sheaf on $E_{\dd}$ whose singular support is contained in $\Lambda_{\dd}^{\flat}$ ($\flat\in\{\nil,\emptyset,(\nil,1),1\}$ is in the category $\mathcal{P}_{\dd}^{\flat}$.
\end{conj}

\begin{remark}
\label{equivconj}
 By Lemma \ref{tfourier}, Remark \ref{eqnilvar} and Lemma \ref{ssfourier}, this conjecture is true for $\flat=\nil$ if and only if it is true for $\flat=\emptyset$ and similarly, it is true for $\flat=(\nil,1)$ if and only if $\flat=1$.
\end{remark}
\begin{remark}[State of Conjecture \ref{conjec}]
 \begin{enumerate}
  \item By Theorem \ref{maintheorem}, Remark \ref{relationsperverse} and Remark \ref{interactionsnilvar}, Conjecture \ref{conjec} is true for any finite type or affine acyclic quiver.
  \item By Theorem \ref{maintheorem}, it is true for cyclic quivers for $\flat=(\nil,1)$ and $\flat=1$ (these two cases coincide). By Theorem \ref{miccharec}, Conjecture \ref{conjec} is true for cyclic quivers and $\flat=\emptyset$. By Remark \ref{equivconj}, it is also true for cyclic quivers and $\flat=\nil$.
  \item For the Jordan quiver, Conjecture \ref{conjec} is true, as a consequence of Springer theory for $\mathfrak{gl}_n$ (for any $n$). We briefly explain it in Section \ref{jordanquiver} below.
  \item In Section \ref{negativequivers}, we prove Conjecture \ref{conjec} for $g$-loops quivers with $g\geq 2$.
  \item The conjecture is still open for wild quivers which are not a $g$-loop quiver.
 \end{enumerate}
\end{remark}

\subsubsection{The situation for the Jordan quiver}\label{jordanquiver}
For the Jordan quiver, Conjecture \ref{equivconj} is true. Let $d\in \N$. We have
\[
 \Lambda_{d}=\{(x,x^*)\in \mathfrak{gl}_d^2\mid [x,x^*]=0 \text{ and $x^*$ is nilpotent}\}.
\]
There are two projections $\pi_j:\Lambda_d\rightarrow\mathfrak{gl}_d$, $j=1,2$. We identify $T^*\mathfrak{gl}_d$ with $E_{\overline{Q},d}$ using the trace pairing. From the point view of $\pi_2$,
\[
 \Lambda_d=\bigsqcup_{\OO\subset \mathfrak{gl}_d}T^*_{\OO}\mathfrak{gl}_d
\]
where the sum runs over nilpotent orbits $\OO\subset \mathfrak{gl}_d$.

From the point of view of $\pi_1$,
\[
 \Lambda_{d}=\bigsqcup_{\mu}\overline{T^*_{\Xi(\mu)}\mathfrak{gl}_d}
\]
where $\mu$ is regular (see Section \ref{stratjordan} for the definition of the strata $\Xi(\mu)$).
Moreover, the Fourier transform gives a bijection between intersection cohomology complexes of nilpotent orbits and perverse sheaves appearing as direct summands of the Springer sheaf, that is the pushforward of the constant sheaf by the Grothendieck-Springer resolution (which is $\pi_{\underline{\dd}}$ for $\underline{\dd}=(1,\hdots,1)\in\N^d$). Therefore, if $\mathscr{F}$ is a $\GL_d$-equivariant perverse sheaf on $\mathfrak{gl}_d$ with singular support in $\Lambda_d^*$, then its Fourier transform is the intersection cohomology of a nilpotent orbit, which means that $\mathscr{F}$ is a direct summand of the Springer sheaf.

\subsubsection{The case $g$-loops quivers}
\label{negativequivers}
In this section, we assume that $Q$ is a $g$-loops quivers with $g\geq 2$.

\begin{theorem}
\label{secondmain}
 Let $Q$ be a $g$-loop quiver. Let $\mathscr{F}$ be an irreducible perverse sheaf on $E_{Q,\dd}$ such that $SS(\mathscr{F})\subset \Lambda_{\dd}^{\nil,1}$. Then, $\mathscr{F}\in\mathcal{P}_{\dd}^{\nil,1}$.
\end{theorem}
\begin{remark}
 Recall that since $Q$ has only one vertex, $\Lambda_{\dd}^{\nil,1}=\Lambda_{\dd}^{\nil}$ and $\mathcal{P}_{\dd}^{\nil,1}=\mathcal{P}_{\dd}^{\nil}$.
\end{remark}

\begin{remark}
\label{remarkequiva}
 To prove this theorem, we surprisingly do not need to assume that $\mathscr{F}$ is $G_{\dd}$-equivariant. It happens to be a consequence of the property on the singular support.
\end{remark}

To prove this theorem, we need several lemmas.

\begin{lemma}
\label{codimtwo}
 Let $Q=S_g$ be the $g$-loop quiver. Let $\underline{\dd}$ and $\underline{\dd}'$ be two flag-types of dimension $\dd$ such that $E_{\underline{\dd}'}^{\nil}\subsetneq E_{\underline{\dd}}^{\nil}$. Then the codimension of $E_{\underline{\dd}'}^{\nil}$ in $E_{\underline{\dd}}^{\nil}$ is at least two.
\end{lemma}
\begin{proof}
 Since $\pi_{\underline{\dd}}^{\nil}$ is small (Proposition \ref{smallnessmor}), it suffices to show that for $\underline{\dd}$ and $\underline{\dd}'$ as in the lemma, $\dim \tilde{\mathcal{F}}_{\underline{\dd}}^{\nil}-\dim\tilde{\mathcal{F}}_{\underline{\dd}'}^{\nil}\geq 2$. However, in this case, \eqref{eq4} reads
 \[
  \dim\tilde{\mathcal{F}}_{\underline{\dd}}^{\nil}=(g+1)\sum_{1\leq t<p\leq r}(\dd_p)_i(\dd_t)_j
 \]
and is therefore a multiple of $g+1\geq 3$. Therefore the codimension is at least $3$.
\end{proof}

We let $Z''=\bigcup_{\substack{\underline{\dd}'\in\mathscr{C}_{\dd}^{1}\\E_{\underline{\dd}'}\subsetneq E_{\underline{\dd}}}}E_{\underline{\dd}'}^{\nil}$. By the previous lemma, it is of codimension at least two (in fact three) in $E_{\underline{\dd}}^{\nil}$.

\begin{lemma}\label{smallmapcodim}
 Let $\pi:X\rightarrow Y$ be a small map. Assume that $Y$ is irreducible. If $Z\subset Y$ is a closed subvariety of codimension at least two, then $\pi^{-1}(Z)\subset X$ is also a closed subvariety of codimension at least two.
\end{lemma}
\begin{proof}
 Since $\pi$ is small, one can find a stratification by locally closed subvarieties $Y=\bigsqcup_{S\in\mathcal{S}}S$ such that for any $S\in \mathcal{S}$, the restriction $\pi_S:\pi^{-1}(S)\rightarrow S$ is a topological fibration and for any $x\in S$,
 \[
  \dim S+2\dim \pi^{-1}(x)<\dim X
 \]
 unless $S$ is the dense stratum, for which we have the equality of the dimensions. We have $Z=\bigsqcup_{S\in\mathcal{S}}Z\cap S$ and $\pi^{-1}(Z)=\bigsqcup_{S\in\mathcal{S}}\pi^{-1}(Z\cap S)$. If $S$ is the dense stratum, $\dim\pi^{-1}(Z\cap S)=\dim Z\cap S$, so $\overline{\pi^{-1}(Z\cap S)}$ is of codimension at least two in $X$. If $S$ is not the dense stratum, but $\pi:\pi^{-1}(S)\rightarrow S$ has finite fibers, $\dim(\pi^{-1}(Z\cap S))=\dim Z\cap S\leq \dim Z$. Therefore, $\overline{\pi^{-1}(Z\cap S)}$ is of codimension at least two in X. Lastly, if $S$ is a stratum over which the fibers of $\pi$ have dimension at least one, then by the smallness of $\pi$, for any $x\in Z\cap S$,
 \[
 \begin{aligned}
  \dim\pi^{-1}(Z\cap S)&=\dim Z\cap S+\dim \pi^{-1}(x)\\
  &\leq \dim S+\dim \pi^{-1}(x)\\
  &<\dim X-\dim \pi^{-1}(x)\\
  &<\dim X-1.
  \end{aligned}
 \]
This last equality allows us to conclude.

\end{proof}

\begin{lemma}
\label{lemmaunique}
 Let $Z'$ be the set of $x\in E_{\underline{\dd}}^{\nil}$ such that $(\pi_{\underline{\dd}}^{\nil})^{-1}(x)$ contains at least two points and $\overline{Z'}$ its Zariski closure (so that $\pi_{\underline{\dd}}^{\nil}:\tilde{\mathcal{F}}_{\underline{\dd}}^{\nil}\setminus(\pi_{\underline{\dd}}^{\nil})^{-1}(\overline{Z'})\rightarrow E_{\underline{\dd}}^{\nil}\setminus \overline{Z'}$ is an isomorphism). Then $(\pi_{\underline{\dd}}^{\nil})^{-1}(\overline{Z'})$ is of codimension at least two in $\tilde{\mathcal{F}}_{\underline{\dd}}^{\nil}$.
\end{lemma}
\begin{proof}
 Consider the diagram
 \[
  \xymatrix{
  \tilde{\mathcal{F}}_{\underline{\dd}}^{\nil}\times_{E_{\dd}}\tilde{\mathcal{F}}_{\underline{\dd}}^{\nil}
  \ar[r]^(.65){pr_1}\ar[rd]_p&\tilde{\mathcal{F}}_{\underline{\dd}}^{\nil}\ar[d]^{\pi_{\underline{\dd}}^{\nil}}\\
  &E_{\underline{\dd}}^{\nil}
  }
 \]
of proper morphisms.
Then, $Z'=p(\tilde{\mathcal{F}}_{\underline{\dd}}^{\nil}\times_{E_{\dd}}\tilde{\mathcal{F}}_{\underline{\dd}}^{\nil}\setminus \tilde{\mathcal{F}}_{\underline{\dd}}^{\nil})$ (recall that $\tilde{\mathcal{F}}_{\underline{\dd}}^{\nil}$ is identified with its diagonal embedding in $\tilde{\mathcal{F}}_{\underline{\dd}}^{\nil}\times_{E_{\dd}}\tilde{\mathcal{F}}_{\underline{\dd}}^{\nil}$) and $(\pi_{\underline{\dd}}^{\nil})^{-1}(Z')=pr_1(\tilde{\mathcal{F}}_{\underline{\dd}}^{\nil}\times_{E_{\dd}}\tilde{\mathcal{F}}_{\underline{\dd}}^{\nil}\setminus \tilde{\mathcal{F}}_{\underline{\dd}}^{\nil})$. Moreover, since a general element of $E_{\underline{\dd}}^{\nil}$ admits a unique filtration of type $\underline{\dd}$ (by Proposition \ref{smallnessmor}), then $\overline{Z'}$ is a proper closed subset of $E_{\underline{\dd}}^{\nil}$ and hence is of codimension at least $1$ in $E_{\underline{\dd}}^{\nil}$. Now, $\overline{Z'}=Z'\sqcup(\overline{Z'}\setminus Z')$ and $\overline{Z'}\setminus Z'$ is of codimension at least two in $E_{\underline{\dd}}^{\nil}$. By Lemma \ref{smallmapcodim}, $(\pi_{\underline{\dd}}^{\nil})^{-1}(\overline{Z'}\setminus Z')$ is of codimension at least two in $\tilde{\mathcal{F}}_{\underline{\dd}}^{\nil}$. It remains therefore to show that $(\pi_{\underline{\dd}}^{\nil})^{-1}(Z')=pr_1(\tilde{\mathcal{F}}_{\underline{\dd}}^{\nil}\times_{E_{\dd}}\tilde{\mathcal{F}}_{\underline{\dd}}^{\nil}\setminus \tilde{\mathcal{F}}_{\underline{\dd}}^{\nil})$ is of codimension at least two in $\tilde{\mathcal{F}}_{\underline{\dd}}^{\nil}$. We write
\[
 \tilde{\mathcal{F}}_{\underline{\dd}}^{\nil}\times_{E_{\dd}}\tilde{\mathcal{F}}_{\underline{\dd}}^{\nil}\setminus \tilde{\mathcal{F}}_{\underline{\dd}}^{\nil}=\bigsqcup_{z\in\Theta(\underline{\dd})\setminus\{z_{\underline{\dd}}\}}(\tilde{\mathcal{F}}_{\underline{\dd}}^{\nil}\times_{E_{\dd}}\tilde{\mathcal{F}}_{\underline{\dd}}^{\nil})_z
\]
If $g\geq 3$, by the second formula of Corollary \ref{formulasmall} and the same argument as in the proof of Proposition \ref{smallnessmor}, we have that $\dim\tilde{\mathcal{F}}_{\underline{\dd}}^{\nil}-(\tilde{\mathcal{F}}_{\underline{\dd}}^{\nil}\times_{E_{\dd}}\tilde{\mathcal{F}}_{\underline{\dd}}^{\nil})_z$ is a non-zero multiple of $g-1$ and hence is $\geq 2$. Consequently, $\dim pr_1((\tilde{\mathcal{F}}_{\underline{\dd}}^{\nil}\times_{E_{\dd}}\tilde{\mathcal{F}}_{\underline{\dd}}^{\nil})_z)\leq \dim \tilde{\mathcal{F}}_{\underline{\dd}}^{\nil}-2$.

If $g=2$, we have to be more careful. If 
\[
 \dim\tilde{\mathcal{F}}_{\underline{\dd}}^{\nil}-(\tilde{\mathcal{F}}_{\underline{\dd}}^{\nil}\times_{E_{\dd}}\tilde{\mathcal{F}}_{\underline{\dd}}^{\nil})_z\geq 2
\]
then the argument above applies. However, it is possible to have
\[
 \dim\tilde{\mathcal{F}}_{\underline{\dd}}^{\nil}-(\tilde{\mathcal{F}}_{\underline{\dd}}^{\nil}\times_{E_{\dd}}\tilde{\mathcal{F}}_{\underline{\dd}}^{\nil})_z=1
\]
and by examining the formula
\[
  \dim\tilde{\mathcal{F}}_{\underline{\dd}}^{\nil}-(\tilde{\mathcal{F}}_{\underline{\dd}}^{\nil}\times_{E_{\dd}}\tilde{\mathcal{F}}_{\underline{\dd}}^{\nil})_z=(g-1)\sum_{\substack{1\leq t<p\leq r\\1\leq q\leq s\leq r}}z^{pq}z^{ts}
\]
we see that this happens if and only if there exists $1\leq i\leq r-1$ such that
\[
 z^{pq}=\left\{
 \begin{aligned}
  &1\text{ if $p=i$, $q=i+1$}\\
  &1 \text{ if $p=i+1$, $q=i$}\\
  &\dd_p\text{ if $p=q$}\\
  &0\text{ else}
 \end{aligned}
\right.
 \]
(In particular, this imposes $\dd_i=\dd_{i+1}=1$). If now $x\in p((\tilde{\mathcal{F}}_{\underline{\dd}}^{\nil}\times_{E_{\dd}}\tilde{\mathcal{F}}_{\underline{\dd}}^{\nil})_z)$ for such a $z$, there exists two flags $\underline{F}$ and $\underline{F'}$ of type $\underline{\dd}$ and in relative position $z$ such that for any $1\leq j\leq r$, $xF_j\subset F_{j-1}$ and $xF'_j\subset F'_{j-1}$. Using the particular form of $z$ described above, we have $F_{i+2}=F_{i+1}+F'_{i+1}$ and therefore $xF_j\subset F_{j-1}$ for $1\leq j\leq i+1$, $xF_{i+2}\subset F_i$ and $xF_{j}\subset F_{j-1}$ for $i+3\leq j\leq r$. We consider the new flag-type $\underline{\dd}'=(\dd'_1,\hdots,\dd'_{r-1})$ such that $\dd'_j=\dd_j$ if $1\leq j\leq i$, $\dd'_{i+1}=2$, $\dd'_j=\dd_{j+1}$ if $i+2\leq j\leq r-1$. Then, $E_{\underline{\dd}'}^{\nil}\subset E_{\underline{\dd}}^{\nil}$ and $p((\tilde{\mathcal{F}}_{\underline{\dd}}^{\nil}\times_{E_{\dd}}\tilde{\mathcal{F}}_{\underline{\dd}}^{\nil})_z)\subset E_{\underline{\dd}'}^{\nil}$. By Lemma \ref{codimtwo} and Lemma \ref{smallmapcodim}, we are done.
\end{proof}

We let now $Z=\overline{Z'}\cup Z''$.

\begin{cor}\label{coriso}
 The restriction $\pi_{\underline{\dd}}^{\nil}:(\pi_{\underline{\dd}}^{\nil})^{-1}(E_{\underline{\dd}}^{\nil}\setminus Z)\rightarrow E_{\underline{\dd}}^{\nil}\setminus Z$ is an isomorphism and $E_{\underline{\dd}}\setminus Z$ is open is $E_{\underline{\dd}}^{\nil}$.
\end{cor}
\begin{proof}
It is clear form what preceeds.
\end{proof}

\begin{lemma}\label{scfv}
 For any $\underline{\dd}\in\mathscr{C}_{\dd}$, the partial flag variety $\mathcal{F}_{\underline{\dd}}$ is simply-connected, and hence $\tilde{\mathcal{F}}_{\underline{\dd}}^{\nil}$ is also simply-connected. Similarly, $\tilde{\mathcal{F}}_{\underline{\dd}}$ is simply-connected.
\end{lemma}
\begin{proof}
 The partial flag variety $\mathcal{F}_{\dd}$ is a partial flag variety of a reductive algebraic group and hence admits a cell decomposition where the cells are affine spaces. Therefore, it is simply connected. Now, the second projection $\tilde{\mathcal{F}}_{\underline{\dd}}^{\nil}\rightarrow\mathcal{F}_{\underline{\dd}}, (x,\underline{F})\mapsto \underline{F}$ is a fiber bundle. Hence, $\tilde{\mathcal{F}}_{\underline{\dd}}^{\nil}$ is also simply connected.
\end{proof}

\begin{proof}[Proof of Theorem \ref{secondmain}]
 Let $\mathscr{F}$ be an irreducible perverse sheaf on $E_{Q,\dd}$ such that $SS(\mathscr{F})\subset \Lambda_{Q,\dd}^{\nil,1}$. Let $p_{\dd}:E_{\overline{Q},\dd}\rightarrow E_{Q,\dd}$ be the cotangent bundle map. Since $p_{\dd}(SS(\mathscr{F}))=\supp(\mathscr{F})$, by Proposition \ref{truenil}, there exists a flag-type $\underline{\dd}\in\mathscr{C}^{1}_{\dd}$ such that $\supp(\mathscr{F})=E_{\underline{\dd}}^{\nil}$ and moreover, any other irreducible component of $SS(\mathscr{F})$ is of the form $T^*_{E_{Q,\underline{\dd}'}^{\nil}}E_{Q,\dd}$ for some flag-type $\underline{\dd}'\in\mathscr{C}_{\dd}^{\nil}$ such that $E_{Q,\underline{\dd}'}^{\nil}\subset E_{Q,\underline{\dd}}$. Let $Z$ be as before Corollary \ref{coriso}. By Corollary \ref{coriso}, the open subset $E_{\underline{\dd}}^{\nil}\setminus Z$ is smooth (being isomorphic to an open subset of $\tilde{\mathcal{F}}_{\underline{\dd}}^{\nil}$, which is smooth) and the restriction $\mathscr{F}_U$ of $\mathscr{F}$ to $U:=E_{\underline{\dd}}^{\nil}\setminus Z$ verifies $SS(\mathscr{F}_{U})=T^*_{U}U$. By Lemma \ref{sysloczerosection}, $\mathscr{F}_U=\mathscr{L}[s]$ for some local system $\mathscr{L}$ on $U$ and $s=\dim U$. Therefore, $(\pi_{\underline{\dd}}^{\nil})^*\mathscr{L}$ is a local system on $\tilde{\mathcal{F}}_{\underline{\dd}}^{\nil}\setminus(\pi_{\underline{\dd}}^{\nil})^{-1}(Z)$. By Lemmas \ref{codimtwo}, \ref{smallmapcodim} and \ref{lemmaunique}, $(\pi_{\underline{\dd}}^{\nil})^{-1}(Z)$ is of codimension at least two in $\tilde{\mathcal{F}}_{\underline{\dd}}^{\nil}$. Since the latter is smooth, $(\pi_{\underline{\dd}}^{\nil})^*\mathscr{L}$ can be extended to a local system on $\tilde{\mathcal{F}}_{\underline{\dd}}^{\nil}$. Since $\tilde{\mathcal{F}}_{\underline{\dd}}^{\nil}$ is simply-connected (Lemma \ref{scfv}), this extension is the trivial local system. Hence, by Lemma \ref{coriso}, $\mathscr{L}$ is the trivial local system on $U$. Therefore, $\mathscr{F}=\ICC(E_{\underline{\dd}}^{\nil})$. By Proposition \ref{explicitdesc}, $\mathscr{F}$ is indeed in $\mathcal{P}_{\dd}^{\nil,1}$.
\end{proof}

\appendix
\section{Local systems on the complement of a normal crossing divisor}
\begin{lemma}\label{slcomplsncd}
Let $X$ be a smooth irreducible curve and $d\geq 1$ an integer. Let $D\subset X$ be a finite set. Then $U=(X\setminus D)^d$ is the complement of a simple normal crossing divisor in $X^d$. 
\end{lemma}
\begin{proof}
 Let $x=(x_1,\hdots,x_d)\in X^d$. By symmetry of the question, we can assume that there exists $r\geq 1$ such that $x_1,\hdots,x_r\in D$ and $x_{r+1},\hdots,x_d\in X\setminus D$. If $y_i$ ($1\leq i\leq d$) is a local coordinate of $X$ around $x_i$, such that $x$ corresponds to $y_1=\hdots=y_d=0$, a local equation for $X\setminus U$ in a neighbourhood of $x$ is $\prod_{i=1}^ry_i$.
\end{proof}

\begin{lemma}\label{lemext}
 Let $D\subset \PP^1(\C)$ be a finite subset and $d\geq 1$. Let $\mathscr{L}$ be a local system on $U=(\PP^1(\C)\setminus D)^d\setminus \Delta$. Then $\mathscr{L}$ extends to a local system on $Y=\PP^1(\C)^d\setminus \Delta$ if and only if for any $x\in D$, there exists an analytic neighbourhood $V$ of $(x,\hdots,x)\in\PP^1(\C)^d$ such that $\mathscr{L}_{V\cap U}$ extends to $V\cap Y$. 
\end{lemma}
\begin{proof}
 The direct implication is trivial. For the other implication, for any $x\in D$, let $V_x$ a neighbourhood of $(x,\hdots,x)\in\PP^1(\C)^d$ such that $\mathscr{L}_{V_x\cap U}$ extends to $V_x\cap Y$. Then any branch of the simple normal crossing divisor $Y\setminus U$ of $Y$ intersects at least one of the $V_x$. Therefore, the monodromies around all branches of $Y\setminus U$ are trivial and $\mathscr{L}$ extends to $Y$.
\end{proof}
We recall that $D(0,1)\subset \C$ denotes the open unit disk and $D^*(0,1)$ the punctured unit disk. Let $\pi : D^*(0,1)^d\setminus\Delta\rightarrow S^dD^*(0,1)\setminus\Delta$. It is a $\mathfrak{S}_d$-covering.
\begin{lemma}\label{lemextd}
 Let $\mathscr{L}$ be a local system on $S^dD^*(0,1)\setminus\Delta$. Then $\mathscr{L}$ can be extended to $S^dD(0,1)\setminus\Delta$ if and only if $\pi^*\mathscr{L}$ can be extended to $D(0,1)^d\setminus\Delta$.
\end{lemma}
\begin{proof}
 The direct implication is obvious. Let $U=S^dD^*(0,1)\setminus\Delta$, $X=S^dD(0,1)\setminus\Delta$, $\tilde{U}=D^*(0,1)^d\setminus\Delta$, $\tilde{X}=D(0,1)^d\setminus\Delta$. Let $\mathscr{L}'$ the extension of $\pi^*\mathscr{L}$ to $\tilde{X}$. Since $\tilde{\pi}:\tilde{X}\rightarrow X$ is a $\mathfrak{S}_d$-covering, $\pi_*\mathscr{L}'$ is a local system on $X$. Moreover, $\pi_*\mathscr{L}'_{U}=\pi_{*}\pi^*\mathscr{L}$ and therefore, $\mathscr{L}$ is a direct summand of $\pi_*\mathscr{L}'_{U}$. The corresponding direct summand of $\pi_*\mathscr{L}'$ gives the extension of $\mathscr{L}$ to $X$.
\end{proof}

 Let $D\subset \PP^1(\C)$ be a nonempty finite subset. Let $\mathscr{L}$ be a local system on $S^d(\PP^1(\C)\setminus D)\setminus\Delta$. We let $U=S^d(\PP^1(\C)\setminus D)\setminus\Delta$, $X=S^d\PP^1(\C)\setminus\Delta$, $\tilde{U}=(\PP^1(\C)\setminus D)^d\setminus \Delta$ and $\tilde{X}=\PP^1(\C)^d\setminus \Delta$. We have the diagram
 \[
  \xymatrix{
  \tilde{U}\ar[d]_{\pi}\ar[r]&\tilde{X}\ar[d]^{\pi'}\\
  U\ar[r]&X
  }
 \]
\begin{lemma}\label{triviallocal}
The local system $\pi^*\mathscr{L}$ is trivial if and only if for any $x\in D$, there exists a neighbourhood $V\subset \PP^1(\C)$ of $x$ such that, with $V_d=S^d(V\setminus \{x\})\setminus \Delta$ and $\tilde{V}_d=(V\setminus \{x\})^d\setminus\Delta$, $\pi_V : \tilde{V}_d\rightarrow V_d$,
\[
 (\pi_V)^*\mathscr{L}_{V_d}
\]
is trivial on $\tilde{V}_d$.
\end{lemma}

\begin{proof}
 By Lemma \ref{lemext} and Lemma \ref{lemextd}, $\mathscr{L}$ extends to $S^d\PP^1(\C)\setminus \Delta$. Let $y\in \PP^1(C)\setminus D$. Then, $\PP^1(\C)\setminus\{y\}\simeq \A^1(\C)$. The inclusion $\A^d(\C)\setminus\Delta\subset S^d\PP^1(\C)\setminus \Delta$ is surjective on the fundamental groups, so that it suffices to check that the restriction of $\pi^*\mathscr{L}$ to $\A^d(\C)\setminus\Delta$ is trivial. Now, for $x\in D$, let $D(x,r)\subset V$ be a small neighbourhood of $x$. Then the inclusion $D(x,r)^d\setminus\Delta\subset \A^d(\C)\setminus\Delta$ induces an isomorphism at the level of fundamental groups. Therefore, it suffices to prove that the restriction of $\pi^*\mathscr{L}$ to $D(x,r)^d\setminus \Delta$ is trivial. But this is clearly implied by the hypotheses of the lemma.
\end{proof}

\section{Equivariant perverse sheaves and local systems}
\subsection{Equivariant perverse sheaves}
We will quite often face the situation described in the following lemma.
\begin{lemma}\label{equiperverse}
Let $G$ be a connected algebraic group and $H\subset G$ a normal closed subgroup. Let $X$ be a $G$-variety on which $H$ acts trivially (\emph{i.e.} a $G/H$-variety). Then $\Perv_{G}(X)\simeq \Perv_{G/H}(X)$.
\end{lemma}
\begin{proof}
 By hypothesis, $G$ and $G/H$ are connected. We have the action map
 \[
  a:G\times X\rightarrow X
 \]
and its factorization
\[
 a':G/H\times X\rightarrow X
\]
together with the projections
\[
 p:G\times X\rightarrow X
\]
and
\[
 p':G/H\times X\rightarrow X.
\]
Let $\pi : G\rightarrow G/H$ be the projection. Then $a=a'\circ (\pi\times \id_X)$ and $p=p'\circ(\pi\times \id_X)$.
The forgetful functor $\Perv_G(X)\rightarrow \Perv(X)$ identifies $\Perv_G(X)$ with the full subcategory of $\Perv(X)$ of perverse sheaves $\mathscr{F}$ such that $p^*\mathscr{F}$ and $a^*\mathscr{F}$ are isomorphic. Similarly, $\Perv_{G/H}(X)$ is identified with the subcategory of $\Perv(X)$ of perverse sheaves $\mathscr{F}$ such that $p'^*\mathscr{F}$ and $a'^*\mathscr{F}$ are isomorphic. Since $\pi\times\id_X$ is smooth, given $\mathscr{F}\in \Perv(X)$, $p'^*\mathscr{F}$ and $a'^*\mathscr{F}$ are isomorphic if and only if $(\pi\times\id_X)^*p'^*\mathscr{F}=p^*\mathscr{F}$ and $(\pi\times\id_X)^*a'^*\mathscr{F}=a^*\mathscr{F}$ are isomorphic. Hence, $\Perv_{G}(X)$ and $\Perv_{G/H}(X)$ are both identified with the same full subcategory of $\Perv(X)$ and hence are equivalent.
\end{proof}

Let $H\subset G$ be a closed subgroup of an algebraic group $G$ and $X$ a $H$-variety. Let
\[
 \begin{matrix}
  i&:&X&\rightarrow &X\times^HG\\
  &&x&\mapsto&(x,e).
 \end{matrix}
\]
The following lemma is used many times in this paper without mention.
\begin{lemma}[Induction equivalence, {\cite{MR1299527}}]
 The functor
 \[
  i^*[\dim H-\dim G]:D^b_G(X\times^HG)\rightarrow D^b_H(X)
 \]
is a perverse equivalence of categories.
\end{lemma}

\subsection{Local systems on the regular semisimple locus of a reductive Lie algebra}
We let $G=\GL_d$, $\mathfrak{g}=\mathfrak{gl}_d$, $T\subset G$ is the maximal torus of diagonal matrices and $\mathfrak{t}$ its Lie algebra. We let $W=\mathfrak{S}_d$ be the Weyl group. We let $\mathfrak{g}^{rss}\subset \mathfrak{g}$ be the open subvariety of regular semisimple elements, $f':\tilde{\mathfrak{g}}^{rss}\rightarrow\mathfrak{g}^{rss}$ the $W$-covering induced by the Grothendieck-Springer simultaneous resolution, and $g$ the semisimplification map. The map $g'$ lifts $g$ in the sense that the following square is cartesian.
 \[
  \xymatrix{
  \tilde{\mathfrak{g}}^{rss}\ar[r]^{f'}\ar[d]_{g'}&\mathfrak{g}^{rss}\ar[d]^g\\
  \mathfrak{t}^{rss}\ar[r]^{f}&\mathfrak{t}^{rss}/W
  }.
 \]
\begin{lemma}\label{locsysg}
 A $G$-equivariant local system $\mathscr{L}$ on $\mathfrak{g}^{rss}$ is the pull-back of a local system $\mathscr{L}'$ on $\mathfrak{t}^{rss}/W$. The local system $\mathscr{L}'$ is unique up to isomorphism. In other words, $g$ is an equivariant $\pi_1$-equivalence. Moreover, $(f')^*\mathscr{L}$ is trivial if and only if $f^*\mathscr{L}'$ is trivial.
\end{lemma}
\begin{proof}
 Let $\mathscr{L}$ be a $G$-equivariant local system on $\mathfrak{g}^{rss}$. We can assume that $\mathscr{L}$ is indecomposable. Since $f'$ is a $W$ covering, $\mathscr{L}$ is a direct summand of $f'_*(f')^*\mathscr{L}$. Assume that $(f')^{*}\mathscr{L}=(g')^*\mathscr{L}''$ for some local system $\mathscr{L}''$ on $\mathfrak{t}^{rss}$. Then by smooth base-change, $f'_*(g')^*\mathscr{L}''\simeq g^*f_*\mathscr{L}''$. Therefore, $\mathscr{L}$ is a direct summand of $g^*f_*\mathscr{L}''$. Since $\mathscr{L}$ is indecomposable, there exists an indecomposable summand $\mathscr{L}'$ of $f_*\mathscr{L}''$ such that $\mathscr{L}$ is a direct summand of $g^*\mathscr{L}'$. Since $g$ is smooth, $g^*\mathscr{L}'$ is indecomposable. Therefore, $\mathscr{L}=g^*\mathscr{L}'$. Therefore, it suffices to prove the existence of $\mathscr{L}''$.

But $\tilde{\mathfrak{g}}^{rss}\simeq \mathfrak{t}^{rss}\times^{T}G$, the action of $T$ on $\mathfrak{t}^{rss}$ being trivial. Therefore, a $G$-invariant local system on $\tilde{\mathfrak{g}}^{rss}$ is the same thing as a $T$-equivariant local system on $\mathfrak{t}^{rss}$, that is (since $T$ is connected) a local system on $\mathfrak{t}^{rss}$. Let
\[
 \mathfrak{t}^{rss}\xrightarrow{i} \tilde{\mathfrak{g}}^{rss}\xrightarrow{g'} \mathfrak{t}^{rss}
\]
where the first row is the closed immersion and the second is $g'$. Let $\mathscr{L}$ be a local system on $\tilde{\mathfrak{g}}^{rss}$. Let $\mathscr{L}''=i^*\mathscr{L}$. Then, $\mathscr{L}=(g')^*\mathscr{L}''$ since $i^*\mathscr{L}=i^*(g')^*\mathscr{L}''$, because $g'\circ i=\id$. This ends the proof.
\end{proof}
\begin{remark}
 This lemma can also be deduced from the fact that the algebraic stacks $\mathfrak{g}^{rss}/G$ and $\mathfrak{t}^{rss}/W$ are isomorphic for any reductive group $G$.
\end{remark}

\subsection{Equivariant local systems on the regular semisimple strata of affine quivers}\label{locsysrssaff}
Let $Q$ be an affine acyclic quiver (resp. a cyclic quiver). Let $\Xi(N,\mu)\subset E_{\dd}$ be a regular semisimple stratum, that is $N$ is regular non-homogeneous (resp. nilpotent), $\mu$ is regular semisimple and $\dd=\dim N+\delta\dim\mu$. Let $d=\dim\mu$. Recall the map $\chi_{N,\mu}:\Xi(N,\mu)\rightarrow S^d\PP_1^{\hom}\setminus\Delta$ from Section \ref{quotientreg} (resp. $\chi_{N,\mu}:\Xi(N,\mu)\rightarrow S^d(\C^*)\setminus\Delta$ from Section \ref{opensubsets}). Recall also the $\mathfrak{S}_d$-covering $\pi_d:(\PP_1^{\hom})^d\setminus\Delta\rightarrow S^d\PP_1^{\hom}\setminus\Delta$ (resp. $\pi_d:(\C^*)^d\setminus\Delta\rightarrow S^d(\C^*)\setminus\Delta$). We have the following analog of Lemma \ref{locsysg}. 

\begin{lemma}\label{locsysreglocus}
 Any $G_{\dd}$-equivariant local system $\mathscr{L}$ on $\Xi(N,\mu)$ is the pull-back by $\chi_{N,\mu}$ of a local system $\mathscr{L}'$ on $S^d(\PP_1^{\hom})\setminus\Delta$ (resp. on $S^d(\C^*)\setminus\Delta$). Moreover, $\mathscr{L}'$ is determined up to isomorphism by $\mathscr{L}$ and $\mathscr{L}$ is a Lusztig local system (resp. an extended Lusztig local system) if and only of $\pi_d^*\mathscr{L}'$ is the trivial local system.
\end{lemma}
\begin{proof}
 The proof is completely similar to that of Lemma \ref{locsysg} using the following facts (stated for acyclic affine quivers; the analogous facts for cyclic quivers are also true).
 
 We have a cartesian diagram
 \[
  \xymatrix{
  \tilde{\Xi}(N,\mu)\ar[r]^{\pi_{N,\mu}}\ar[d]_{\tilde{\chi}_{N,\mu}}&\Xi(N,\mu)\ar[d]^{\chi_{N,\mu}}\\
  (\PP_1^{\hom})^d\setminus\Delta\ar[r]^{\pi_{d}}&S^d\PP_1^{\hom}\setminus\Delta.
  }
 \]
 where $\pi_{N,\mu}$ and $\pi_d$ are $\mathfrak{S}_d$-coverings. Moreover, $\tilde{\chi}_{N,\mu}$ (and thus $\chi_{N,\mu}$) is smooth. This can be proved as follows. By the definition of $\chi_{N,\mu}$ in Section \ref{quotientreg}, it suffices to prove that $\tilde{\chi}_{\mu}:\tilde{\Xi}(\mu)\rightarrow (\PP_1^{\hom})^d\setminus\Delta$ is smooth. In dimension $\delta$, the morphism $E_{\delta}^{\reghom}\rightarrow \PP_1^{\hom}$ is a $G_{\delta}/\C^*$-principal bundle. Therefore it is smooth. Consider the following cartesian diagram (which defines $Z$):
 \[
  \xymatrix{
  Z\ar[r]\ar[d]_p&(E_{\delta}^{\reghom})^d\ar[d]\\
  (\PP_1^{\reghom})^d\setminus\Delta\ar[r]&(\PP_1^{\reghom})^d.
  }
 \]
By base-change, the vertical left-most map $p$ is smooth. Then, $\tilde{\Xi}(\mu)\simeq Z\times^{(G_{\delta})^d}G_{d\delta}$. In the commutative triangle
\[
 \xymatrix{
 Z\times G_{\delta}\ar[r]\ar[rd]_{p\circ \pr_1}&\tilde{\Xi}(\mu)\ar[d]^{\tilde{\chi}_{N,\mu}}\\
 &(\PP_1^{\reghom})^d
 },
\]
the horizontal arrow is a $(G_{\delta})^d$-principal bundle, hence is smooth and $p\circ pr_1$ is also smooth. As a consequence, $\tilde{\chi}_{N,\mu}$ is smooth.

\end{proof}

\section{Singular support of constructible complexes}\label{singularsupports}

Let $f : X\rightarrow Y$ be a morphism between smooth complex algebraic varieties. Then, we have the classical correspondence between cotangent bundles:
\[
 \xymatrix{
 &X\times_YT^*Y\ar[ld]_{pr_2}\ar[rd]^{(df)^*}&\\
 T^*Y&&T^*X
 }
\]
There are properties on the morphism $f$ that we recall here ensuring that the singular support satisfies natural functorialities with respect to these morphisms. We refer to the original references for proofs.

\subsection{Singular support of the pull-back by a smooth morphism}
For a proof of the following proposition, we refer to \cite{MR794557}.
\begin{proposition}[{\cite[Proposition 4.1.2]{MR794557}}]\label{pbsmooth}
 Suppose that $f : X\rightarrow Y$ is smooth. Let $\mathscr{F}$ be a constructible sheaf on $Y$. Then
 \[
  SS(f^*\mathscr{F})=(df)^*(pr_2^{-1}(SS(\mathscr{F}))).
 \]
\end{proposition}

\begin{cor}\label{corpbsmooth}
 Let $f : X\rightarrow Y$ be a smooth morphism and $\mathscr{F}\in D^b(Y)$ a constructible complex. Write
 \[
  SS(\mathscr{F})=\bigcup_{S\in \mathcal{S}}\overline{T^*_SY}
 \]
 for some set $\mathcal{S}$ of locally closed subvariety of $Y$. Then
 \[
  SS(f^*\mathscr{F})=\bigcup_{S\in\mathcal{S}}\overline{T^*_{f^{-1}(S)}X}.
 \]
\end{cor}

\begin{proof}
 In this case, $(df)^*$ is a closed immersion. Moreover, $pr_2^{-1}(T^*_SY)=X\times_YT^*_SY$ is isomorphic to $T^*_{f^{-1}(S)}X$ via $(df)^*$.
\end{proof}

\subsection{Singular support and the pushforward by a proper morphism}
The following proposition is taken from \cite[Proposition 5.4.4]{MR1074006}.
\begin{proposition}\label{pushforward}
 Let $X\rightarrow Y$ be a proper morphism of manifolds, $\mathscr{F}\in D^b(Y)$. Then
 \[
  SS(f_*\mathscr{F})\subset pr_2((df)^{*-1}(SS(\mathscr{F}))) 
 \]
and this inclusion is an equality if $f$ is a closed immersion.
\end{proposition}

\subsection{Induction of singular supports}\label{lemmass}

Let $H\subset G$ be connected algebraic groups. Let $X$ be a $H$-variety. Let $i$ be the closed immersion
\[
 \begin{matrix}
  i&:&X&\rightarrow &X\times^HG\\
  &&x&\mapsto&(x,e)
 \end{matrix}.
\]
Then, we have a triangulated equivalence of categories preserving the categories of perverse sheaves (\cite[2.6]{MR1299527}):
\[
 i^{0}:=i^*[\dim H-\dim G] : D^b_G(X\times^HG)\rightarrow D^b_H(X).
\]
The inverse equivalence is described by the induction $\gamma_H^G$ as follows (\cite[1.4]{MR948107}). Consider the diagram:
\begin{equation}\label{inddiag}
 \xymatrix{
 &X\times G\ar[ld]_{pr_1}\ar[rd]^{\pi}\\
 X&&X\times^HG
 }
\end{equation}
Then for $\mathscr{F}$ a constructible $H$-equivariant complex on $X$, $pr_1^{*}\mathscr{F}$ is $H$-equivariant on $X\times G$, so there is a unique constructible complex (up to isomorphism) on $X\times^HG$, $\mathscr{G}$, such that $\pi^*\mathscr{G}\simeq pr_1^{*}\mathscr{F}$. We let $\gamma_H^G=\mathscr{G}[\dim G-\dim H]$.

\begin{lemma}\label{indcv}
 Let $X$ be an $H$-variety and $H\rightarrow G$ an injective group homomorphism. Let $\mathcal{F}\in D^b_H(X)$ and $\mathcal{G}\in D^b_G(X\times^HG)$ be two perverse sheaves corresponding to each other via the equivalence $D^b_H(X)\simeq D^b_G(X\times^HG)$. Then
 \[
  SS(\mathcal{F})=\bigsqcup_{S\in\mathcal{S}}T^*_SX
 \]
 is and only if
 \[
  SS(\mathcal{G})=\bigsqcup_{S\in\mathcal{S}}T^*_{S\times^HG}(X\times^HG).
 \]
 \end{lemma}
\begin{proof}
  In the induction diagram $\eqref{inddiag}$, both $pr_1$ and $\pi$ are smooth. We only prove the direct implication, the proof of the converse being similar. Write
  \[
   SS(\mathscr{G})=\bigcup_{S'\in \mathcal{S}'}\overline{T^*_{S'}(X\times^HG)}
  \]
  for some $G$-invariant locally closed strata $S'\in\mathcal{S}'$ of $X\times^HG$. It is easily seen that by $G$-invariance, any $S'\in\mathcal{S}'$ can be written $S'=S'_Y\times^HG$ where $S'_Y\subset Y=i^{-1}(S')$ is $H$-invariant and locally closed in $Y$. Then, by Corollary \ref{corpbsmooth},
  \[
   SS(\pi^*(\mathscr{G}))=\bigcup_{S'\in \mathcal{S}'}\overline{T^*_{S'_Y\times G}(X\times G)}.
  \]
Moreover, $\pi^*\mathscr{G}\simeq pr_1^*\mathscr{F}$ and again by Corollary \ref{corpbsmooth},
\[
 SS(pr_1^*(\mathscr{F}))=\bigcup_{S\in \mathcal{S}}\overline{T^*_{S\times G}(X\times G)}.
\]
This ends the proof.
\end{proof}

\begin{lemma}\label{sysloczerosection}
 Let $\mathscr{F}$ be a perverse sheaf on a smooth irreducible variety $X$ such that $SS(\mathscr{F})=T^*_XX$. Then $\mathscr{F}=\mathscr{L}[\dim X]$ for some local system $\mathscr{L}$ on $X$.
\end{lemma}
\begin{proof}
 This is a particular case of \cite[Theorem 4.3.15 iv)]{MR2050072} in the case where $X$ is trivially stratified.
\end{proof}

\section{Fourier-Sato transform}\label{fouriersato}

\subsection{Monodromic sheaves}
Let $E\rightarrow X$ be a complex fiber bundle over an algebraic variety $X$. We consider the $\C^*$-action of weight $1$ contracting the fibers. We call a sheaf on $E$ \emph{monodromic} if it is locally constant on $\C^*$-orbits of $E$.

\begin{lemma}\label{lemmamono}
 Let $i : Y\rightarrow E$ be a $\C^*$-invariant locally closed subvariety. Let $\mathscr{F}\in D^b_{mon}(Y,\C)$ be a constructible monodromic complex. Then $i_{!*}\mathscr{F}$ is a constructible monodromic complex on $E$.
\end{lemma}

\begin{cor}\label{cormono}
 Let $\mathscr{F}\in D^b(E,\C)$ be an irreducible perverse sheaf whose singular support is the union of conormal bundles to $\C^*$-invariant subvarieties. Then $\mathscr{F}$ is monodromic.
\end{cor}
\begin{proof}
 The support of $\mathscr{F}$, $SS(\mathscr{F})\cap T^*_EE$ is a $\C^*$-invariant closed subvariety of $E$. There exists a smooth $\C^*$-invariant open subset $j : U\rightarrow \supp\mathscr{F}$ such that $SS(j^*\mathscr{F})=T^*_UU$. Then, $j^*\mathscr{F}$ is a local system on $U$. Then, by Lemma \ref{lemmamono}, $\mathscr{F}=j_{!*}(j^*\mathscr{F})$ is monodromic on $E$.
\end{proof}

\subsection{The Fourier-Sato transform}\label{definitionfs}

The definition and basic facts for the Fourier-Sato transform in the complex-analytic setting are given in \cite[2.7]{MR3218317}.

Let $X$ be a complex algebraic variety and $p : E\rightarrow X$ a complex vector bundle. Let $\check{p} : E^*\rightarrow X$ be its dual.

Consider the correspondence
\[
 \xymatrix{&\ar[ld]_{q}Q\ar[rd]^{\check{q}}&\\
 E&&E^*}
\]
where
\[
 Q=\{(x,y)\in E\times_XE^*\mid \Re(\langle x,y\rangle)\leq 0\}\subset E\times_XE^*.
\]

Then the Fourier-Sato transform is the the functor
\[
\begin{matrix}
 &\varPhi_E &:& D^b_{mon}(E,\C)&\rightarrow& D^b_{mon}(E^*,\C)\\
 &&&\mathscr{F}&\mapsto&\check{q}_!q^*\mathscr{F}[\rank E].
 \end{matrix}
\]
It is an equivalence of categories preserving perverse sheaves. Useful compatibilities between Fourier transform and other functors are stated in \cite[2.7 and Appendix A]{MR3218317}. The Fourier-Sato transform exists also in the equivariant setting. If $X$ is a $G$-variety for some algebraic group $G$ and $E\rightarrow X$ is a $G$-equivariant vector bundle, its dual is also an equivariant vector bundle and the Fourier-Sato transform gives a perverse equivalence of categories
\[
 \Phi_E : D^b_{G,mon}(E)\rightarrow D^b_{G,mon}(E^*).
\]

\subsection{Action on the singular support}\label{fssingularsupport}

The main result of this section is that the Fourier-Sato transform preserves the singular support.

Let $p : E\rightarrow X$ be a complex vector bundle and $\check{p} : E^*\rightarrow X$ its dual. We can identify $T^*E$ and $T^*E^*$ as complex fiber bundles (and also symplectic varieties) over $E^*$ as follows (\cite[V.5.5]{MR1074006}). First consider the relative cotangent bundle of $p$, $T^*(E/X)$. It is the cokernel of the monomorphism of vector bundles $E\times_{X}T^*X\rightarrow T^*E$. Then, we have a morphism $T^*(E/X)\rightarrow E^*$ given on fibers over $X$ by
\[
 T^*(E/X)_x\simeq E_x\times (E_x)^*\rightarrow (E_x)^*.
\]
Composing with $T^*E\rightarrow T^*(E/X)$, we obtain the map $p' : T^*E\rightarrow E^*$. Proposition 5.5.1 in \cite{MR1074006} gives the existence of an isomorphism of vector bundles over $E^*$
\[
\xymatrix{
T^*E\ar[rr]^{\Xi}\ar[rd]_{p'}&& T^*E^*\ar[ld]^{\check{p}}\\
&E^*&
 }.
\]
 It is given in local coordinates by $(x,y,\zeta,\xi)\mapsto (x,\xi,\zeta,-y)$, where $x$ is a local coordinate on $X$ and $y$ a local coordinate on the fiber $E_x$. In our context, Theorem 5.5.5 of \emph{loc. cit.} can be formulated as follows.

\begin{theorem}\label{ssfourier}
 Let $\mathscr{F}\in D^b_{mon}(E,\C)$. Then,
 \[
  \Xi(SS(\mathscr{F}))=SS(\varPhi(\mathscr{F})).
 \]

\end{theorem}
\subsection{The Fourier-Sato transform for quivers}
We briefly describe how we will use the Fourier-Sato transform for quivers. Let $Q=(I,\Omega)$ be a quiver. Let $\Omega=\Omega_1\sqcup \Omega_2$ be a partition of the set of arrows. We obtain new quivers $Q_1=(I,\Omega_1)$ and $Q_2=(I,\Omega_2)$. Let $\overline{\Omega}_2$ be the opposite set of arrows (we reverse the direction of arrows in $\Omega_2$) and $\Omega'=\Omega_1\sqcup \overline{\Omega}_2$. Let $Q'=(I,\Omega')$. For $\dd\in\N^I$, we have two vector bundles over $E_{Q_1,\dd}$:
\[
 \xymatrix{
 E_{Q,\dd}\ar[rd]_{\pi}&&E_{Q',\dd}\ar[ld]^{\check{\pi}}\\
 &E_{Q_1,\dd}&
 }
\]
given by the projection on the corresponding direct factor of $E_{Q,\dd}$ (resp. $E_{Q',\dd}$). The trace maps identifies $\check{\pi}$ with the dual of $\pi$. The Fourier-Sato transform then gives a perverse equivalence of categories
\[
 \Phi:D^b_{G_{\dd},mon}(E_{Q,\dd})\rightarrow D^b_{G_{\dd},mon}(E_{Q',\dd}).
\]

\section*{Acknowledgements}
The author warmly thanks Olivier Schiffmann for his constant support, availability and useful comments concerning this work. I also benefited from discussions with Masaki Kashiwara and Anton Mellit. Thanks to Ben Webster for pointing out the reference \cite{MR3651581}, to George Lusztig for the paper \cite{MR1182165} and to Ivan Losev for helping me to understand how to provide an alternative proof of Theorem \ref{mainresult} by combining \cite{bl} and \cite{losev}, together with \cite{MR3651581} (see Remark \ref{discl}).

\bibliographystyle{alpha}
\bibliography{bibliographie}
\end{document}